\documentclass{amsart}
\usepackage{color}
\usepackage{amscd}
\usepackage{graphicx}
\usepackage{hyperref}
\usepackage{amssymb}
\usepackage{amsfonts}
\usepackage{euscript}
\usepackage[all]{xy}

\numberwithin{equation}{section}
\renewcommand{\subsection}[1]{\hspace{-\parindent}\refstepcounter{subsection}{\bf (\arabic{section}\alph{subsection}) #1.}\addcontentsline{toc}{subsection}{\bf #1.}}

\newenvironment{nouppercase}{%
  \renewcommand{\uppercasenonmath}[1]{}}{}

\theoremstyle{plain}
\newtheorem{thm}{Theorem}[section]
\newtheorem{theorem}[thm]{Theorem}

\newtheorem{definition}[thm]{Definition}

\newtheorem{remark}[thm]{Remark}

\newtheorem{proposition}[thm]{Proposition}
\newtheorem{example}[thm]{Example}

\newtheorem{lemma}[thm]{Lemma}
\newtheorem{setup}[thm]{Setup}

\newtheorem{conjecture}[thm]{Conjecture}
\newtheorem{conventions}[thm]{Conventions}

\newtheorem*{claim*}{Claim} 
\newtheorem*{lemma*}{Lemma}
\newtheorem*{theorem*}{Theorem}
\newtheorem*{conjecture*}{Conjecture}

\newcommand{\bC}{{\mathbb C}}

\newcommand{\bF}{{\mathbb F}}

\newcommand{\bK}{{\mathbb K}}

\newcommand{\bP}{{\mathbb P}}

\newcommand{\bR}{{\mathbb R}}

\newcommand{\bZ}{{\mathbb Z}}

\newcommand{\scrA}{\EuScript A}
\newcommand{\scrB}{\EuScript B}

\newcommand{\scrD}{\EuScript D}

\newcommand{\scrF}{\EuScript F}
\newcommand{\scrG}{\EuScript G}

\newcommand{\scrL}{\EuScript L}

\newcommand{\scrP}{\EuScript P}
\newcommand{\scrQ}{\EuScript Q}
\newcommand{\scrR}{\EuScript R}

\newcommand{\scrW}{\EuScript W}

\newcommand{\half}{{\textstyle\frac{1}{2}}}

\newcommand{\iso}{\cong}
\newcommand{\htp}{\simeq}
\newcommand{\smooth}{C^\infty}

\newcommand{\frakg}{\mathfrak{g}}

\renewcommand{\hom}{\mathit{hom}}

\newcommand{\Sym}{\mathit{Sym}}

\title[LEFSCHETZ FIBRATIONS]{\Large\larger\rm Fukaya $A_\infty$-structures associated to\\ Lefschetz fibrations. II}
\author{Paul Seidel}

\begin{document}
\begin{nouppercase}
\maketitle
\end{nouppercase}
\begin{abstract}
Consider the Fukaya category associated to a Lefschetz fibration. It turns out that the Floer cohomology of the monodromy around $\infty$ gives rise to natural transformations from the Serre functor to the identity functor, in that category. We pay particular attention to the implications of that idea for Lefschetz pencils.
\end{abstract}

\section{Introduction}
This paper is part of an investigation of the Floer-theoretic structures arising from Lefschetz fibrations. Its purpose is to introduce a new piece of that puzzle; and also, to set up a wider algebraic framework into which that piece should (conjecturally) fit, centered on a notion of noncommutative pencil.

\subsection{Symplectic geometry}
It is well-known that there is a version of the Fukaya category tailored to a Lefschetz fibration (the idea is originally due to Kontsevich). In \cite{seidel06}, it was pointed out that these categories always come with some added structure: a distinguished natural transformation from the Serre functor to the identity functor (see \cite{seidel08, bourgeois-ekholm-eliashberg09, seidel12b, abouzaid-seidel13} for related theoretical developments, and \cite{maydanskiy09, maydanskiy-seidel09, abouzaid-seidel10} for applications). Concretely, let's fix a symplectic Lefschetz fibration
\begin{equation} \label{eq:lefschetz-pencil}
\pi: E^{2n} \longrightarrow \bC
\end{equation}
with fibre $M^{2n-2}$. For technical simplicity, we will impose an exactness condition on the symplectic form, hence assume that the fibre is a Liouville domain. Choose a basis of Lefschetz thimbles, and let $\scrA$ be the associated (directed) Fukaya $A_\infty$-algebra, defined over a coefficient field $\bK$. In general, $\scrA$ is $\bZ/2$-graded, and this lifts to a $\bZ$-grading if $c_1(E) = 0$. Natural transformations of degree $n$ from the Serre functor to the identity functor (in the triangulated envelope, or equivalently derived $A_\infty$-category, of $\scrA$) can be described as maps of $A_\infty$-bimodules
\begin{equation} \label{eq:natural-transformation}
\scrA^\vee[-n] \longrightarrow \scrA.
\end{equation}
Here $\scrA$ is the diagonal bimodule; and $\scrA^\vee$ is its dual, to which we have applied an upwards shift by $n$. To be more precise, $A_\infty$-bimodules form a dg category, denoted here by $[\scrA,\scrA]$. By a bimodule map \eqref{eq:natural-transformation}, we mean an element of
\begin{equation} \label{eq:bimodule-morphisms}
H^0(\hom_{[\scrA,\scrA]}(\scrA^\vee[-n],\scrA)) \iso H^n(\hom_{[\scrA,\scrA]}(\scrA^\vee,\scrA)).
\end{equation}
In these terms, the story so far can be summarized as:

\begin{lemma} \label{th:old}
$\scrA$ always comes with a distinguished element of \eqref{eq:bimodule-morphisms}, denoted here by $\rho$.
\end{lemma}

The definition in \cite{seidel12b} uses the geometry and Floer cohomology of Lefschetz thimbles in $E$. However, there is another description (equivalent up to automorphisms of $\scrA^\vee$), in terms of the functor from $\scrA$ to the standard Fukaya category $\scrF(M)$ of the fibre (this second description was the one originally proposed in \cite{seidel06}; the relation between the two is \cite[Corollary 7.1]{seidel12b}).

%
The specific contribution of this paper is to introduce another geometric source of maps \eqref{eq:natural-transformation}. This uses fixed point Floer cohomology for symplectic automorphisms $\phi$ of $M$ (which are exact and equal the identity near the boundary). To define it, we make an auxiliary choice of perturbation, by the time $\epsilon$ map of the Reeb flow near the boundary, and write the outcome as $\mathit{HF}^*(\phi,\epsilon)$ (for small $|\epsilon|$, this appears in \cite{seidel00b, mclean12}; for the general case, see \cite{uljarevic14} or the exposition in the body of this paper). 

\begin{theorem} \label{th:main}
Let $\mu: M \rightarrow M$ be the monodromy of the Lefschetz fibration around a large circle. For sufficiently small $\epsilon>0$, there is a canonical map
\begin{equation} \label{eq:fix-to-aa}
\mathit{HF}^{*+2}(\mu,\epsilon) \longrightarrow H^*(\hom_{[\scrA,\scrA]}(\scrA^\vee[-n],\scrA)).
\end{equation}
\end{theorem}

This is the outcome of a Floer cohomology computation in $E$ which follows \cite{mclean12}, together with an application of a suitable open-closed string map, similar to those in \cite{abouzaid-ganatra14}. Generally, both sides of \eqref{eq:fix-to-aa} are $\bZ/2$-graded; when $c_1(E) = 0$, this lifts to $\bZ$-gradings, and the map between them will respect those gradings.

An important special class of Lefschetz fibrations comes from anticanonical Lefschetz pencils: by which we mean, on a monotone closed symplectic manifold $E^{\mathit{cl}}$, a symplectic Lefschetz pencil of hypersurfaces representing the first Chern class. One obtains $E$ from $E^{\mathit{cl}}$ by removing a suitable neighbourhood of one such hypersurface; hence, $c_1(E) = 0$ always holds in this case. The Reeb flow on $\partial M$ is periodic (say with period $1$), and the monodromy is the boundary twist associated to that periodic flow:
\begin{equation} \label{eq:mu-tau}
\mu = \tau_{\partial M}.
\end{equation}
This satisfies
\begin{equation} \label{eq:trivial-monodromy}
\mathit{HF}^{*+2}(\tau_{\partial M},\epsilon) \iso 
\begin{cases} H^*(M,\partial M) & \epsilon \in (0,1), \\ 
H^*(M) & \epsilon \in (1,2).
\end{cases}
\end{equation}
Theorem \ref{th:main} applies only to the first case of \eqref{eq:trivial-monodromy}. However, by a more precise analysis, one can show that in the second case, an analogue of the map \eqref{eq:fix-to-aa} can be defined in the lowest degree $\ast = 0$. Hence, the element of $\mathit{HF}^2(\tau_{\partial M},\epsilon)$, $\epsilon \in (1,2)$, corresponding to $1 \in H^0(M)$ still gives rise to a map \eqref{eq:bimodule-morphisms}, which we denote by $\sigma$. We summarize the consequence:

\begin{theorem} \label{th:fano}
For a Lefschetz fibration arising from an anticanonical Lefschetz pencil, there is a distinguished pair $(\rho,\sigma)$ of maps \eqref{eq:natural-transformation} (obtained, respectively, by Lemma \ref{th:old}, and a modified version of the argument from Theorem \ref{th:main}).
\end{theorem}


\begin{remark} \label{th:fractional-cy-1}
More generally, one can consider a Lefschetz pencil such that $c_1(E^{\mathit{cl}})$ is $(1+m)$ times the class of the hypersurfaces in the pencil, for some $m \in \bZ$. One then still has $c_1(E) = 0$, but the appropriate generalization of \eqref{eq:trivial-monodromy} is
\begin{equation} \label{eq:trivial-monodromy-2}
\mathit{HF}^{*+2}(\tau_{\partial M},\epsilon) \iso 
\begin{cases} H^{*+2m}(M,\partial M) & \epsilon \in (0,1), \\ 
H^{*+2m}(M) & \epsilon \in (1,2).
\end{cases}
\end{equation}
While $\rho$ still has degree $0$, our construction now leads to a $\sigma$ which has degree $-2m$.
\end{remark}
%
%

\subsection{Noncommutative geometry\label{subsec:nc-geometry}}
The starting point for our algebraic framework is the notion of {\em noncommutative divisor}. While the terminology is new, the concept already appeared in \cite{seidel08}, and a version was considered in \cite{seidel12b, kontsevich-vlassopoulos13}. The last two references include a cyclic symmetry condition, which corresponds more specifically to a {\em noncommutative anticanonical divisor} (see Remark \ref{th:cyclic}); we have chosen not to impose that condition here, for the sake of simplicity. 

While precise definitions will be given later on, it makes sense to give an outline now. Let $\scrA$ be an $A_\infty$-algebra. A noncommutative divisor on $\scrA$ consists of an $A_\infty$-bimodule $\scrP$ which is invertible (with respect to tensor product), together with an $A_\infty$-algebra structure on 
\begin{equation} \label{eq:b-space}
\scrB = \scrA \oplus \scrP[1], 
\end{equation}
which extends the given one on $\scrA$ as well as the $\scrA$-bimodule structure of $\scrP$. If one then considers $\scrB$ as an $\scrA$-bimodule, it fits into a short exact sequence
\begin{equation} \label{eq:a-b-sequence}
0 \rightarrow \scrA \longrightarrow \scrB \longrightarrow \scrP[1] \rightarrow 0.
\end{equation}
As boundary homomorphism of that sequence, one gets a bimodule map $\theta: \scrP \rightarrow \scrA$, which we call the {\em section} associated to the noncommutative divisor. We extend the algebro-geometric language to related structures:
\begin{equation} \label{eq:list-of-categories-1}
\left\{\!\!\!\!\!\!
\parbox{35em}{
\begin{itemize} \itemsep.5em
\item[(i)] The $A_\infty$-algebra $\scrB$ is thought of as describing the {\em divisor by itself} (independently of its relationship with the {\em ambient space} $\scrA$).

\item[(ii)] There is a localization process, in which one makes $\theta$ invertible by passing to a quotient $A_\infty$-algebra. We consider this to be the abstract analogue of taking the {\em complement of the divisor} (see \cite[Section 1]{seidel08} for an explanation), and write the outcome accordingly as $\scrA \setminus \scrB$. This depends only on the section associated to the divisor.
\end{itemize} 
} \right.
\end{equation}

Generalizing the previous concept, we will introduce {\em noncommutative pencils}. As in classical algebraic geometry, such a pencil has an associated pair of bimodule maps (sections, in our geometrically inspired terminology)
\begin{equation}
\rho,\sigma: \scrP \longrightarrow \scrA.
\end{equation}
This leads to enhancements of the previously described constructions:
\begin{equation} \label{eq:list-of-categories-2}
\left\{\!\!\!\!\!\!
\parbox{35em}{
\begin{itemize} \itemsep.5em
\item[(i)] A noncommutative pencil determines a family of noncommutative divisors, parametrized by $z \in \bP^1 = \bK \cup \{\infty\}$. The sections associated to those divisors are (at least up to a scalar multiple)
\[
\begin{aligned}
& \theta_z = \sigma + z\rho \quad \text{for $z \in \bK$}, \\
& \theta_\infty = \rho.
\end{aligned}
\]
In particular, one gets an $A_\infty$-algebra $\scrB_\infty$. More interestingly, one can work with varying $z$. For instance, taking $z = 1/q$, where $q$ is a formal variable, leads to a formal deformation of $\scrB_\infty$, which we denote by $\hat\scrB_\infty$.

\item[(ii)] $\scrA \setminus \scrB_\infty$ also acquires a distinguished (curved) formal deformation, whose deformation parameter has degree $2$. Let's call this the {\em noncommutative Landau-Ginzburg model}, and denote it by $\scrL\scrG$.
\end{itemize}
}\right.
\end{equation}

\begin{remark}
This by no means exhausts the noncommutative geometry structures which arise in this context. For one thing, from a purely algebraic perspective, there is no particular reason to single out the point $z = \infty$, which means that there are more general versions of the constructions listed above. Two other notions, the {\em graph} and {\em base locus} of a noncommutative pencil, will be mentioned briefly in Remark \ref{th:graph}. Different kinds of algebraic structures, notably ones involving Hochschild and cyclic homology (such as, the Gau\ss-Manin connection for the periodic cyclic homology of the fibres of the noncommutative pencil), are entirely beyond the scope of our discussion.
\end{remark}

\subsection{Speculations}
The $A_\infty$-algebra $\scrA$ associated to a Lefschetz fibration comes with a natural structure of a noncommutative divisor, with associated bimodule 
\begin{equation} \label{eq:serre-n}
\scrP = \scrA^\vee[-n]. 
\end{equation}
The algebraic constructions introduced above have the following meaning:
\begin{equation} \label{eq:list-of-categories-3}
\left\{\!\!\!\!\!\!
\parbox{35em}{
\begin{itemize} \itemsep.5em
\item[(i)] The divisor by itself, $\scrB$, is the full subcategory of $\scrF(M)$ consisting of the vanishing cycles in our basis. In fact, this is how the noncommutative divisor was constructed in \cite{seidel06,seidel08}, by identifying $\scrA$ with the directed $A_\infty$-subcategory of $\scrB$.

\item[(ii)] The complement of the divisor, $\scrA \setminus \scrB$, yields a full subcategory of the wrapped Fukaya category of the total space, $\scrW(E)$. This is the main result of \cite{abouzaid-seidel13}.
\end{itemize}
} \right.
\end{equation}

We will now propose a continuation of this line of thought, to the point where it can include Theorem \ref{th:fano}. In spite of its tentative nature, this direction seems worth while exploring, because of the potential implications for symplectic geometry and mirror symmetry. 

\begin{conjecture} \label{th:conjecture}
Take an exact symplectic Lefschetz fibration coming from an anticanonical Lefschetz pencil. Then, the $A_\infty$-algebra $\scrA$ carries a canonical structure of a noncommutative pencil, satisfying \eqref{eq:serre-n}, whose associated sections are those described in Theorem \ref{th:fano}.
\end{conjecture}

The appeal of the noncommutative pencil structure in this context is that it can serve as an intermediate object between various other Fukaya $A_\infty$-structures. Consider:
\begin{equation} \label{eq:list-of-categories-4}
\left\{\!\!\!\!\!\!
\parbox{35em}{
\begin{itemize} \itemsep.5em
\item[(i)] $\scrB_\infty$ should be a full subcategory of $\scrF(M)$ (formally, this is the same statement as in \eqref{eq:list-of-categories-3}(i); but the level of difficulty of proving it depends on how one would go about constructing the noncommutative pencil structure). Now, $M$ admits a natural closure $M^{\mathit{cl}}$, obtained by gluing in a disc bundle over the base locus of the original pencil. Associated to that is the relative Fukaya category \cite{seidel02, seidel03b, sheridan11b}, which is a formal deformation of $\scrF(M)$. More generally, one can modify the deformation by including a suitable bulk term, as in \cite{fukaya-oh-ohta-ono11}. The conjecture is that, for a specific choice of bulk term, and after a change in the deformation parameter, this deformation is $\hat\scrB_\infty$. This is made more precise in \cite{seidel15}.

\item[(ii)]
$E$ itself embeds into a closed manifold $E^{\mathit{cl}}$, the manifold on which the original Lefschetz pencil was defined. This leads to another relative Fukaya category, but where now (because of monotonicity) the deformation parameter has degree $2$. At least on the basis of formal analogy, one expects this to be related (by a cohomologically full embedding) to the Landau-Ginzburg category $\scrL\scrG$ from \eqref{eq:list-of-categories-2}(ii).
\end{itemize}
} \right.
\end{equation}

\begin{remark} \label{th:fractional-cy-2}
We have limited ourselves to anticanonical Lefschetz pencils (the setup of Theorem \ref{th:fano}) for the sake of concreteness, but other cases are also of interest. For instance, along the same lines as in Conjecture \ref{th:conjecture}, the situation from Remark \ref{th:fractional-cy-1} should lead to a $\bZ/2 m$-graded noncommutative pencil. The residual information given by the $\bZ$-grading of $\scrA$ means that this pencil is homogeneous with respect to a circle action on $\bP^1$ of weight $2m$. 
\end{remark}

\subsection{Homological Mirror Symmetry\label{subsec:mirror}}
The mirror ($B$-model) counterpart of Conjecture \ref{th:conjecture} is far more straightforward, and amounts to embedding ordinary algebraic geometry into its noncommutative cousin. Suppose that we have a smooth algebraic variety $A$, a line bundle $P$ on it, and a section $t \in \Gamma(A,P^{-1})$ of its dual. Let $\scrA$ be an $A_\infty$-algebra underlying the bounded derived category $D^b\mathit{Coh}(A)$. Then, $P$ gives rise to an invertible bimodule $\scrP$, and $t$ to a bimodule map $\theta: \scrP \rightarrow \scrA$. In fact, there is a noncommutative divisor structure on $\scrA$ of which $\theta$ is part. This fits in with mirror symmetry in the formulation of \cite{auroux07}, where the crucial ingredient is a choice of anticanonical divisor (see \cite[Section 6]{seidel08}, where the connection is made explicit in a simple example). Similarly, suppose that for $A$ and $P$ as before, we have two elements of $\Gamma(A,P^{-1})$; this yields the structure of a noncommutative pencil on $\scrA$. 

Homological Mirror Symmetry can now be thought of as operating on the level of pencils, schematically like this:
\begin{equation}
{\begin{tabular}{|c|} \hline Symplectic \\ geometry \\($A$-model) \\ \hline \it Symplectic \\ \it pencil \\ \hline \end{tabular}}
\longrightarrow 
{\begin{tabular}{|c|} \hline Homological algebra \\ (Noncommutative geometry) \\ \hline \it Noncommutative pencil \\ \hline \end{tabular}}
\longleftarrow
{\begin{tabular}{|c|} \hline Algebraic \\ geometry \\ ($B$-model) \\ \hline \it Algebraic \\ \it pencil\\ \hline \end{tabular}}
\end{equation}
As a consequence, one gets a unified approach to the associated equivalences of categories. By \eqref{eq:list-of-categories-4}(ii), this picture includes counts of holomorphic discs, following the model of \cite{auroux07}.

{\em Acknowledgments.} Conversations with Anatoly Preygel were useful at an early stage of this project. Mohammed Abouzaid and Maxim Kontsevich provided very valuable comments about \eqref{eq:list-of-categories-4}(i), which are reflected in its current formulation. The definition \eqref{eq:katz1} was suggested by Ludmil Katzarkov. I am grateful to the anonymous referee for keeping me on the straight and narrow. Partial support was provided by: NSF grant DMS-1005288; the Simons Foundation, through a Simons Investigator award; and the Radcliffe Institute for Advanced Study at Harvard University, through a Radcliffe Fellowship. I would also like to thank Boston College, where the first version of this paper was written, for its hospitality.

\section{Homological algebra\label{sec:algebra}}

This section is of an elementary algebraic nature. To make the exposition more self-contained, we recall the relation between Hochschild homology and bimodule maps. We then discuss noncommutative divisors, largely following \cite[Section 3]{seidel08} except for the terminology. Finally, we adapt the same ideas to noncommutative pencils.

\subsection{$A_\infty$-bimodules}
Let $\scrA$ be a graded vector space over a field $\bK$. We denote by $T(\scrA[1])$ the tensor algebra over the (downwards) shifted space $\scrA[1]$. The structure of an $A_\infty$-algebra on $\scrA$ is given by a map
\begin{equation} \label{eq:a}
\mu_{\scrA}: T(\scrA[1]) \longrightarrow \scrA[2],
\end{equation}
which vanishes on the constants $\bK \subset T(\scrA[1])$. The components of \eqref{eq:a} are operations $\mu_{\scrA}^d: \scrA^{\otimes d} \rightarrow \scrA[2\!-\!d]$ (for $d \geq 1$). We will assume that $\scrA$ is strictly unital, denoting the unit by $e_{\scrA}$, and write $\bar{\scrA} = \scrA/\bK\,e_{\scrA}$. 

An $A_\infty$-bimodule over $\scrA$ consists of a graded vector space $\scrP$ and a map
\begin{equation}
\mu_{\scrP}: T(\scrA[1]) \otimes \scrP \otimes T(\scrA[1]) \longrightarrow \scrP[1],
\end{equation}
whose components we write as $\mu_{\scrP}^{s;1;r}: \scrA^{\otimes s} \otimes \scrP \otimes \scrA^{\otimes r} \rightarrow \scrP[1\!-\!r\!-\!s]$ (for $r,s \geq 0$). All $A_\infty$-bimodules will be assumed to be strictly unital. A basic example is the diagonal bimodule $\scrP = \scrA$, which has
\begin{equation}
\mu_{\scrA}^{s;1;r}(a'_s,\dots,a'_1;a;a_r,\dots,a_1) = (-1)^{\|a_1\| + \cdots + \|a_r\|+1} 
\mu_{\scrA}^{r+1+s}(a'_s,\dots,a'_1,a,a_r,\dots,a_1).
\end{equation}
Here and later on, $\|a\| = |a|-1$ stands for the reduced degree.
$A_\infty$-bimodules over $\scrA$ form a dg category, which we denote by $[\scrA,\scrA]$. An element of $\mathit{hom}_{[\scrA,\scrA]}(\scrP,\scrQ)$ of degree $d$ is given by a map
\begin{equation} \label{eq:bimodule-map}
\phi: T(\scrA[1]) \otimes \scrP \otimes T(\scrA[1]) \longrightarrow \scrQ[d],
\end{equation}
which factors through the projection to $T(\bar{\scrA}[1])$ on both sides. We denote the components of \eqref{eq:bimodule-map} again by $\phi^{s;1;r}$. The cocycles in $\mathit{hom}_{[\scrA,\scrA]}(\scrP,\scrQ)$ are solutions of
\begin{equation} \label{eq:bimodule-d}
\begin{aligned}
& \;\;\; \sum_{i,j} (-1)^{|\phi| (\|a_1\|+\cdots+\|a_i\|)} \; \mu_{\scrQ}^{s-j;1;i}(a'_s,\dots;
\phi^{j;1;r-i}(a'_j,\dots,a'_1;p;a_r,\dots,a_{i+1});\dots,a_1) \\
= & \;\;\;\sum_{i,j} (-1)^{|\phi|+\|a_1\|+\cdots+\|a_i\|} \;
\phi^{s-j;1;i}(a'_s,\dots;\mu^{j;1;r-i}_{\scrP}(a'_j,\dots,a'_1;p;a_r,\dots,a_{i+1});\dots,a_1) \\
& \!+ \sum_{i,j} (-1)^{|\phi|+\|a_1\|+\cdots+\|a_i\|} \;
\phi^{s;1;r-j+1}(a'_s,\dots;p;a_r,\dots,\mu_{\scrA}^j(a_{i+j},\dots,a_{i+1}),\dots,a_1) \\
& \!+ \sum_{i,j} (-1)^{|\phi|+\|a_1\|+\cdots+\|a_r\|+|p|+\|a'_1\|+\cdots+\|a'_i\|}
\phi^{s-j+1;1;r}(a'_s,\dots, \\[-1.5em]
& \qquad \qquad \qquad \qquad \qquad \qquad \qquad \mu_{\scrA}^j(a_{i+j}',\dots,a_{i+1}'),\dots;p;a_r,\dots,a_1).
\end{aligned}
\end{equation}
For instance, the identity endomorphism $\mathit{id}_{\scrP}$ has only one nonzero component, $\mathit{id}_{\scrP}^{0;1;0}(p) = p$.

To relate these structures to their (classical) cohomology level counterparts, on equips $H(\scrA)$ with the graded associative algebra structure given by
\begin{equation}
[a_2] \cdot [a_1] = (-1)^{|a_1|} [\mu^2_{\scrA}(a_2,a_1)],
\end{equation}
and makes $H(\scrP)$ into a graded bimodule over this algebra by setting
\begin{equation}
\begin{aligned}
& [a'] \cdot [p] = -(-1)^{|p|} [\mu^{1;1;0}_{\scrP}(a';p)], 
\\
& [p] \cdot [a] = [\mu^{0;1;1}_{\scrP}(p;a)].
\end{aligned}
\end{equation}
If $\phi: \scrP \rightarrow \scrQ$ is as in \eqref{eq:bimodule-d}, the associated cohomology level bimodule map is
\begin{equation}
\begin{aligned}
& H(\phi): H^*(\scrP) \longrightarrow H^{*+|\phi|}(\scrQ), \\
& [p] \longmapsto (-1)^{|\phi|\,|p|} [\phi^{0;1;0}(p)].
\end{aligned}
\end{equation}

\begin{conventions}
Our sign conventions for $A_\infty$-algebras follow \cite{seidel04}. The sign conventions for $A_\infty$-bimodules follow \cite{seidel08} except that, for greater compatibility with \cite{seidel04}, we reverse the ordering of the entries. This applies in particular to \eqref{eq:bimodule-d}, which agrees with \cite{seidel08} up to ordering (but differs from the convention for $A_\infty$-modules in \cite{seidel04}). 

The following observation may help address some sign issues. Given any $A_\infty$-bimodule structure, the sign change
\begin{equation} \label{eq:tilde-p}
\begin{aligned}
& \mu_{\scrP}^{s;1;r}(a_s',\dots,a_1';p;a_r,\dots,a_1) \longmapsto \\ & \qquad (-1)^{\|a_1\|+\cdots+\|a_r\|+\|a_1'\|+\cdots+\|a_s'\|+1} \mu_{\scrP}^{s;1;r}(a_s',\dots,a_1'; p; a_r,\dots,a_1)
\end{aligned}
\end{equation}
yields another $A_\infty$-bimodule structure. The two structures are generally distinct, but isomorphic (by an isomorphism that acts by $\pm 1$ on each graded piece). Hence, whenever we define some construction of $A_\infty$-bimodules, there are potentially at least two equivalent versions, which differ by \eqref{eq:tilde-p}.
\end{conventions}

We need to recall a few specific operations on $A_\infty$-bimodules:
\begin{itemize} \itemsep1em
\item The shifted space $\scrP[1]$ becomes an $A_\infty$-bimodule, with 
\begin{equation}
\mu_{\scrP[1]}^{s;1;r}(a_s',\dots,a_1';p;a_r,\dots,a_1) = (-1)^{\|a_1\|+\cdots+\|a_r\|+1}
\mu_{\scrP}^{s;1;r}(a_s',\dots,a_1';p;a_r,\dots,a_1).
\end{equation}
\item
The dual is $\scrP^\vee = \mathit{Hom}(\scrP,\bK)$, with
\begin{equation}
\langle \mu_{\scrP^\vee}^{s;1;r}(a_s,\dots,a_1;\pi;a'_r,\dots,a'_1), p \rangle = 
(-1)^{|p|+1} \langle \pi, \mu_{\scrP}^{r;1;s}(a'_r,\dots,a'_1;p;a_s,\dots,a_1) \rangle.
\end{equation}
\item Given two $A_\infty$-bimodules $\scrQ$ and $\scrP$, one defines the tensor product $\scrQ \otimes_{\scrA} \scrP$ to be the graded vector space $\scrQ \otimes T(\bar\scrA[1]) \otimes \scrP$, with the differential
\begin{equation}
\begin{aligned}
& \mu^{0;1;0}_{\scrQ \otimes_{\scrA} \scrP}(q \otimes a'_t \otimes \cdots \otimes a'_1 \otimes p) = \\
& \qquad
 \sum_i (-1)^{|p|+\|a_1'\|+\cdots+\|a_{i}'\|} \mu_{\scrQ}^{0;1;t-i}(q;a'_t,\dots,a'_{i+1}) \otimes \cdots \otimes a'_1 \otimes p 
\\ 
& \quad + \sum_{i,j} (-1)^{|p|+\|a_1'\|+\cdots+\|a_i'\|} q \otimes a'_t \otimes \cdots \otimes \mu_{\scrA}^j(a'_{i+j},\dots,a'_{i+1}) \otimes \cdots \otimes a'_1 \otimes p \\
& \quad + \sum_i q \otimes a'_t \otimes \cdots \otimes \mu_{\scrP}^{i;1;0}(a'_i,\dots,a'_1;p);
\end{aligned}
\end{equation}
the operations (for $r>0$ or $s>0$)
\begin{equation}
\begin{aligned}
& \mu^{s;1;0}_{\scrQ \otimes_{\scrA} \scrP}(a''_s,\dots,a''_1; q \otimes a'_t \otimes \cdots \otimes a'_1 \otimes p) = \\
& \quad \sum_i (-1)^{|p|+\|a_1'\|+\cdots+\|a_{i}'\|} \mu_{\scrQ}^{s;1;t-i}(a''_s,\dots,a''_1;q;a'_t,\dots,a'_{i+1}) 
\otimes \cdots \otimes a'_1 \otimes p,
\\
& \mu^{0;1;r}_{\scrQ \otimes_{\scrA} \scrP}(q \otimes a'_t \otimes \cdots \otimes a'_1 \otimes p; a_r,\dots, a_1) = \\
& \quad \sum_i q \otimes a'_t \otimes \cdots \otimes \mu_{\scrP}^{i;1;r}(a'_i,\dots,a'_1;p;a_r,\dots,a_1);
\end{aligned}
\end{equation}
and with $\mu_{\scrQ \otimes_{\scrA} \scrP}^{s;1;r} = 0$ if both $r$ and $s$ are positive.
\end{itemize}

Tensor product with the diagonal bimodule is essentially a trivial operation. More precisely, there are canonical quasi-isomorphisms \cite[Eqn.~(2.21)--(2.24)]{seidel08}
\begin{equation} \label{eq:neutral}
\scrA \otimes_{\scrA} \scrP \htp \scrP \htp \scrP \otimes_{\scrA} \scrA.
\end{equation}
For arbitrary $\scrA$-bimodules $\scrP$, $\scrQ$, $\scrR$, there are canonical isomorphisms 
\begin{align}
&
\label{eq:2-rotation}
\mathit{hom}_{[\scrA,\scrA]}(\scrP,\scrQ^\vee) \iso \mathit{hom}_{[\scrA,\scrA]}(\scrQ,\scrP^\vee),
\\
&
\label{eq:3-rotation}
\mathit{hom}_{[\scrA,\scrA]}(\scrP \otimes_{\scrA} \scrQ, \scrR^\vee) \iso \mathit{hom}_{[\scrA,\scrA]}(\scrQ \otimes_{\scrA} \scrR, \scrP^\vee) \iso
\mathit{hom}_{[\scrA,\scrA]}(\scrR \otimes_{\scrA} \scrP, \scrQ^\vee).
\end{align}
As an application of \eqref{eq:neutral} and \eqref{eq:3-rotation}, one gets
\begin{align}
& \label{eq:dual-rotate-1}
\mathit{hom}_{[\scrA,\scrA]}(\scrP^\vee,\scrP^\vee) 
\htp \mathit{hom}_{[\scrA,\scrA]}(\scrA \otimes_{\scrA} \scrP^\vee, \scrP^\vee) 
\iso \mathit{hom}_{[\scrA,\scrA]}(\scrP^\vee \otimes_{\scrA} \scrP, \scrA^\vee), 
\\
& \label{eq:dual-rotate-2}
\mathit{hom}_{[\scrA,\scrA]}(\scrP^\vee,\scrP^\vee) \htp \mathit{hom}_{[\scrA,\scrA]}(\scrP^\vee \otimes_{\scrA} \scrA, \scrP^\vee) \iso \mathit{hom}_{[\scrA,\scrA]}(\scrP \otimes_{\scrA} \scrP^\vee, \scrA^\vee).
\end{align}
By starting with $\mathit{id}_{\scrP^\vee}$ and going through \eqref{eq:dual-rotate-1}, \eqref{eq:dual-rotate-2}, one obtains canonical maps
\begin{align} \label{eq:left-pi}
& \lambda_{\mathit{left}}: \scrP^\vee \otimes_{\scrA} \scrP \longrightarrow \scrA^\vee, \\
\label{eq:right-pi}
& \lambda_{\mathit{right}}: \scrP \otimes_{\scrA} \scrP^\vee \longrightarrow \scrA^\vee.
\end{align}
There are explicit formulae for \eqref{eq:neutral}--\eqref{eq:right-pi}, but we prefer to omit them. For more foundational material on $A_\infty$-bimodules, see \cite{tradler01, lefevre, kontsevich-soibelman06, lyubashenko-manzyuk08, seidel08}.

\subsection{Invertible bimodules}
%
%
A bimodule $\scrP$ is called invertible if there is another bimodule $\scrP^{-1}$ and quasi-isomorphisms
\begin{equation}
\scrP^{-1} \otimes_{\scrA} \scrP \htp \scrA \htp \scrP \otimes_{\scrA} \scrP^{-1}.
\end{equation}
In that case, taking the tensor product with $\scrP$ (on objects, and the tensor product with $\mathit{id}_{\scrP}$ on morphisms) is an automorphism, or more rigorously a quasi-equivalence, of the dg category $[\scrA,\scrA]$.

\begin{lemma}
If $\scrP$ is invertible, $\lambda_{\mathit{left}}$ and $\lambda_{\mathit{right}}$ are quasi-isomorphisms. 
\end{lemma}

\begin{proof}
Consider the chain of quasi-isomorphisms
\begin{equation} \label{eq:quasi-shift}
\begin{aligned}
& \mathit{hom}_{[\scrA,\scrA]}(\scrQ,\scrP^\vee) \htp \hom_{[\scrA,\scrA]}(\scrA \otimes_{\scrA} \scrQ, \scrP^\vee) \\ & \qquad \iso \hom_{[\scrA,\scrA]}(\scrQ \otimes_{\scrA} \scrP, \scrA^\vee) \htp
\hom_{[\scrA,\scrA]}(\scrQ, \scrA^\vee \otimes_{\scrA} \scrP^{-1}).
\end{aligned}
\end{equation}
This is functorial in $\scrQ$, hence (by the Yoneda Lemma) arises from a quasi-isomorphism
\begin{equation} \label{eq:v-inverse-quasi-iso}
\scrP^\vee \htp \scrA^\vee \otimes_{\scrA} \scrP^{-1}. 
\end{equation}
More concretely, one obtains that quasi-isomorphism by setting $\scrQ = \scrP^\vee$ in \eqref{eq:quasi-shift}, and then taking the identity in the leftmost morphism group. By comparing this with \eqref{eq:dual-rotate-1}, one sees that the quasi-isomorphism is in fact $\lambda_{\mathit{left}} \otimes_{\scrA} \mathit{id}_{\scrP^{-1}}$. By tensoring with $\mathit{id}_{\scrP}$, it follows that $\lambda_{\mathit{left}}$ itself is a quasi-isomorphism. The other case is similar.
\end{proof}

Note that, by combining \eqref{eq:v-inverse-quasi-iso} with its counterpart for $\scrP^{-1} \otimes_{\scrA} \scrA^\vee$, we can get a canonical quasi-isomorphism (for invertible $\scrP$)
\begin{equation} \label{eq:central-tensor}
\scrA^\vee \otimes_{\scrA} \scrP^{-1} \htp \scrP^{-1} \otimes_{\scrA} \scrA^\vee.
\end{equation}

\subsection{Hochschild homology\label{subsec:hochschild-homology}}
The Hochschild homology of $\scrA$ with coefficients in a bimodule $\scrP$, written as $\mathit{HH}_*(\scrA,\scrP)$, is the homology of $\mathit{CC}_*(\scrA,\scrP) = T(\bar\scrA[1]) \otimes \scrP$, with differential
\begin{equation}
\begin{aligned}
\partial(a_d \otimes \cdots \otimes a_1 \otimes p) = 
& \sum_{i,j} (-1)^{|p|+\|a_1\|+\cdots+\|a_i\|} a_d \otimes \cdots \otimes \mu_{\scrA}^j(a_{i+j},\dots,a_{i+1}) \otimes \cdots \otimes a_1 \otimes p \\[-1em]
& \qquad + \sum_{i,j} (-1)^\ast a_{d-i} \otimes \cdots \otimes \mu_{\scrP}^{j;1;i}(a_j,\dots,a_1;p;a_d,\dots,a_{d-i+1}),
\end{aligned}
\end{equation}
where $\ast = (\|a_{d-i+1}\| + \cdots + \|a_d\|)(|p|+\|a_1\| + \cdots + \|a_{d-i}\|)$. In spite of the subscript notation, the grading on $\mathit{CC}_*(\scrA,\scrP)$ is cohomological, meaning that $\partial$ has degree $1$. 

There is a canonical isomorphism
\begin{equation} \label{eq:general-hochschild}
\mathit{CC}_*(\scrA, \scrQ \otimes_{\scrA} \scrP)^\vee \iso
\mathit{hom}_{[\scrA,\scrA]}(\scrP,\scrQ^\vee).
\end{equation}
In particular,
%
\begin{equation} \label{eq:second-hochschild}
\mathit{CC}_*(\scrA,\scrP)^\vee \htp
\hom_{[\scrA,\scrA]}(\scrA,\scrP^\vee).
\end{equation}
By spelling out \eqref{eq:second-hochschild}, one obtains the following:

\begin{lemma} \label{th:recognize-diagonal}
Suppose that we have a cocycle $\xi \in \mathit{CC}_0(\scrA,\scrP)^\vee$, whose constant term $\xi^0 \in \scrP^\vee$ has the property that
\begin{equation} \label{eq:zero-order-map}
\begin{aligned}
& H^*(\scrA) \longrightarrow H^*(\scrP)^\vee, \\
& [a] \longmapsto [\xi^0(\mu_{\scrP}^{1;1;0}(a; \cdot))]
\end{aligned}
\end{equation}
is an isomorphism. Then, under \eqref{eq:second-hochschild}, $\xi$ gives rise to a quasi-isomorphism $\scrA \htp \scrP^\vee$.
\end{lemma}

If $\scrA$ is proper (has finite-dimensional cohomology), the natural map $\scrA \rightarrow (\scrA^\vee)^\vee$ is a quasi-isomorphism. In that case, \eqref{eq:general-hochschild} implies that
\begin{equation} \label{eq:first-hochschild}
\hom_{[\scrA,\scrA]}(\scrP,\scrA) \htp
\mathit{CC}_*(\scrA, \scrA^\vee \otimes_{\scrA} \scrP)^\vee.
\end{equation}

Hochschild homology is cyclically invariant, which means that there is a canonical isomorphism of chain complexes
\begin{equation} \label{eq:cyclic-hh}
\mathit{CC}_*(\scrA,\scrQ \otimes_{\scrA} \scrP) \iso \mathit{CC}_*(\scrA,\scrP \otimes_{\scrA} \scrQ).
\end{equation}
The dual of this isomorphism, using \eqref{eq:general-hochschild}, is \eqref{eq:2-rotation}. Similarly, \eqref{eq:cyclic-hh} implies that
\begin{equation}
\mathit{CC}_*(\scrA,\scrR \otimes_{\scrA} \scrQ \otimes_{\scrA} \scrP) \iso 
\mathit{CC}_*(\scrA,\scrP \otimes_{\scrA} \scrR \otimes_{\scrA} \scrQ),
\end{equation}
and the dual of that is \eqref{eq:3-rotation}. One can apply the same idea to the tensor powers of a single bimodule, leading to an action of the cyclic group $\bZ/i$ on $\mathit{CC}_*(\scrA,\scrP^{\otimes_{\scrA} i})$. By applying \eqref{eq:first-hochschild}, one gets a $\bZ/i$-action on $H^*(\mathit{hom}_{[\scrA,\scrA]}( (\scrA^\vee)^{\otimes_{\scrA} i-1}, \scrA))$, assuming that $\scrA$ is proper.

\subsection{Hochschild cohomology\label{subsec:hochschild-coho}}
Even though we have emphasized Hochschild homology, it makes sense to also mention the Hochschild cohomology $\mathit{HH}^*(\scrA,\scrP)$, whose underlying chain complex is $\mathit{CC}^*(\scrA,\scrP) = \mathit{Hom}(T(\bar\scrA[1]), \scrP)$. It is well-known that
\begin{equation} \label{eq:hh-diag}
\mathit{CC}^*(\scrA,\scrP) \htp \mathit{hom}_{[\scrA,\scrA]}(\scrA,\scrP).
\end{equation}
One can interpret \eqref{eq:second-hochschild} and Lemma \ref{th:recognize-diagonal} in those terms, given that
\begin{equation}
\mathit{CC}_*(\scrA,\scrP)^\vee = \mathit{CC}^*(\scrA,\scrP^\vee).
\end{equation}

There is another, less straightforward, connection between the two Hochschild theories. Namely, suppose that $\scrA$ is proper and homologically smooth, in which case $\scrA^\vee$ is always invertible (see e.g.\ \cite[Theorem 4.5]{shklyarov07b}). Then (see e.g.\ \cite[Remark 8.2.4]{kontsevich-soibelman06})
\begin{equation} \label{eq:serre-hochschild}
\mathit{CC}_*(\scrA,\scrP) \htp 
\mathit{hom}_{[\scrA,\scrA]}((\scrA^\vee)^{-1},\scrP) \htp
\mathit{CC}^*(\scrA,\scrA^\vee \otimes_{\scrA} \scrP)
\end{equation}
If one specializes this to $\scrP = \scrA$, the outcome is a quasi-isomorphism between $\mathit{CC}_*(\scrA,\scrA)$ and its dual, which gives rise to a nondegenerate pairing on $\mathit{HH}_*(\scrA,\scrA)$ (see \cite{shklyarov07b} for an extended discussion). 

\subsection{Self-conjugation}
Fix an invertible bimodule $\scrP$, and some $i \geq 1$. We define an endomorphism of $\mathit{hom}_{[\scrA,\scrA]}(\scrP^{\otimes_{\scrA} i}, \scrA)$, unique up to chain homotopy, through the homotopy commutative diagram
\begin{equation} \label{eq:k-diagram}
\xymatrix{
\mathit{hom}_{[\scrA,\scrA]}(\scrP^{\otimes_{\scrA} i}, \scrA)
\ar[d]_-{\mathit{id}_{\scrP} \otimes_{\scrA} \cdot}^-{\htp}
\ar@{-->}[rr]
&&
\mathit{hom}_{[\scrA,\scrA]}(\scrP^{\otimes_{\scrA} i}, \scrA)
\ar[d]^-{\cdot \otimes_{\scrA} \mathit{id}_{\scrP}}_-{\htp} 
\\
\mathit{hom}_{[\scrA,\scrA]}(\scrP^{\otimes_{\scrA} i+1},\scrP \otimes_{\scrA} \scrA)
\ar[r]^-{\htp}
&
\mathit{hom}_{[\scrA,\scrA]}(\scrP^{\otimes_{\scrA} i+1},\scrP)
&
\mathit{hom}_{[\scrA,\scrA]}(\scrP^{\otimes_{\scrA} i+1},\scrA \otimes_{\scrA} \scrP).
\ar[l]_-{\htp}
}
\end{equation}
Again up to homotopy, these maps are compatible with the ring structure induced by the tensor product
\begin{equation}
\begin{aligned}
& \mathit{hom}_{[\scrA,\scrA]}(\scrP^{\otimes_{\scrA} j}, \scrA) \otimes
\mathit{hom}_{[\scrA,\scrA]}(\scrP^{\otimes_{\scrA} i}, \scrA) \longrightarrow
\mathit{hom}_{[\scrA,\scrA]}(\scrP^{\otimes_{\scrA} i+j}, \scrA \otimes_{\scrA} \scrA)
\\ & \qquad \qquad \qquad \qquad \qquad \qquad \qquad \qquad \qquad \qquad
\stackrel{\htp}{\longrightarrow} \mathit{hom}_{[\scrA,\scrA]}(\scrP^{\otimes_{\scrA} i+j}, \scrA).
\end{aligned}
\end{equation}
We call the induced cohomology automorphisms {\em self-conjugation maps}, and denote them by
\begin{equation} \label{eq:c-automorphism}
\xymatrix{
\ar@(dl,ul)[0,0]^-{K_i} & \hspace{-3em}
H^*(\mathit{hom}_{[\scrA,\scrA]}(\scrP^{\otimes_{\scrA} i},\scrA)).
}
\end{equation}

\begin{example} \label{th:exterior-algebra}
Let $A = \Lambda^*(\bK)$ be the exterior algebra in one variable (denoted by $x$, and given degree $1$). Given any $\lambda \in \bK^\times$, one has an invertible $A$-bimodule $P_{\lambda}$, whose underlying graded vector space is the same as $A$, but where the action of $x$ on the right is multiplied by $\lambda$ (while that on the left remains the same). Consider the bimodule map (of degree $1$) $t: P_{\lambda} \rightarrow A$ which sends $1$ to $x$. Then
\begin{equation}
\begin{cases}
P_{\lambda} \otimes_A P_{\lambda} \xrightarrow{t \otimes_A \mathit{id}_{P_\lambda}} A \otimes_A P_{\lambda} \iso P_{\lambda}
& \text{maps $1 \otimes 1 \mapsto x$,} \\
P_{\lambda} \otimes_A P_{\lambda} \xrightarrow{\mathit{id}_{P_\lambda} \otimes_A t} P_{\lambda} \otimes_A A \iso P_{\lambda}
& \text{maps $1 \otimes 1 \mapsto \lambda x$.}
\end{cases}
\end{equation}
This translates straightforwardly into the $A_\infty$-world, and provides an example where \eqref{eq:c-automorphism}, for $i = 1$, is not the identity.
\end{example}

For any invertible bimodule $\scrP$, there is a homotopy commutative diagram
\begin{equation} \label{eq:top-diagram}
\xymatrix{
\hom_{[\scrA,\scrA]}(\scrA^\vee, \scrA^\vee \otimes_{\scrA} \scrP^{-1}) 
\ar[dd]_-{\eqref{eq:central-tensor}}
&&
\hom_{[\scrA,\scrA]}(\scrA, \scrP^{-1})
\ar[ll]_-{\mathit{id}_{\scrA^\vee} \otimes_{\scrA} \cdot}
\ar[d]_-{\cdot \otimes_{\scrA} \mathit{id}_{\scrP}}
\\ 
& 
\hom_{[\scrA,\scrA]}(\scrA^\vee,\scrP^\vee)  
\ar[ul]
\ar[dl]
&
\hom_{[\scrA,\scrA]}(\scrP,\scrA)
\ar[l]_-{\text{dualize}}
\\
\hom_{[\scrA,\scrA]}(\scrA^\vee,\scrP^{-1} \otimes_{\scrA} \scrA^\vee)
&&
\hom_{[\scrA,\scrA]}(\scrA,\scrP^{-1})
\ar[ll]^-{\cdot \otimes_{\scrA} \mathit{id}_{\scrA^\vee}}
\ar[u]^-{\mathit{id}_{\scrP} \otimes_{\scrA} \cdot}
}
\end{equation}
The diagonal arrows are \eqref{eq:v-inverse-quasi-iso} and its counterpart; hence, the left triangle in the diagram commutes by definition of \eqref{eq:central-tensor}. Proving the commutativity of the remaining parts requires one to go back to \eqref{eq:quasi-shift}, and we will not explain the details here. 
%
Going vertically down the right column of \eqref{eq:top-diagram} gives an automorphism of $H^*(\mathit{hom}_{[\scrA,\scrA]}(\scrA,\scrP^{-1}))$, which is a form of $K_1$. Suppose that $\scrA^\vee$ is an invertible bimodule. Then, by going around the diagram the other way, we get another description of that automorphism, which now involves the interaction of $\scrP$ and $\scrA^\vee$.

\begin{example}
Suppose that $\scrA$ is weakly Calabi-Yau of dimension $n$, by which we mean that it comes with a quasi-isomorphism 
\begin{equation} \label{eq:weak-cy-structure}
\scrA^\vee[-n] \htp \scrA.
\end{equation}
By inserting that quasi-isomorphism on both sides of \eqref{eq:central-tensor}, one gets a distinguished automorphism of $\scrP^{-1}$ (and correspondingly of $\scrP$), which describes the failure of the tensor product with that bimodule to be compatible with the Calabi-Yau structure. It then follows from \eqref{eq:top-diagram} that $K_1$ is the composition with that automorphism. This agrees with Example \ref{th:exterior-algebra}.
\end{example}

%
\begin{lemma} \label{th:z-action}
Suppose that $\scrA$ is proper. Suppose also that $\scrA^\vee$ is invertible (which holds if $\scrA$ is smooth). Then, the generator of the $\bZ/(i+1)$-action on $H^*(\mathit{hom}_{[\scrA,\scrA]}((\scrA^\vee)^{\otimes_{\scrA} i},\scrA))$ equals the self-conjugation map $K_i$, for the bimodule $\scrA^\vee$.
\end{lemma}

We will only explain the counterpart of this in classical algebra. Given a finite-dimensional algebra $A$, a bimodule map 
\begin{equation} \label{eq:classical-bimodule}
(A^\vee)^{\otimes_A i} \longrightarrow A
\end{equation}
is given by an element
\begin{equation} \label{eq:tensor-element}
\sum_j a^j_{i+1} \otimes \cdots \otimes a^j_1 \in A^{\otimes_{\bK} i+1}
\end{equation}
satisfying the equation
\begin{equation} \label{eq:alpha-eq}
\sum_j a^j_{i+1}a \otimes a^j_i \otimes \cdots \otimes a^j_1 = \sum_j
a^j_{i+1} \otimes aa^j_i \otimes \cdots \otimes a^j_1
\end{equation}
for $a \in A$, as well as all its cyclic permutations. The bimodule map associated to \eqref{eq:tensor-element} is
\begin{equation} \label{eq:cyclic-relations}
\alpha_1 \otimes \cdots \otimes \alpha_i \longmapsto 
\sum_j \alpha_1(a^j_1) \cdots \alpha_i(a^j_i) a^j_{i+1} \in A.
\end{equation}
If one tensors this with the identity map of $A^\vee$ on the right or left, the outcome are bimodule maps $(A^\vee)^{\otimes_A i+1} \rightarrow A^\vee$ given by, respectively,
\begin{align}
& \alpha_1 \otimes \cdots \otimes \alpha_{i+1} \longmapsto \sum_j \alpha_1(a^j_1) \cdots \alpha_i(a^j_i) \alpha_{i+1}(\cdot\, a^j_{i+1}), \label{eq:lambda-1-map} \\
& \alpha_1 \otimes \cdots \otimes \alpha_{i+1} \longmapsto \sum_j \alpha_2(a_1^j) \cdots \alpha_{i+1}(a^j_i) \alpha_1(a^j_{i+1}\, \cdot). \label{eq:lambda-2-map}
\end{align}
Again using \eqref{eq:alpha-eq}, one can write the right hand side of \eqref{eq:lambda-2-map} as
\begin{equation} 
\sum_j \alpha_2(\cdot \, a^j_1) \alpha_3(a_2^j) \cdots \alpha_1(a_{i+1}^j),
\end{equation}
which is indeed obtained from \eqref{eq:lambda-1-map} by cyclically permuting the tensor factors in \eqref{eq:tensor-element}. 

\begin{example}
For some $i \geq 1$, consider a graded algebra $A = \bigoplus_{k \in \bZ/(i+1)} \bK e_k \oplus \bK x_k$. We draw it in quiver form as follows:
\begin{equation}
\xymatrix{
\stackrel{e_0}{\bullet} \ar[r]^-{x_1} &
\stackrel{e_1}{\bullet} \ar[r]^-{x_2} & \cdots \ar[r]^-{x_i} & \stackrel{e_i}{\bullet} 
\ar@/^1pc/[lll]^-{x_0}
}
\end{equation}
The vertices give rise to mutually orthogonal idempotents $e_k$, and the $x_k$ satisfy
\begin{equation}
\begin{aligned}
& e_{k+1}x_k e_k = x_k, \\
& x_{k+1}x_k = 0.
\end{aligned}
\end{equation}
The $e_k$ have degree zero, and we choose the degrees of the $x_k$ so that they add up to $i-1$. Following \eqref{eq:tensor-element}, each of the expressions
\begin{equation} \label{eq:cyclic-expressions}
x_k \otimes x_{k-1} \otimes \cdots \otimes x_{k-i}
\end{equation}
gives rise to a bimodule map \eqref{eq:classical-bimodule} of degree $i-1$. 

If one thinks of $A$ (in the straightforward way) as an $A_\infty$-algebra $\scrA$, the elements \eqref{eq:cyclic-expressions} induce $A_\infty$-bimodule maps $(\scrA^\vee)^{\otimes_{\scrA} i} \rightarrow \scrA[i-1]$, which are cyclically exchanged by $K^i$ (up to signs).
It is a familiar fact that $\scrA$ is ``virtually Calabi-Yau of dimension $(i-1)/(i+1)$'', meaning that 
\begin{equation} \label{eq:virtual-cy}
(\scrA^\vee)^{\otimes_{\scrA} i+1} \htp \scrA[i-1]. 
\end{equation}
Hence,
\begin{equation} \label{eq:i-plus-1-morphism}
H^{i-1}(\hom_{[\scrA,\scrA]}((\scrA^\vee)^{\otimes_{\scrA} i},\scrA)) \iso
H^0(\hom_{[\scrA,\scrA]}(\scrA,\scrA^\vee)) \iso \mathit{HH}_0(\scrA,\scrA)^\vee.
\end{equation}
One easily computes that
\begin{equation} \label{eq:am-hochschild}
\mathit{HH}_*(\scrA,\scrA) = \begin{cases} \bK^{i+1} & \ast = 0, \\
\bK & \ast < 0, \\
0 & \ast > 0.
\end{cases}
\end{equation}
The elements of \eqref{eq:i-plus-1-morphism} which we have constructed are a dual basis of $\mathit{HH}_0(\scrA,\scrA)$.
\end{example}

\begin{example}
Start with the graded algebra from the previous example (with $i>1$), but now let $\scrA$ be an $A_\infty$-deformation of it, in which $\mu_{\scrA}^{i+1}(x_k,\dots,x_{k-i})$ is a nonzero multiple of $e_k$ (if this is the case for one $k$, it must be true for all $k$, since these operations are related to each other by the $A_\infty$-associativity equations; in fact, $\scrA$ is unique up to $A_\infty$-isomorphism). The counterpart of \eqref{eq:am-hochschild} is that
\begin{equation}
\mathit{HH}_*(\scrA,\scrA) = \begin{cases} \bK^i & \ast = 0, \\ 0 & \ast \neq 0.
\end{cases}
\end{equation}
This deformation still satisfies \eqref{eq:virtual-cy}, and hence \eqref{eq:i-plus-1-morphism}. The $\bZ/(i+1)$-action on \eqref{eq:i-plus-1-morphism} generated by $K_i$ is the action on $\bK^i$ obtained by taking the standard cyclic permutation representation, and dividing by the one-dimensional diagonal subspace.
\end{example}

\subsection{Noncommutative divisors\label{subsec:div}}
We now have all the ingredients needed to flesh out the outline previously given in Section \ref{subsec:nc-geometry}. 

\begin{definition}
A noncommutative divisor on $\scrA$, with underlying invertible bimodule $\scrP$, is an $A_\infty$-algebra structure $\mu_{\scrB}$ on the graded vector space \eqref{eq:b-space}, with the following properties:
\begin{itemize} \itemsep.5em
\item[(i)] $\scrA \subset \scrB$ is an $A_\infty$-subalgebra.
\item[(ii)] Consider $\scrB$ as an $\scrA$-bimodule (by restriction of the diagonal bimodule $\scrB$ to the subalgebra $\scrA$). Then, the induced $\scrA$-bimodule structure on $\scrB/\scrA = \scrP[1]$ agrees with the previously given one. 
\end{itemize} 
Two noncommutative divisors (with the same underlying bimodule) are isomorphic if there is an isomorphism of the associated $A_\infty$-algebras $\scrB$ which restricts to the identity map on $\scrA$, and induces the identity map on the $\scrA$-bimodule $\scrB/\scrA$.
\end{definition}

\begin{example}
There is a trivial special case, denoted by $\scrB^{\mathit{triv}}$, which is the trivial extension algebra obtained from $\scrA$ and $\scrP$. In that case, the only nonzero components of $\mu_{\scrB}$ are those dictated by (i) and (ii) above.
\end{example}

%
What does this mean concretely? If we restrict $\mu_{\scrB}$ to $T(\scrA[1]) \subset T(\scrB[1])$, it takes values in $\scrA$ and agrees with $\mu_{\scrA}$. 
Next, if we restrict $\mu_{\scrB}$ to $T(\scrA[1])\otimes \scrP[2] \otimes T(\scrA[1]) \subset T(\scrB[1])$, it has the form
\begin{equation} \label{eq:reduces-to-p}
\begin{aligned}
\mu_{\scrB}^{r+s+1}(a'_s,\dots,a'_1,p,a_r,\dots,a_1) = & \; \mu_{\scrP}^{s;1;r}(a'_s,\dots,a'_1;p;a_r,\dots,a_1) \\ & + \theta^{s;1;r}(a'_s,\dots,a'_1;p;a_r,\dots,a_1)
\end{aligned}
\end{equation}
for some
\begin{equation}
\theta: T(\scrA[1]) \otimes \scrP \otimes T(\scrA[1]) \longrightarrow \scrA,
\end{equation}
which is the first new piece of information (not described by $\scrA$ and $\scrP$) in the noncommutative divisor. The $A_\infty$-associativity equation for $\mu_{\scrB}$ implies that $\theta$ is a bimodule map $\scrP \rightarrow \scrA$, meaning that it satisfies \eqref{eq:bimodule-d}. We call its cohomology class
\begin{equation} \label{eq:leading-order-part}
[\theta] \in H^0(\mathit{hom}_{[\scrA,\scrA]}(\scrP,\scrA))
\end{equation}
the section associated to the noncommutative divisor. It is an isomorphism invariant, which classifies the bimodule extension \eqref{eq:a-b-sequence}. For instance, $\scrB^{\mathit{triv}}$ has $\theta = 0$. Hence, \eqref{eq:leading-order-part} is an obstruction (not the only one, in general) to the triviality of a noncommutative divisor.

The next piece of the structure of a noncommutative divisor is obtained by restricting $\mu_{\scrB}$ to $T(\scrA[1]) \otimes \scrP[2] \otimes T(\scrA[1]) \otimes \scrP[2] \otimes T(\scrA[1])$, and then projecting the outcome to $\scrP[3]$. This can be viewed as an element of $\hom_{[\scrA,\scrA]}(\scrP \otimes_{\scrA} \scrP, \scrP)$ of degree $-1$. It is not a cocycle; instead, as shown in \cite[Lemma 3.2]{seidel08}, it provides a homotopy (in the dg category of bimodules) between the two sides of the diagram
\begin{equation} \label{eq:h-commutative}
\xymatrix{
& \scrP \otimes_{\scrA} \scrP
\ar[dl]_-{\mathit{id}_{\scrP} \otimes_{\scrA} \theta} \ar[dr]^-{\theta \otimes_{\scrA} \mathit{id}_{\scrP}}
\\
\scrP \otimes_{\scrA} \scrA
\ar[dr]_{\htp} &&
\scrA \otimes_{\scrA} \scrP
\ar[dl]^{\htp} \\
& \scrP.
}
\end{equation}
In terms of \eqref{eq:c-automorphism}, this means that
\begin{equation} \label{eq:k-fixed}
K_1([\theta]) = [\theta].
\end{equation}

One can put the previous discussion on a more systematic footing, as follows. Consider the Hochschild cochain complex of $\scrB^{\mathit{triv}}$. We recall (from Section \ref{subsec:hochschild-coho}) that this is 
\begin{equation} \label{eq:cc}
\mathit{CC}^*(\scrB^{\mathit{triv}}, \scrB^{\mathit{triv}})[1] = \mathit{Hom}\big(T(\bar\scrB^{\mathit{triv}}[1]),\scrB^{\mathit{triv}}[1]\big),
\end{equation}
with a differential given by $\mu_{\scrB^{\mathit{triv}}}$, which means by $\mu_{\scrA}$ and $\mu_{\scrP}$. Let's equip $\scrB^{\mathit{triv}}$ with an additional grading (called {\em weight grading} to distinguish it from the usual one), in which $\scrA$ has weight $0$ and $\scrP[1]$ has weight $-1$. For any $i > 0$, let $F^i\frakg$ be the subspace of \eqref{eq:cc} consisting of maps which increase the weight by (exactly) $i$. Write $\frakg$ for the (bigraded) space which combines all the $F^i\frakg$. The Hochschild differential maps each $F^i \frakg$ to itself. Moreover, its cohomology forms a long exact sequence
\begin{multline} \label{eq:g-les}
 \cdots \rightarrow H^{*-2i}\big(\hom_{[\scrA,\scrA]}(\scrP^{\otimes_{\scrA} i+1},\scrP)\big) \longrightarrow H^*(F^i\frak g) \\ \longrightarrow
H^{*-2i+1}\big(\hom_{[\scrA,\scrA]}(\scrP^{\otimes_{\scrA} i}, \scrA )\big) \longrightarrow
H^{*-2i+1}\big(\hom_{[\scrA,\scrA]}(\scrP^{\otimes_{\scrA} i+1}, \scrP)\big) \rightarrow \cdots
\end{multline}
where the last map is the difference between $\phi \mapsto \mathit{id}_{\scrP} \otimes_{\scrA} \phi$ and $\phi \mapsto \phi \otimes_{\scrA} \mathit{id}_{\scrP}$. Since either map is an isomorphism, one can identify $H^*(\hom_{[\scrA,\scrA]}(\scrP^{\otimes_{\scrA} i+1},\scrP)) \iso H^*(\hom_{[\scrA,\scrA]}(\scrP^{\otimes_{\scrA} i},\scrA))$, and then \eqref{eq:g-les} becomes
\begin{multline} \label{eq:g-les-2}
 \cdots \rightarrow H^{*-2i}\big(\hom_{[\scrA,\scrA]}(\scrP^{\otimes_{\scrA} i},\scrA)\big) \longrightarrow H^*(F^i\frak g) \\ \longrightarrow
H^{*-2i+1}\big(\hom_{[\scrA,\scrA]}(\scrP^{\otimes_{\scrA} i}, \scrA )\big) \xrightarrow{K_i-\mathit{id}}
H^{*-2i+1}\big(\hom_{[\scrA,\scrA]}(\scrP^{\otimes_{\scrA} i}, \scrA)\big) \rightarrow \cdots
\end{multline}
where $K_i$ is \eqref{eq:c-automorphism}. From this point of view, \eqref{eq:k-fixed} holds because the structure of a noncommutative divisor specifies a lift of $[\theta]$ to $H^1(F^1\frakg)$.

\begin{remark} \label{th:cyclic}
In our applications, $\scrA$ is proper and smooth (hence $\scrA^\vee$ is invertible), and 
\begin{equation} \label{eq:p-dual}
\scrP = \scrA^\vee[-n].
\end{equation}
By Lemma \ref{th:z-action}, the kernel of the boundary map in \eqref{eq:g-les} is $H^*(\mathit{hom}_{[\scrA,\scrA]}((\scrA^\vee)^{\otimes_{\scrA} i}, \scrA))^{\bZ/(i+1)}$. This hints at the existence of a more refined notion, considered in \cite{seidel12b, kontsevich-vlassopoulos13}. Let's suppose for simplicity that $\scrA$ is finite-dimensional. More importantly, we assume that our coefficient field $\bK$ has characteristic $0$. Then, a {\em noncommutative anticanonical divisor} is one with $\scrP$ as in \eqref{eq:p-dual}, and such that the $A_\infty$-structure on $\mu_{\scrB}$ is cyclic with respect to the obvious inner product on $\scrB = \scrA \oplus \scrA^\vee[1-n]$. The counterparts $F^i\frakg^{\mathit{cyc}}$ of the spaces $F^i\frakg$ satisfy
\begin{equation}
H^*(F^i\frakg^{\mathit{cyc}}) \iso H^{*+(n-2)i + 1}\big(\mathit{hom}_{[\scrA,\scrA]}((\scrA^\vee)^{\otimes_{\scrA} i}, \scrA)^{\bZ/(i+1)}\big).
\end{equation}
\end{remark}

We keep the assumption that the coefficient field $\bK$ has characteristic $0$. The complex \eqref{eq:cc} carries the structure of dg Lie algebra, and since the bracket respects weight gradings, we get the structure of a (bigraded) dg Lie algebra on $\frakg$. A noncommutative divisor is given by a solution of the Maurer-Cartan equation \cite{goldman-millson88,manetti04}
\begin{equation} \label{eq:mc-deformation-theory}
\beta \in \frakg^1, \;\; d_{\frakg} \beta + \half [\beta,\beta] = 0
\end{equation}
(the trivial solution $\beta = 0$ corresponds to $\scrB^{\mathit{triv}}$). One can decompose \eqref{eq:mc-deformation-theory} into a series of equations for the weight components $(\beta^1,\beta^2,\dots)$, which are 
\begin{equation} \label{eq:weighted-maurer-cartan}
d_{\frakg} \beta^i + \half \sum_j [\beta^j,\beta^{i-j}] = 0.
\end{equation}
The notion of isomorphism for noncommutative divisors reduces to the usual gauge equivalence for solutions of \eqref{eq:mc-deformation-theory}. Standard obstruction theory yields:

\begin{lemma} \label{th:obstructions}
Suppose that 
\begin{equation} \label{eq:no-negative-degree}
H^*\big(\mathit{hom}_{[\scrA,\scrA]}(\scrP^{\otimes_{\scrA} i}, \scrA)\big) = 0 \quad \text{for all $i\geq 1$ and $\ast<0$.}
\end{equation}
Then, the structure of a noncommutative divisor is determined (up to isomorphism) by \eqref{eq:leading-order-part}, which moreover must satisfy \eqref{eq:k-fixed}. \qed
\end{lemma}

Here is another equivalent formulation, which will be convenient when considering generalizations. Take the one-dimensional vector space $\bK[-1]$, placed in degree $1$ and weight $1$, and considered as a dg Lie algebra with trivial differential and bracket. A solution of \eqref{eq:weighted-maurer-cartan} is the same as an $L_\infty$-homomorphism 
\begin{equation} \label{eq:k-map}
\bK[-1] \longrightarrow \frakg
\end{equation}
which preserves the weight grading (more precisely, the image of $1^{\otimes i}$ under the $i$-linear part of \eqref{eq:k-map} should be divided by $i!$ to get $\beta^i$; this is the usual action of $L_\infty$-homomorphisms on Maurer-Cartan elements \cite[Section 4.3]{kontsevich97}).


Following up on \eqref{eq:list-of-categories-1}, we have:
\begin{equation} \label{eq:curved-d-algebra}
\left\{\!\!\!\!\!\!
\parbox{35em}{
\begin{itemize}\itemsep0.5em
\item[(i)] Of course, $\scrB$ itself is an $A_\infty$-algebra.

\item[(ii)]
Take $\scrB[[u]]$, where the formal variable $u$ has degree $2$. Let 
\[
\scrD = \scrA \oplus \prod_{i \geq 1} u^i\scrB \subset = 
\scrA[[u]] \oplus u(\scrP[1])[[u]] \subset \scrB[[u]]
\]
be the space of $\scrB$-valued formal series in $u$ whose constant term lies in $\scrA$. Following \cite{seidel06}, we make $\scrD$ into a curved $A_\infty$-algebra by extending $\mu_{\scrB}$ $u$-linearly, and then adding a curvature term $\mu^0_{\scrD} = u\, e_{\scrA}$.
Then, $\scrA$ is a (right) $A_\infty$-module over $\scrD$, by pullback under the projection $\scrD \rightarrow \scrA$. One can then define $\scrA \setminus \scrB$ as the endomorphism ring of that module (it is important that we consider $\scrD$ to be defined over $\bK$, but equipped with the $u$-adic topology). Explicitly, the underlying vector space \cite[Equation (4.7)]{seidel08} is
\[
\begin{aligned}
\scrA \setminus \scrB = & \bigoplus_{\substack{k \geq 0 \\ i_1, \dots, i_k > 0}}
\mathit{Hom}\big(\scrA \otimes T(\bar{\scrA}[1]) \otimes u^{i_k}\scrB[1] \otimes T(\bar{\scrA[1]})
\otimes \cdots 
\\[-2em] & \qquad \qquad\qquad \qquad \qquad \qquad \cdots
\otimes u^{i_1}\scrB[1] \otimes T(\bar{\scrA}[1]), \scrA\big).
\end{aligned}
\]
The identification of this construction with a suitable categorical localization (or quotient) is provided by \cite[Theorem 4.1]{seidel08}.
\end{itemize}
}
\right.
\end{equation}

\subsection{Noncommutative pencils\label{subsec:pencil}}
We continue in the same framework as before, including the assumption that $\bK$ has characteristic $0$. Fix a two-dimensional vector space $V$, which will be given weight grading $-1$ (and ordinary grading $0$). Denote the symmetric algebra by $\Sym^*(V)$, and the dual vector space by $W = V^\vee$.

\begin{definition}
A noncommutative pencil on $\scrA$, with underlying bimodule $\scrP$, is a map
\begin{equation} \label{eq:wp}
\wp: T(\scrB[1]) \longrightarrow \scrB[2] \otimes \mathit{Sym}^*(V)
\end{equation}
which preserves both the ordinary grading and the one by weight, and with the following property: for any $w \in W$, we can use the associated evaluation map $\mathit{Sym}^*(V) \longrightarrow \bK$ to specialize \eqref{eq:wp} to $\mu_{\scrB,w}: T(\scrB[1]) \longrightarrow \scrB[2]$; and all these should be noncommutative divisors.
\end{definition}

As before, this definition lends itself to piecewise analysis. The first term not determined by $\scrA$ and $\scrP$ has the form (after dualizing)
\begin{equation} \label{eq:w1}
W \otimes T(\scrA[1]) \otimes \scrP \otimes T(\scrA[1]) \longrightarrow \scrA.
\end{equation}
This is a family of bimodule maps $\scrP \rightarrow \scrA$, depending linearly on $W$. To make things even more concrete, suppose that we identify $V$ and $W$ with $\bK^2$ by choosing dual bases. Then,  \eqref{eq:w1} consists of two bimodule maps $\rho$ and $\sigma$, corresponding to $w = (1,0)$ resp.\ $w = (0,1)$, which we have called the sections associated to the noncommutative pencil. 


Let's return to the bigraded dg Lie algebra $\frakg$. Take $W[-1]$, as a vector space placed in degree $1$ and weight $1$, and equip it with the trivial dg Lie structure. A noncommutative pencil is the same as an $L_\infty$-homomorphism, preserving the weight grading,
\begin{equation}
W[-1] \longrightarrow \frakg.
\end{equation}
There is a standard obstruction theory for such homomophisms, which in particular yields an analogue of Lemma \ref{th:obstructions}:

\begin{lemma} \label{th:obstructions-2}
If \eqref{eq:no-negative-degree} holds, the structure of a noncommutative pencil is determined up to isomorphism by the cohomology classes $[\rho],\, [\sigma]$ (each of which must be fixed by $K_1$). \qed
\end{lemma}

By definition, a noncommutative pencil yields a noncommutative divisor for each $w \in W$. For $w = 0$, this reduces to the trivial extension algebra $\scrB^{\mathit{triv}}$. Assuming as before that $W = \bK^2$, we write $\scrB_\infty$ for the noncommutative divisor associated to $w = (1,0)$, and $\scrB_z$ for that associated to $w = (z,1)$. The associated sections are those given in \eqref{eq:list-of-categories-2}. This essentially exhausts all the possibilities, since the noncommutative divisors associated to $w$ and $\lambda w$, for $\lambda \in \bK^\times$, are related by an automorphism of $\scrB$ (which acts trivially on $\scrA$ and rescales $\scrP$ by $\lambda$). 

Let's spell out the other structures mentioned in \eqref{eq:list-of-categories-2}:
\begin{equation}
\left\{\!\!\!\!\!\!
\parbox{35em}{
\begin{itemize} \itemsep.5em
\item[(i)] Set $w = (1,q)$, where $q$ is a formal variable. The corresponding specialization of \eqref{eq:wp} is a map $T(\scrB[1]) \longrightarrow (\scrB[2])[[q]]$. This defines a $\bC[[q]]$-linear $A_\infty$-structure on $\scrB[[q]]$, which is a deformation of $\scrB_\infty$. This is what we denoted by $\hat\scrB_{\infty}$.

\item[(ii)] Using the basis $r = (1,0)$ and $s = (0,1)$ of $V$, write $\wp^d = \sum_{i,j} \wp^d_{i,j} \otimes r^i s^j$. Because $\wp^d_{i,j}$ increases weights by $i+j$, it is nontrivial only if $d \geq i+j$. Moreover, if we extend it $u$-linearly to a map $\scrD^{\otimes d} \rightarrow \scrD$, then it is divisible by $u^{i+j}$. Hence, the expression
\[
\sum_{i,j} \wp^d_{i,j} u^{-j} h^j
\]
makes sense, as a family of maps $\scrD^{\otimes d} \rightarrow \scrD$ depending on an auxiliary formal variable $h$ of degree $2$. We add to it a curvature term as in \eqref{eq:curved-d-algebra}, and get a formal deformation of our previous curved $A_\infty$-structure on $\scrD$, with parameter $h$.
This induces a deformation of $\scrA \setminus \scrB_\infty$ (with a curvature term, which vanishes if we set $h = 0$). We write it as $\scrL\scrG$.
\end{itemize}
}\right.
\end{equation}

\begin{remark} \label{th:graph}
(i) A pencil in ordinary algebraic geometry has a (compactified) graph, see e.g.\ \cite[p.~154]{tibar07}. Here is one possible approach towards defining its noncommutative analogue: over the projective line $\mathbb{P}^1 = \mathbb{P}(W)$, consider the graded quasi-coherent sheaf 
\begin{equation}
{\mathcal A} = \scrA \otimes {\mathcal O}_{\mathbb P^1} \oplus \scrP[1] \otimes {\mathcal O}_{\mathbb P^1}(-1).
\end{equation}
Then, \eqref{eq:wp} makes this into a sheaf of $A_\infty$-algebras. Using e.g.\ Cech resolutions, one can associate to this a single $A_\infty$-algebra, whose underlying cohomology is $H^*(\mathbb{P}^1, {\mathcal A} \otimes \mathit{End}({\mathcal E}))$ for some generator ${\mathcal E}$ of the derived category of $\mathbb{P}^1$, for instance ${\mathcal E} = {\mathcal O}_{\mathbb P^1} \oplus {\mathcal O}_{\mathbb P^1}(-1)$.
%

(ii)
A pencil in classical algebraic geometry also has a base locus (or axis). A first idea for its noncommutative analogue is
\begin{equation} \label{eq:katz1}
\scrB_0 \otimes_{\scrA} \scrB_\infty, 
\end{equation}
but this is not a priori an $A_\infty$-algebra (only a bimodule over $\scrA$, and as such, does not use the full structure of the noncommutative pencil). However, even at this level, the notion has geometric significance: we conjecture that for, the noncommutative pencils obtained from anticanonical Lefschetz pencils as in Conjecture \ref{th:conjecture}, \eqref{eq:katz1} is trivial, meaning quasi-isomorphic to zero (this property would distinguish them from Lefschetz fibrations with singular fibre at $\infty$).
%

In classical algebraic geometry, (i) and (ii) are related, since (under suitable smoothness assumptions) the graph is obtained by blowing up the base locus. Noncommutative counterparts of this relation would be of interest, because they arise when considering the relation between the Fukaya category of an ample hypersurface and the wrapped Fukaya category of its complement. Concretely, in the case where \eqref{eq:katz1} is trivial, the Fukaya category of the ample hypersurface should appear as a formal completion of the wrapped Fukaya category of the complement.
\end{remark}
%
%

\subsection{Discussion of the assumptions}
Throughout this section, we have required that all $A_\infty$-structures should be $\bZ$-graded. One can actually work with $\bZ/2$-gradings throughout, with the same results, except for Lemmas \ref{th:obstructions} and \ref{th:obstructions-2}.

\begin{remark}
This is not quite the end of this topic, as illustrated by the intermediate situation mentioned in Remark \ref{th:fractional-cy-2}. There, $\scrA$ and $\scrP$ are $\bZ$-graded, and so is $\wp$ provided that we take $V = \bK \oplus \bK[2m]$. The associated sections $\rho$ and $\sigma$ have degrees $0$ and $2m$, respectively. While $\scrB_\infty$ is still $\bZ$-graded, the other $\scrB_z$, $z \in \bK$, are only $\bZ/2$-graded. The original $\bZ$-grading leads to an isomorphism $\scrB_z \rightarrow \scrB_{\lambda^{2m}z}$, for any $\lambda \in \bK^\times$. Similarly, the formal deformation $\hat\scrB_\infty$ is $\bZ$-graded if we assign degree $-2m$ to the deformation variable $q$.
\end{remark}

The other assumption we have imposed, starting with \eqref{eq:mc-deformation-theory}, is that the coefficient field $\bK$ has characteristic $0$. This can be dropped as well, at the cost of making some formulations a little more complicated. Maurer-Cartan theory, which means the classification theory of solutions of \eqref{eq:mc-deformation-theory}, does not apply in positive characteristic. However, a simpler obstruction theory argument suffices to prove Lemma \ref{th:obstructions} (and its variant, Lemma \ref{th:obstructions-2}), and that does not require any assumption on $\bK$. The second case where we have used the assumption on $\bK$ is in the definition of noncommutative pencil, which was formulated by specializing two-variable polynomials, appearing as elements of $\mathit{Sym}(V)$, to one-dimensional linear subspaces. In positive characteristic, such specializations may fail to recover the original polynomial ($f(x,y) = x^py - xy^p$ vanishes on every line, for $\bK = \bF_p$). Hence, the definition of noncommutative pencil would have to be reformulated as a condition on \eqref{eq:wp} itself.

Finally, there is a variation on our general setup which will be important for applications. Namely, one can replace the ground field $\bK$ by the semisimple ring 
\begin{equation} \label{eq:semisimple}
R = \bK^m = \bK e_1 \oplus \cdots \oplus \bK e_m. 
\end{equation}
An $A_\infty$-algebra over $R$ consists of a graded $R$-bimodule $\scrA$, together with $R$-bimodule maps 
\begin{equation} \label{eq:r-bilinear}
\mu_{\scrA}^d: \scrA \otimes_R \scrA \otimes_R \cdots \otimes_R \scrA \longrightarrow \scrA[2\!-\!d].
\end{equation}
This is in fact the same as an $A_\infty$-category with objects labeled by $\{1,\dots,m\}$. The strict unit $e_{\scrA}$ must lie in the diagonal part of the $R$-bimodule $\scrA$, which means that it can be written as $e_{\scrA} = e_{\scrA,1} + \cdots + e_{\scrA,m}$ with $e_{\scrA,i} \in e_i \scrA e_i$. Similarly, an $\scrA$-bimodule consists of a graded $R$-bimodule $\scrP$, with structure maps that are linear over $R$ as in \eqref{eq:r-bilinear}. The same applies to the definition of morphisms in $[\scrA,\scrA]$ and of Hochschild homology. All definitions and results in this section then carry over without any other significant modifications.

\section{Hamiltonian Floer cohomology}
This section recalls the most familiar version of Floer cohomology \cite{floer88}, adapted to Liouville domains following \cite{viterbo97a}. The only non-standard aspect of the exposition is our extensive use of the BV (Batalin-Vilkovisky or loop rotation) operator \cite{seidel07, seidel-solomon10, bourgeois-oancea09b}. As a specific example, we will consider Liouville domains whose boundary is a contact circle bundle. Their Floer cohomology is very well understood in relation to Symplectic Field Theory, thanks to \cite{bourgeois-oancea09, diogo12} and ongoing work of Diogo-Lisi; but our limited needs can be met by a more elementary approach.

\subsection{Basic definitions}
We will work in the following setting, which benefits from strong exactness assumptions:

\begin{setup} \label{th:hamiltonian-setup}
(i) Let $M^{2n}$ be a Liouville domain. This means that it is a compact manifold with boundary, equipped with an exact symplectic structure $\omega_M = d\theta_M$, such that the Liouville vector field $Z_M$ dual to $\theta_M$ points strictly outwards along the boundary. The Liouville flow provides a canonical collar embedding $(-\infty,0] \times \partial M \hookrightarrow M$. The exponential of the time variable then gives a function $\rho_M$ (defined on the image of the embedding in $M$) such that
\begin{equation} \label{eq:h-function}
\left\{\begin{aligned}
& \rho_M|\partial M = 1, \\
& Z_M.\rho_M = \rho_M. 
\end{aligned}
\right.
\end{equation}
We denote by $R_M$ the Hamiltonian vector field of $\rho_M$. This is the natural extension of the Reeb vector field $R_{\partial M}$ associated to the contact one-form $\theta_M|\partial M$. 

(ii) During part of our argument, we will additionally assume that $M$ comes with a symplectic Calabi-Yau structure (a distinguished homotopy class of trivializations of its canonical bundle $K_M = \Lambda^n_{\bC}(TM)^\vee$, for some compatible almost complex structure).

(iii) We consider functions $H \in \smooth(M,\bR)$ with
\begin{equation} \label{eq:hamiltonian}
H = \epsilon \rho_M \quad \text{near $\partial M$}
\end{equation}
for some $\epsilon \in \bR$. Throughout, it is assumed that
\begin{equation} \label{eq:no-reeb}
\text{$\epsilon R_{\partial M}$ has no $1$-periodic orbits;}
\end{equation}
equivalently, $\epsilon \neq 0$ and there are no Reeb orbits on $\partial M$ whose period is $|\epsilon|/k$, for any positive integer $k$.

(iii) We use compatible almost complex structures $J$ such that
\begin{equation} \label{eq:j-convex}
\theta_M \circ J = d\rho_M \quad \text{near $\partial M$;}
\end{equation}
equivalently, $J Z_M = R_M$.
\end{setup}

\begin{example} \label{th:contact-circle-bundle}
(i) The most important case for us is when $\partial M$ is a {\em contact circle bundle}, by which we mean that the Reeb flow is a free $S^1$-action (for concreteness, we assume that this has period $1$). In that case, $M$ can be embedded into a closed symplectic manifold $M^{\mathit{cl}}$ by attaching a disc bundle over $\partial M/S^1$. The zero-section $\partial M/S^1 \hookrightarrow M^{\mathit{cl}}$ represents a positive multiple of the symplectic class $[\omega_{M^{\mathit{cl}}}] \in H^2(M^{\mathit{cl}};\bR)$.

(ii) In the same situation as before, suppose additionally that 
\begin{equation}
c_1(M^{\mathit{cl}}) = m\, [\partial M/S^1] \in H^2(M^{\mathit{cl}};\bZ), \quad \text{for some $m \in \bZ$.}
\end{equation}
For $m = 0$, this means that $K_{M^{\mathit{cl}}}$ is trivial; one can then choose a symplectic Calabi-Yau structure on $M^{\mathit{cl}}$, which then obviously restricts to one on $M$. For general $m$, one can still find a section of $K_{M^{\mathit{cl}}}$ which has an order $m$ ``pole'' along $\partial M/S^1$, and is nonzero elsewhere. By restricting that to $M$, one again obtains a symplectic Calabi-Yau structure.
\end{example}

In its most basic form, the Floer cohomology $\mathit{HF}^*(M,\epsilon)$ is a $\bZ/2$-graded vector space over a field $\bK$ of characteristic $2$, which depends on $M$ and a choice of constant $\epsilon$ satisfying \eqref{eq:no-reeb}. To define it, take a family of functions $H = (H_t)$ as in \eqref{eq:hamiltonian}, as well as a family of almost complex structures $J = (J_t)$ as in \eqref{eq:j-convex}, both parametrized by $t \in S^1 = \bR/\bZ$. On the free loop space $\scrL = \smooth(S^1, M)$, consider the $H$-perturbed action functional:
\begin{equation} \label{eq:action}
A_H(x) = \int_{S^1} -x^*\theta_M + H_t(x(t)) \, dt.
\end{equation}
Its critical points are solutions $x \in \scrL$ of
\begin{equation} \label{eq:periodic-y}
\dot{x} = X_{H,t},
\end{equation}
where $X_H$ is the time-dependent Hamiltonian vector field of $H$. Such $x$ correspond bijectively to fixed points $x(1)$ of $\phi_H^1$, where $(\phi_H^t)$ is the Hamiltonian isotopy generated by $H$. By assumption on $\epsilon$, all such $x$ are contained in the interior of $M$. For a generic choice of $H$, all $x$ are nondegenerate; assume from now on that this is the case. One defines the Floer cochain space $\mathit{CF}^*(M,H)$ to be the $\bZ/2$-graded vector space with one summand $\bK$ for each $x$ (the $\bZ/2$-grading encodes the local Lefschetz number). Consider solutions of Floer's gradient flow equation for \eqref{eq:action}. These ``Floer trajectories'' are maps $u: \bR \times S^1 \rightarrow M$, satisfying
\begin{equation} \label{eq:floer}
\partial_s u + J_t (\partial_t u - X_{H,t}) = 0, 
\end{equation}
with limits 
\begin{equation} \label{eq:floer-limits}
\textstyle \lim_{s \rightarrow \pm\infty} u(s,\cdot) = x_{\pm}
\end{equation}
as in \eqref{eq:periodic-y}. While $M$ is not closed, a standard maximum principle argument shows that Floer trajectories can never reach $\partial M$. We consider the moduli space of solutions to \eqref{eq:floer}, \eqref{eq:floer-limits} (more precisely, solutions that are not constant in $s$, modulo translation in that direction). For generic choice of $J$, these moduli spaces are regular. One counts isolated points in them to obtain the coefficients of the differential $d$ on $\mathit{CF}^*(M,H)$, of which $\mathit{HF}^*(M,\epsilon)$ is the cohomology.

The free loop space admits an obvious circle action (loop rotation). In classical cohomology, a circle action gives rise to an operator of degree $-1$, dual to moving cycles around orbits. The analogue for Floer cohomology is the BV operator
\begin{equation} \label{eq:bv}
\Delta: \mathit{HF}^*(M,\epsilon) \longrightarrow \mathit{HF}^{*-1}(M,\epsilon).
\end{equation}
To define it, suppose that we have fixed $(H,J)$ as necessary to define the Floer cochain complex. The rotated versions, for $r \in S^1$, are
\begin{equation}
(H^{(r)}_t,J^{(r)}_t) = (H_{t-r},J_{t-r}).
\end{equation} 
Similarly, if $x$ is as in \eqref{eq:periodic-y}, we write $x^{(r)}(t) = x(t-r)$. Choose a family of functions and almost complex structures $(H_{r,s,t}, J_{r,s,t})$, depending on $(r,s,t) \in S^1 \times \bR \times S^1$, which lie in the same class as before and satisfy
\begin{equation} \label{eq:bv-data}
(H_{r,s,t}, J_{r,s,t}) = \begin{cases} (H_t^{(r)}, J_t^{(r)}) & s \leq -1, \\
(H_t,J_t) & s \geq 1.
\end{cases}
\end{equation}
One considers the parametrized moduli space of pairs $(r,u)$, consisting of $r \in S^1$ and a solution $u: \bR \times S^1 \rightarrow M$ of the $r$-dependent equation
\begin{equation} \label{eq:bv-equation}
\partial_s u + J_{r,s,t} (\partial_t u - X_{H,r,s,t}) = 0,
\end{equation}
with limits
\begin{equation} \label{eq:bv-limits}
\left\{
\begin{aligned}
& \textstyle \lim_{s \rightarrow -\infty} \; u(s,\cdot) = x_-^{(r)}, \\
& \textstyle \lim_{s \rightarrow +\infty} \, u(s,\cdot) = x_+.
\end{aligned}
\right.
\end{equation}
The same maximum principle argument as before applies to solutions of \eqref{eq:bv-equation}, \eqref{eq:bv-limits}, ensuring that they can never reach $\partial M$. For suitably generic choices, counting isolated points in the parametrized moduli space yields a degree $-1$ chain map 
\begin{equation} \label{eq:bv-chain}
\delta: \mathit{CF}^*(M,H) \longrightarrow \mathit{CF}^{*-1}(M,H),
\end{equation}
which induces \eqref{eq:bv} on cohomology. Unlike the differential $d$, the operator $\delta$ is not always compatible with the action filtration. More precisely, given a solution of \eqref{eq:bv-equation}, \eqref{eq:bv-limits}, one has
\begin{equation} \label{eq:energy-cont}
E(u) = \int_{\bR \times S^1} |\partial_s u|^2\,  \mathit{ds} \wedge \mathit{dt} \leq
A_H(x_-) - A_H(x_+) + \int_{\bR \times S^1} \|\partial_s H_{r,s,t}\|_{\infty} \, \mathit{ds} \wedge \mathit{dt}.
\end{equation}
If $H$ was autonomous (independent of $t$) one could choose $H_{r,s,t} = H$ for all $(r,s,t)$, and then the last term in \eqref{eq:energy-cont} would vanish. In general, this interferes with the nondegeneracy of $1$-periodic orbits; but one can always choose $H$ to be a small perturbation of an autonomous function, and thereby make that term as small as needed. For such choices, $\delta$ will decrease actions only by a small amount.

\begin{remark} \label{th:signs-and-grading}
By counting Floer trajectories with appropriate signs, one can define Floer cohomology over an arbitrary coefficient field, as already pointed out in \cite{floer88}. A symplectic Calabi-Yau structure on $M$ determines a lift of the $\bZ/2$-grading on $\mathit{HF}^*(M,\epsilon)$ to a $\bZ$-grading \cite{salamon-zehnder92}. Since these are standard additions to the basic construction, we will continue our exposition in the simplest setup ($\mathrm{char}(\bK) = 2$ and $\bZ/2$-gradings), but take care that all formulae remain correct when signs and $\bZ$-gradings are added. For instance, in the $\bZ$-graded case, the BV operator has degree $-1$, as indicated in \eqref{eq:bv}.
\end{remark}

Floer cohomology has a Poincar{\'e} type duality, obtained by reversing loops in $\scrL$:
\begin{equation} \label{eq:id-poincare-duality}
\mathit{HF}^*(M,-\epsilon) \iso \mathit{HF}^{2n-*}(M,\epsilon)^\vee.
\end{equation}
This comes from an isomorphism of the underlying chain complexes, assuming that the choices of $(H,J)$ have been suitably coordinated. For equally straightforward reasons, it is compatible with BV operators.

As is implicit in the notation, Floer cohomology is independent (up to canonical isomorphism) of all the auxiliary choices made in its construction, and the same holds for the BV operator. It is maybe useful to elaborate on this slightly. Suppose that we have two possible choices, giving rise to Floer complexes $\mathit{CF}^*(M,H)$ and $\mathit{CF}^*(M,\tilde{H}$), with their differentials $d$, $\tilde{d}$ and chain level BV operators $\delta$, $\tilde{\delta}$. The continuation map method \cite{salamon-zehnder92} provides maps
\begin{equation} \label{eq:cont-maps}
\begin{aligned}
& k: \mathit{CF}^*(M,H) \longrightarrow \mathit{CF}^*(M,\tilde{H}), \\
& \kappa: \mathit{CF}^*(M,H) \longrightarrow \mathit{CF}^{*-2}(M,\tilde{H}).
\end{aligned}
\end{equation}
The first one is a chain map (and actually a quasi-isomorphism), and the second one satisfies
\begin{equation} \label{eq:compatibility-with-delta}
\tilde{d} \kappa - \kappa d = \tilde{\delta} k -  k \delta.
\end{equation}
Moreover, the maps \eqref{eq:cont-maps} are themselves unique in a suitable homotopical sense.
%

\subsection{Reeb orbits}
We now consider the dependence on $\epsilon$. This usually appears as part of the construction of symplectic cohomology (in the sense of \cite{viterbo97a}).

\begin{lemma} \label{th:no-wall}
Suppose that all $\epsilon \in [\epsilon_-,\epsilon_+]$ satisfy \eqref{eq:no-reeb}. Then we have an isomorphism, compatible with BV operators,
\begin{equation}
\mathit{HF}^*(M,\epsilon_-) \iso \mathit{HF}^*(M,\epsilon_+).
\end{equation}
\end{lemma}

\begin{proof}
Even though that is not absolutely necessary, we find it convenient to introduce finite enlargements of our Liouville domain. Such an enlargement, by an amount $C>1$, is 
\begin{equation} \label{eq:attach-cone}
\hat{M} = M \cup_{\partial M} ([1,C] \times \partial M).
\end{equation}
The conical part which we have added carries the one-form $\theta_{\hat{M}} = r (\theta_M|\partial M)$ ($r$ being the coordinate in $[1,C]$). One extends $\rho_M$ to $\hat{M}$ by setting $\rho_{\hat{M}}(r,x) = r$. Choose a function
\begin{equation} \label{eq:turn}
\left\{
\begin{aligned}
& h_+: (0,1] \longrightarrow \bR, \\
& h_+(a) = \epsilon_- a \quad \text{for $a$ sufficiently small,} \\
& h_+'(a) = \epsilon_+ \quad \text{for $a$ close to $1$,} \\
& h_+''(a) \geq 0 \quad \text{everywhere}, \\
& h_+''(a) > 0 \quad \text{for all $a$ such that $\epsilon_- < h_+'(a) < \epsilon_+$.}
\end{aligned}
\right.
\end{equation} 
Given the $H_-$ used to define $\mathit{HF}^*(M,\epsilon_-)$, we extend it to a time-dependent function $\hat{H}_+$ on $\hat{M}$ by setting
\begin{equation} \label{eq:extend-hamiltonian}
\hat{H}_{+,t}(r,x) =C\, h_+(C^{-1}r). 
\end{equation}
This makes sense provided that $C$ is large (so that $h_+(a) = \epsilon_- a$ near $a = C^{-1}$). This extension does not quite belong to the class \eqref{eq:hamiltonian} for the manifold $\hat{M}$ and constant $\epsilon_+$, but that could be remedied by adding a constant, which has no effect on our construction. Hence, $\hat{H}_+$ can be used to define $\mathit{HF}^*(\hat{M},\epsilon_+)$.

Note that on $[1,C] \times \partial M$,
\begin{equation}
X_{\hat{H}_+,t} = h'_+(C^{-1}r)\, R_{\partial M}.
\end{equation}
Because of the assumption, it follows that all $1$-periodic orbits of \eqref{eq:extend-hamiltonian} are actually contained in $M$. A general almost complex structure on $M$ does not naturally extend to $\hat{M}$, so there we proceed in the other direction: we choose $\hat{J} = (\hat{J}_t)$ on $\hat{M}$ which satisfy the analogue of \eqref{eq:j-convex} in a neighbourhood of $[1,C] \times \partial M$, and then take $J = (J_t)$ to be the restriction to $M$. A maximum principle argument shows that Floer trajectories in $\hat{M}$ are in fact all contained in $M$. As a consequence, we get an isomorphism (of complexes, and hence of Floer cohomology groups)
\begin{equation}
\mathit{HF}^*(\hat{M},\epsilon_+) \iso \mathit{HF}^*(M,\epsilon_-).
\end{equation}
For similar reasons, this is compatible with the BV operators. A repetition of the same argument, with a linear function instead of $h_+$, shows that
\begin{equation}
\mathit{HF}^*(\hat{M},\epsilon_+) \iso \mathit{HF}^*(M,\epsilon_+).
\end{equation}
Combining the two isomorphisms yields the desired result.
\end{proof}

\begin{lemma} \label{th:wall}
Take $\epsilon_- < \epsilon_+$, both of which satisfy \eqref{eq:no-reeb}, and suppose that there is exactly one $\epsilon \in (\epsilon_-,\epsilon_+)$ for which \eqref{eq:no-reeb} fails. Then there is a long exact sequence
\begin{equation} \label{eq:les-f}
\cdots \rightarrow \mathit{HF}^*(M,\epsilon_-) \longrightarrow \mathit{HF}^*(M,\epsilon_+) \longrightarrow H^*(Q) \rightarrow \cdots
\end{equation}
where $H^*(Q)$ depends only on the local geometry near the $1$-periodic orbits of $\epsilon R_{\partial M}$. Moreover, $H^*(Q)$ carries an endomorphism of degree $-1$, with the same locality property, and which fits in with \eqref{eq:les-f} and the BV operators on $\mathit{HF}^*(M,\epsilon_{\pm})$.
\end{lemma}

\begin{proof}
Let's use the same extension $\hat{H}_+$ as in the proof of Lemma \ref{th:no-wall}. This time, there are additional $1$-periodic orbits in the conical part. Their actions are
\begin{equation} \label{eq:h-hat-x-action}
C (h_+(a_*)-ah'_+(a_*)), \quad \text{for the unique $a_*$ such that $h'_+(a_*) = \epsilon$.}
\end{equation}
The convexity properties of $h_+$ ensure that $h_+(a_*) - a_*h'_+(a_*) < 0$. Therefore, \eqref{eq:h-hat-x-action} goes to $-\infty$ as we make $C$ large. In particular, we can choose $C$ so that \eqref{eq:h-hat-x-action} is smaller than the actions of any $1$-periodic orbits of $H_-$. The $1$-periodic orbits lying in the conical part will be degenerate, but one can remedy that by a small time-dependent perturbation of $\hat{H}_+$ (concentrated near those orbits), and the statement about actions will continue to hold. Hence, the $1$-periodic orbits lying in $M$ form a subcomplex of $\mathit{CF}^*(\hat{M},\hat{H}_+)$. Using the integrated maximum principle (see \cite{abouzaid-seidel07}), one can show that the differential on this subcomplex agrees with the Floer differential in $M$ (alternatively, one can arrive at the same conclusion by taking $C$ very large, and using the Monotonicity Lemma; or by the standard maximum principle and a Gromov compactness argument, applied to the limit where the perturbation of $\hat{H}_+$ goes to $0$). In other words, one has an inclusion of chain complexes
\begin{equation} \label{eq:small-h-complex}
\mathit{CF}^*(M,H_-) \subset \mathit{CF}^*(\hat{M},\hat{H}_+).
\end{equation}
Defining $Q^*$ to be the quotient, one obviously gets a long exact sequence
\begin{equation} \label{eq:reeb-les}
\cdots \rightarrow \mathit{HF}^*(M,\epsilon_-) \longrightarrow \mathit{HF}^*(\hat{M},\epsilon_+) \longrightarrow H^*(Q) \rightarrow \cdots
\end{equation}
In the original situation \eqref{eq:h-hat-x-action}, all $1$-periodic orbits in the conical part had the same action. The perturbation will destroy that, but only by a small amount. Hence, $H^*(Q)$ is still a local Floer cohomology group, in the sense of \cite{pozniak, ginzburg06b}. As in Lemma \ref{th:no-wall}, one can replace $\mathit{HF}^*(\hat{M},\epsilon_+)$ by $\mathit{HF}^*(M,\epsilon_+)$ in \eqref{eq:reeb-les}, which yields \eqref{eq:les-f}.

The corresponding statements concerning BV operators are slightly more tricky. Let's choose all our Hamiltonians to be small perturbations of autonomous ones. In that case, it follows from \eqref{eq:h-hat-x-action} and \eqref{eq:energy-cont} that the chain level BV map on $\mathit{CF}^*(\hat{M},\hat{H}_+)$ will preserve the subcomplex \eqref{eq:small-h-complex}. A Gromov compactness argument (where one decreases the size of the perturbations, as briefly mentioned before) implies that the induced endomorphism of $Q^*$ is again local in nature. To show that its restriction to $\mathit{CF}^*(M,H_-)$ counts only solutions lying in $M$, one combines the same idea of Gromov compactness with the methods we've used for the differential (integrated maximum principle, or Monotonicity Lemma).
\end{proof}

\begin{lemma} \label{th:wall-2}
In the situation of Lemma \ref{th:wall}, assume additionally that the $1$-periodic orbits of $\epsilon R_{\partial M}$ form a Morse-Bott nondegenerate connected manifold $F$. Then,
\begin{equation}
H^*(Q) \iso H^{*+k}(F;\xi),
\end{equation}
for some $k$ and local $\bK$-coefficient system $\xi \rightarrow F$ (with holonomy $\pm 1$). That coefficient system is equivariant with respect to the $S^1$-action given by loop rotation on $F$. Hence, $H^*(F;\xi)$ comes with an endomorphism of degree $-1$, which fits in with the BV operators and \eqref{eq:les-f}.
\end{lemma}

The connectedness assumption is for notational convenience only (otherwise, each component of $F$ comes with its own index offset $k$, and one has to take the direct sum of their contributions).

\begin{example}
Suppose that $\epsilon R_{\partial M}$ has exactly one $1$-periodic orbit (up to loop rotation), which moreover is transversally nondegenerate. This means that $F \iso S^1$. The local coefficient system $\xi$ is trivial if the orbit is ``good'', and has holonomy $-1$ if it is ``bad'' (in the terminology from Symplectic Field Theory; for an explanation in a framework close to ours, see \cite[\S 5]{bourgeois-mohnke04} or \cite[\S 4.4]{bourgeois-oancea09c}).
\end{example}

Lemma \ref{th:wall-2} can be proved by using the Morse-Bott formalism for Floer cohomology \cite{austin-braam95, bourgeois02}. While we will not give a complete proof, it makes sense to describe that formalism in the appropriate form, since we'll return to it later on.

We start with a given $H_-$ and extend that to $\hat{H}_+$ as in \eqref{eq:extend-hamiltonian}, but do not perturb this extension further. As usual, $C$ is assumed to be large, to get the necessary action inequalities. Fix a Morse-Smale pair on $F$, consisting of a Morse function $f$ and a metric $g$. The Morse-Bott-Floer complex $\mathit{CF}^*(\hat{M},\hat{H}_+,f)$ has two kinds of generators: ones coming from the $1$-periodic orbits in $M$, and ones coming from the critical points of $f$ (the constant $k$ and local coefficient system $\xi$ appear when we take gradings and signs into account). Let's write this as
\begin{equation} \label{eq:morse-bott-c}
\mathit{CF}^*(\hat{M},\hat{H}_+,f) = \mathit{CF}^*(M,H_-) \oplus \mathit{CM}^{*+k}(f;\xi). 
\end{equation}
The differential on \eqref{eq:morse-bott-c} is of the form
\begin{equation} \label{eq:morse-bott-d}
d = \begin{pmatrix} 
d_{\mathit{Floer}} & d_{\mathit{mixed}} \\
0 & d_{\mathit{Morse}}
\end{pmatrix},
\end{equation}
where the various parts are defined as follows:
\begin{itemize} \itemsep.5em
\item[(i)] If $x_-$ and $x_+$ are $1$-periodic orbits in $M$, one counts Floer trajectories as before, and this yields $d_{\mathit{Floer}}$. 

\item[(ii)] If $x_-$ and $x_+$ are both critical points of $f$, one counts trajectories of $-\nabla_g f$ going from $x_-$ to $x_+$, which yields $d_{\mathit{Morse}}$.

\item[(iii)] Suppose that $x_-$ lies in $M$, and $x_+$ is a critical point of $f$. In that case, one considers a mixed moduli space consisting of pairs $(u,v)$, where $u$ is a Floer trajectory, and $v: [0,\infty) \rightarrow F$ is a positive half flow line of $-\nabla_g f$, with
\begin{equation} 
\left\{
\begin{aligned}
& \textstyle
\lim_{s \rightarrow -\infty} u(s,\cdot) = x_-, \\ & \textstyle
\lim_{s \rightarrow +\infty} u(s,\cdot) = v(0), \\ & \textstyle
\lim_{s \rightarrow +\infty} v(s) = x_+. 
\end{aligned}
\right.
\label{eq:matching-conditions}
\end{equation}
Counting solutions of this equation yields $d_{\mathit{mixed}}$. 
\end{itemize}

Assume that the original $H_-$ was a small perturbation of an autonomous Hamiltonian. We then choose a family of Hamiltonians $H_{-,r,s,t}$ on $M$ as in \eqref{eq:bv-data}, which are perturbations of the same autonomous Hamiltonian, and extend them to $\hat{H}_{+,r,s,t}$ on $\hat{M}$ as in \eqref{eq:extend-hamiltonian}. The main property, established by similar arguments as in the proof of Lemma \ref{th:wall}, is that any solution of \eqref{eq:bv-equation} whose left-hand limit $x_-^{(r)}$ lies in $F$ must be stationary, which means $u(s,t) = x_-^{(r)}(t) = x_+(t)$. Then, the Morse-Bott version of \eqref{eq:bv-chain} takes on a form parallel to \eqref{eq:morse-bott-d}:
\begin{equation} \label{eq:morse-bott-delta}
\delta = \begin{pmatrix}
\delta_{\mathit{Floer}} & \delta_{\mathit{left}} + \delta_{\mathit{right}} \\
0 & \delta_{\mathit{Morse}}
\end{pmatrix}.
\end{equation}
The definition is a more elaborate version of the previous one:
\begin{itemize} \itemsep.5em
\item[(i)] $\delta_{\mathit{Floer}}$ is defined by considering solutions of \eqref{eq:bv-equation} whose limits both lie in $M$.

\item[(ii)]
Loop rotation restricts to an $S^1$-action on $F$. Denote by $f^{(r)}$ and $g^{(r)}$ the images of $f$ and $g$ under that action, for time $r$. Choose generic families $f_{r,s}$ and $g_{r,s}$, depending on $(r,s) \in S^1 \times \bR$, which satisfy the analogue of \eqref{eq:bv-data}:
\begin{equation}
(f_{r,s},g_{r,s}) = \begin{cases} (f^{(r)},g^{(r)}) & s \leq -1, \\ (f,g) & s \geq 1. \end{cases}
\end{equation}
The counterpart of \eqref{eq:bv-equation}, \eqref{eq:bv-limits} is an equation for $r \in S^1$ and $v: \bR \rightarrow F$:
\begin{equation} \label{eq:morse-bv}
\left\{ 
\begin{aligned}
& \partial_s v + \nabla_{g_{r,s}} f_{r,s} = 0, \\
& \textstyle \lim_{s \rightarrow -\infty} v(s) = x_-^{(r)}, \\
& \textstyle \lim_{s \rightarrow +\infty} v(s) = x_+,
\end{aligned}
\right.
\end{equation}
where $x_\pm$ are critical points of $f$. By counting isolated solutions of this equation, one defines $\delta_{\mathit{Morse}}$.

\item[(iii)]
Suppose that $x_-$ lies in $M$, and $x_+$ is a critical point of $f$. Consider triples $(r,u,v)$, where $(r,u)$ is a solution of the same parametrized equation as in (i), and $v$ is a half flow line of $-\nabla_g f$, with asymptotic conditions
\begin{equation}
\left\{\begin{aligned}
& \textstyle 
\lim_{s \rightarrow -\infty} u(s,\cdot) = x_-^{(r)}, \\
& \textstyle 
\lim_{s \rightarrow +\infty} u(s,\cdot) = v(0), \\
& \textstyle 
\lim_{s \rightarrow +\infty} v(s) = x_+. 
\label{eq:matching-conditions-2}
\end{aligned}
\right.
\end{equation}
This defines $\delta_{\mathit{left}}$, which satisfies
\begin{equation} \label{eq:delta-left}
\delta_{\mathit{left}} d_{\mathit{Morse}} + \delta_{\mathit{Floer}} d_{\mathit{mixed}} + d_{\mathit{Floer}} \delta_{\mathit{left}} + b = 0.
\end{equation}
The two last terms in \eqref{eq:delta-left} correspond to limits where a Floer trajectory ``breaks off'' from a sequence of solutions $u$ on the left hand side $s \ll 0$. If the limit of the $u$ themselves is non-stationary, such broken solutions are accounted for by $d_{\mathit{Floer}} \delta_{\mathit{left}}$ in the standard way. The other term $b$ counts the remaining limiting configurations. It is defined using a parametrized moduli space of triples $(r,u,v)$, where: $u$ is a solution of Floer's equation; $v$ is a half flow line; and
\begin{equation} \label{eq:matching-conditions-3}
\left\{\begin{aligned}
& \textstyle
\lim_{s \rightarrow -\infty} u(s,\cdot) = x_-, \\
& \textstyle
\lim_{s \rightarrow +\infty} u(s,\cdot) = v(0)^{(-r)}, \\
& \textstyle
\lim_{s \rightarrow +\infty} v(s) = x_+.
\end{aligned}
\right.
\end{equation}

\item[(iv)]
For $x_{\pm}$ as in (iii), we consider an equation for pairs $(u,v)$ with additional parameters $(q,r) \in [0,\infty) \times S^1$. Asymptotic conditions are as in \eqref{eq:matching-conditions-3}, and $u$ is always a solution of Floer's equation. The $q$-dependence lies entirely in the equation satisfied by $v$. For $q \gg 0$, that equation should be that which defines $\delta_{\mathit{Morse}}$, but applied to the positive half-line only, and with the $s$-parameter shifted by $q$:
\begin{equation} \label{eq:morse-bv-shifted}
\partial_s v + \nabla_{g_{r,s-q}} f_{r,s-q} = 0.
\end{equation}
In contrast, for $q = 0$ the equation should be the gradient flow equation for $-\nabla_g f$ (and one connects those two behaviours by choosing some intermediate data). The outcome is that
\begin{equation} \label{eq:delta-right}
d_{\mathit{Floer}} \delta_{\mathit{right}} + \delta_{\mathit{right}} d_{\mathit{Morse}} + d_{\mathit{mixed}} \delta_{\mathit{Morse}} - b = 0,
\end{equation}
where the last two terms correspond to $q \rightarrow \infty$ and $q = 0$, respectively.
\end{itemize}
%

Lemma \ref{th:wall-2} follows directly once one establishes that the Morse-Bott approach yields the same Floer cohomology (and BV operator) as the original definition. There are several ways of doing that, the most natural one being a suitable generalization of continuation maps; but we will not explain the details.

\subsection{Autonomous Hamiltonians}
The relation between Floer cohomology and ordinary cohomology is established by the following classical result:

\begin{proposition} \label{th:bv-vanish}
For sufficiently small $\epsilon > 0$, we have $\mathit{HF}^*(M,\epsilon) \iso H^*(M)$. Moreover, the BV operator $\Delta$ vanishes.
\end{proposition}

The most natural proof is a version of \cite{piunikhin-salamon-schwarz94}, but here we choose instead to follow \cite{floer-hofer-salamon94, hofer-salamon95}, which is more elementary. Take a time-independent $H$ and $J$, such that $H$ is Morse and the metric associated to $J$ is Morse-Smale. Then, after possibly multiplying $H$ with a small positive constant, we have:
\begin{equation} \label{eq:time-independent}
\left\{\!\!\!\!\!\! \parbox{35em}{
\begin{itemize} \itemsep0.5em
\item[(i)] 
all one-periodic orbits of $X_H$ are constant (at the critical points of $H$);
\item[(ii)]
any solution of Floer's equation is constant in $t$, hence a negative gradient flow line for $H$ \cite[Lemma 7.1]{hofer-salamon95};
\item[(iii)]
for the linearization of Floer's equation (at a gradient flow line), all solutions are also constant in $t$ \cite[Proposition 4.2]{salamon-zehnder92}; together with an easy index computation, this shows that the moduli spaces are all regular.
\end{itemize}
}
\right.
\end{equation}
Hence, $\mathit{CF}^*(M,H)$ is isomorphic to the Morse complex of $H$ (this continues to hold when signs are taken into account). Moreover, when defining the BV operator, one may choose $(H_{r,s,t},J_{r,s,t}) = (H,J)$, in which case there are no isolated solutions of \eqref{eq:bv-equation}, since $r$ can be changed freely. Hence, $\delta$ itself vanishes.

We have proved that the BV operator is zero on the chain level, for a specific kind of Hamiltonian and almost complex structure. One can be slightly more precise about the meaning of this vanishing result; for that, it is convenient to introduce some simple algebraic language.

\begin{setup} \label{th:null}
(i) Suppose that we have a chain complex $(C^*,d)$ together with a chain endomorphism $\delta$ of degree $-1$. Recall that a nullhomotopy for $\delta$ is a map $\chi: C^* \rightarrow C^{*-2}$ satisfying $d\chi - \chi d = \delta$. Two nullhomotopies are called equivalent if the difference between them is a nullhomotopic chain map of degree $-2$. Equivalence classes are called {\em homotopy trivializations} of $\delta$.

(ii) Take two chain complexes $(C,d)$ and $(\tilde{C},\tilde{d})$, with endomorphisms $\delta$ and $\tilde\delta$. Suppose that we have maps
\begin{equation}
\begin{aligned}
& k: C^* \longrightarrow \tilde{C}^*, \\
& \kappa: C^* \longrightarrow \tilde{C}^{*-2}, 
\end{aligned}
\end{equation}
of which the first one is a chain map, and the second satisfies \eqref{eq:compatibility-with-delta}. Given homotopy trivializations $[\chi]$ and $[\tilde{\chi}]$ on our complexes, we say that they are {\em compatible with $(k,\kappa)$} if there is a map
\begin{equation} \label{eq:ggg}
g: C^* \longrightarrow \tilde{C}^{*-3}
\end{equation}
such that
\begin{equation}
\tilde{\chi} k - k \chi - \kappa = \tilde{d} g + g d.
\end{equation}
It is easy to see that compatibility depends only on the equivalence classes of $\chi$ and $\tilde{\chi}$. Moreover, if $k$ is a quasi-isomorphism, every homotopy trivialization on one complex determines a unique compatible homotopy trivialization on the other one.
\end{setup}

In this terminology, the desired statement is that the chain complex underlying $\mathit{HF}^*(M,\epsilon)$ (still in the situation of Proposition \ref{th:bv-vanish}, but now with $(H,J)$ arbitrary) comes with a canonical homotopy trivialization of the BV operator, which is compatible with \eqref{eq:cont-maps}. To prove that, one takes two choices $(H,J)$ and $(\tilde{H},\tilde{J})$ to which the proof of Proposition \ref{th:bv-vanish} applies. The associated BV operators vanish on the chain level, hence have trivial nullhomotopies. From \eqref{eq:compatibility-with-delta}, one gets a degree $-2$ chain map between these Floer complexes (itself canonical up to chain homotopy). One can show that, within a slightly more precisely defined class of Hamiltonians and almost complex structures, those chain maps are always nullhomotopic. Nullhomotopies for them provide the desired maps \eqref{eq:ggg}, which show compatibility of the homotopy trivializations. After that, one uses continuation maps to extend to arbitrary $(H,J)$. We omit the details.

\begin{proposition} \label{th:bv-vanish-2}
Suppose that $\partial M$ is a contact circle bundle. Take $\epsilon \in (1,2)$. Then there is a long exact sequence (of $\bZ/2$-graded spaces)
\begin{equation} \label{eq:morse-bott-sequence}
\cdots \rightarrow H^*(M) \longrightarrow \mathit{HF}^*(M,\epsilon) \longrightarrow H^*(\partial M) \rightarrow \cdots
\end{equation}
The connecting map in \eqref{eq:morse-bott-sequence}, $\mathit{H}^*(\partial M) \rightarrow H^{*+1}(M)$, vanishes on $H^0(\partial M)$. Moreover, the BV operator vanishes on the preimage of $H^0(\partial M)$ inside $\mathit{HF}^*(M,\epsilon)$. Finally, suppose that $M$ carries a symplectic Calabi-Yau structure constructed as in Example \ref{th:contact-circle-bundle}(ii). Then, \eqref{eq:morse-bott-sequence} becomes $\bZ$-graded if we take the rightmost term to be $H^{*+2m-2}(\partial M)$.
\end{proposition}

The existence of such an exact sequence is a special case of Lemma \ref{th:wall-2}, combined with Proposition \ref{th:bv-vanish}; what's new are the additional properties (of the connecting map and BV operator).

\begin{proof}
Choose $\epsilon_- > 0$ small, and define $\mathit{HF}^*(M,\epsilon_-)$ using time-independent $H_-$ (and the same for the almost complex structure), as in Proposition \ref{th:bv-vanish}. For $\epsilon_+ = \epsilon \in (1,2)$, extend $H_-$ to $\hat{H}_+$ as in \eqref{eq:extend-hamiltonian}. This enlargement produces a Morse-Bott-nondegenerate manifold of $1$-periodic orbits $F$, which is a copy of $\partial M$ (and one can check that it carries a trivial local system $\xi$). Floer trajectories whose limits $x_{\pm}$ lie in $M$ remain entirely within that subset, and have the same description as in Proposition \ref{th:bv-vanish}. The remaining point is to ensure that Floer trajectories with mixed limits ($x_-$ lies in $M$, and $x_+$ in $F$) can be made regular. Because $x_+$ is a simple periodic orbit, such trajectories are themselves simple, in the sense of \cite[p.~279]{floer-hofer-salamon94}. Then, \cite[Theorem 7.4]{floer-hofer-salamon94} ensures that one can choose $J$ (by varying it near $x_-$) so that regularity holds generically. Three technical remarks are appropriate. First, transversality theory needs to be carried out in an analytic formalism suitable for the Morse-Bott case. Secondly, while the results in \cite{floer-hofer-salamon94} are stated for trajectories with both limits being constant orbits, only one such limit is actually necessary for the argument to go through (and, since the constant orbits are still nondegenerate in our situation, their treatment does not need to modified). Finally, we need a technical condition on the Hessian of $H_-$ at its critical points \cite[Definition 7.1]{floer-hofer-salamon94}, but it is unproblematic to arrange that it holds.

Given that, we can use this same almost complex structure everywhere in the definition of \eqref{eq:morse-bott-d} and \eqref{eq:morse-bott-delta}. Then, as in Proposition \ref{th:bv-vanish} (which means using the fact that $r$ is a free parameter), we get
\begin{align} \label{eq:1-vanish}
& \delta_{\mathit{Floer}} = 0, \\
& \delta_{\mathit{left}} = 0.
\end{align}

Now let $x_+$ be a local minimum (Morse index $0$ critical point) of $f$, so that for degree reasons,
\begin{equation}
\delta_{\mathit{Morse}}(x_+) = 0.
\end{equation}
There is an open subset of $F$ consisting of points $v(0)$ which flow to $x_+$ under $-\nabla_g f$. The Floer trajectories $u$ such that $\lim_{s \rightarrow +\infty} u(s,\cdot)$ lies in that open subset form a space of dimension $\geq 1$ (since one can rotate them a little in the $t$-variable). The same applies if one replaces the gradient flow equation by \eqref{eq:morse-bv-shifted}. As a consequence,
\begin{align}
& d_{\mathit{mixed}}(x_+) = 0, \\
& \delta_{\mathit{right}}(x_+) = 0. \label{eq:5-vanish}
\end{align}
The desired properties of \eqref{eq:morse-bott-sequence} and $\Delta$ follow directly from this. The observation about $\bZ$-gradings is a standard Conley-Zehnder index computation.
\end{proof}

As before, it can be useful to encode this observation in a more abstract framework, which also allows one to formulate a uniqueness property.

\begin{setup} \label{th:truncation}
(i) In the situation of Setup \ref{th:null}(i), fix some integer $j$. Let $C^{\leq j} \subset C^*$ be the subcomplex consisting of all cocycles of degree $j$, together with all cochains of degree $<j$. Note that this is automatically preserved by $\delta$. A {\em homotopy trivialization of $\delta$ in degrees $\leq j$} is a homotopy trivialization of $\delta^{\leq j} = \delta|C^{\leq j}$, in the previously defined sense. The existence of such a trivialization implies that $[\delta]: H^*(C) \rightarrow H^{*-1}(C)$ vanishes in degrees $\leq j$, since the map $H^*(C^{\leq j}) \rightarrow H^*(C)$ is an isomorphism in those degrees.

(ii) In the situation of Setup \ref{th:null}(ii), suppose that we have homotopy trivializations of $C$ and $\tilde{C}$ in degrees $\leq j$. We say that these are compatible if there exists a map \eqref{eq:ggg} defined on $C^{\leq j}$, with the same properties as before.
\end{setup}

Suppose that we are in the situation of Proposition \ref{th:bv-vanish-2}, with a symplectic Calabi-Yau structure constructed as in Example \ref{th:contact-circle-bundle}(ii). Then, the proof of that Proposition yields (for a very special choice of Morse-Bott complex underlying Floer cohomology) a homotopy trivialization of the BV operator in degrees $\leq 2m-2$. The relevant uniqueness result would say that this is compatible with continuation maps. Proving this requires an analogue of Proposition \ref{th:bv-vanish-2} for continuation maps, which we will not explain here.

\begin{remark}
The statement of Proposition \ref{th:bv-vanish-2} is actually not optimal. One can show that the connecting map vanishes on $\mathit{im}(H^*(\partial M/S^1) \rightarrow H^*(\partial M))$, and that the BV operator is zero on the preimage of that subspace inside $\mathit{HF}^*(M,\epsilon)$. To do that, one has to use a chain level model for $H^*(\partial M)$ which is more closely adapted to the circle action than what we've done (one possibility is to perturb the Reeb flow on $\partial M$ so that the one-periodic orbits become transversally nondegenerate).

Alternatively, one can use the $S^1$-equivariant version of Hamiltonian Floer cohomology (see e.g.\ \cite{seidel07}), which is a $\bK[[u]]$-module $\mathit{HF}^*_{S^1}(M,\epsilon)$. In the situation of Proposition \ref{th:bv-vanish-2}, it fits into a long exact sequence of such modules,
\begin{equation} \label{eq:s1-equi-les}
\cdots \rightarrow H^*(M)[[u]] \longrightarrow \mathit{HF}^*_{S^1}(M,\epsilon) \longrightarrow H^*(\partial M/S^1) \rightarrow \cdots
\end{equation}
where $u$ acts on $H^*(\partial M/S^1)$ by cup product with the first Chern class of the circle bundle $\partial M \rightarrow \partial M/S^1$. Since that action is nilpotent, the boundary operator of \eqref{eq:s1-equi-les} is necessarily zero. Now, \eqref{eq:s1-equi-les} and \eqref{eq:morse-bott-sequence} fit into a commutative diagram whose rows and columns are long exact sequences:
\begin{equation}
\xymatrix{
& 
\vdots
\ar[d]^-{\text{zero}}
& 
\vdots
\ar[d]
&
\vdots
\ar[d]
&
\\
\cdots
\ar[r]^-{\text{zero}}
& 
H^*(M)[[u]] 
\ar[r] \ar[d]^-{u}
& 
\mathit{HF}^*_{S^1}(M,\epsilon)
\ar[r] \ar[d]^-{u}
&
H^*(\partial M/S^1)
\ar[r]^-{\text{zero}} \ar[d]^-{u} 
& 
\cdots
\\
\cdots
\ar[r]^-{\text{zero}}
& 
H^*(M)[[u]] 
\ar[r] \ar[d]
& 
\mathit{HF}^*_{S^1}(M,\epsilon)
\ar[r] \ar[d]
&
H^*(\partial M/S^1)
\ar[r]^-{\text{zero}} \ar[d]
& 
\cdots
\\
\cdots
\ar[r]
&
H^*(M)
\ar[r]
\ar[d]^-{\text{zero}}
&
\mathit{HF}^*(M,\epsilon)
\ar[r]
\ar[d]
&
H^*(\partial M) 
\ar[r]
\ar[d] 
&
\cdots
\\
& 
\vdots
& 
\vdots
& 
\vdots
&
}
\end{equation}
The right hand column is the standard Gysin sequence; and the BV operator is the composition of two vertical arrows (going from $\mathit{HF}^*(M,\epsilon)$ to $\mathit{HF}^*_{S^1}(M,\epsilon)$ and then back). Diagram-chasing yields the desired result: if a class in $\mathit{HF}^*(M,\epsilon)$ has the property that its image in $H^*(\partial M)$ is the pullback of a class in $H^*(\partial M/S^1)$, then it necessarily comes from a class in $\mathit{HF}^*_{S^1}(M,\epsilon)$, hence is killed by the BV operator.
\end{remark}

\section{Fixed point Floer cohomology}

The generalisation of Floer cohomology from Hamiltonian to general symplectic automorphisms was introduced in \cite{dostoglou-salamon93, dostoglou-salamon94}. Here, we use a version for Liouville domains, as in \cite[Section 4]{seidel00b}, \cite{mclean12}, or \cite{uljarevic14}. 

\subsection{Definition}
For $M$ as before, we will consider the following situation.

\begin{setup} \label{th:auto-setup}
(i) Let $\phi$ be a symplectic automorphism of $M$ which is equal to the identity near $\partial M$ and exact. The latter condition means that there is a function $G_\phi$ (necessarily locally constant near $\partial M$) such that $\phi^*\theta_M - \theta_M = dG_\phi$. 

(ii) If $M$ carries a symplectic Calabi-Yau structure, we will assume that $\phi$ is compatible with it, and in fact comes with a choice of grading (making it a graded symplectic automorphism \cite{seidel99}).
\end{setup}

\begin{example} \label{th:boundary-twist}
Suppose that $\partial M$ is a contact circle bundle. Fix a function $F$ on $M$ which agrees with $\rho_M$ near the boundary, and let $(\phi^t_F)$ be its flow. By construction, $\phi_F^{-1}$ is the identity near the boundary. We call it the {\em boundary twist} of $M$ (see \cite[Section 4]{seidel99} for a general discussion of such symplectic automorphisms), and denote it by $\tau_{\partial M}$.

Let's now assume that $M$ carries a symplectic Calabi-Yau structure as in Example \ref{th:contact-circle-bundle}(ii). Then, $\tau_{\partial M}$ can be equipped with the structure of a graded symplectic automorphism. In fact, there are two reasonable choices (which coincide for $m = 1$):

(i) One can take the trivial grading of the identity, and extend it continuously over the isotopy $(\phi_F^t)$ to get a grading of $\tau_{\partial M}$. Near the boundary, that grading is a shift $[2-2m]$.

(ii) Alternatively, by changing the previous grading by a constant $2-2m$, one can get the unique grading of $\tau_{\partial M}$ which is trivial near the boundary.
\end{example}

Take families of functions $H = (H_t)$ and $J = (J_t)$, parametrized by $t \in \bR$. Each of those should be as in \eqref{eq:hamiltonian} and \eqref{eq:j-convex}, but now with $\phi$-twisted periodicity properties:
\begin{equation} \label{eq:phi-periodicity}
\left\{
\begin{aligned}
& H_{t+1}(x) = H_t(\phi(x)), \\
& J_{t+1} = \phi^*J_t.
\end{aligned}
\right.
\end{equation}
Take the space $\scrL_\phi = \{x: \bR \rightarrow M \,:\, x(t) = \phi(x(t+1))\}$, with action functional
\begin{equation} \label{eq:phi-action}
A_{\phi,H}(x) = \Big( \int_0^1 -x^*\theta_M + H_t(x(t)) \, dt \Big) - G_\phi(x(1)).
\end{equation}
Its critical points are solutions $x \in \scrL_\phi$ of \eqref{eq:periodic-y}, and correspond bijectively to fixed points $x(1)$ of $\phi_H^1 \circ \phi$. For a generic choice of $H$, these will be nondegenerate, and we use them as generators of a $\bZ/2$-graded $\bK$-vector space $\mathit{CF}^*(\phi,H)$. To define the differential, one considers solutions $u: \bR^2 \rightarrow M$ of \eqref{eq:floer}, but now satisfying
\begin{equation} \label{eq:floer-2}
u(s,t) = \phi(u(s,t+1)). \\
\end{equation}
The resulting Floer cohomology is denoted by $\mathit{HF}^*(\phi,\epsilon)$, with the previous $\mathit{HF}^*(M,\epsilon)$ being the special case $\phi = \mathit{id}_M$. Remark \ref{th:signs-and-grading} carries over to this more general situation; except that to make $\mathit{HF}^*(\phi,\epsilon)$ $\bZ$-graded, one needs to impose conditions on $\phi$ as well as on $M$, as in Setup \ref{th:auto-setup}(ii).

\begin{remark} \label{th:circle-and-discrete-actions}
In general, $\mathit{HF}^*(\phi,\epsilon)$ does not carry a BV operator. The exception is when $\phi$ can be written as the time-one map of a Hamiltonian flow $(\phi^t)$. Here, the general $\phi^t$ do not have to be equal to the identity near $\partial M$, but they have to preserve $\theta_M$ near $\partial M$ (one example of this would be the flow that leads to $\phi = \tau_{\partial M}$). One then defines a circle action on $\scrL_\phi$ by mapping a twisted loop $x$ to
\begin{equation}
x^{(r)}(t) = \phi^t(x(t-r)).
\end{equation}

There is a related situation where fixed point Floer cohomology admits a discrete symmetry, which was extensively studied in \cite{seidel14c, polterovich-shelukhin15}. Namely, suppose that $\phi$ admits a $k$-th root $\phi^{1/k}$ (such that $(\phi^{1/k})^*\theta_M = \theta_M$ near $\partial M$). This gives rise to an action of $\bZ/k$ on $\scrL_\phi$, whose generator maps $x$ to
\begin{equation}
x^{(1/k)}(t) = \phi^{1/k}(x(t-1/k)).
\end{equation}
One gets an induced action of $\bZ/k$ on $\mathit{HF}^*(\phi)$ (in the previously considered case of Hamiltonian flows, $\phi$ has $k$-th roots for any $k$; but the induced $\bZ/k$-actions on Floer cohomology are trivial, which is intuitively clear since they embed into a continuous symmetry).
\end{remark}

Floer cohomology is invariant under isotopies of $\phi$ (within the class of symplectic automorphisms that are allowed). One has a straightforward counterpart of \eqref{eq:id-poincare-duality}:
\begin{equation}\label{eq:phi-poincare-duality}
\mathit{HF}^*(\phi^{-1},-\epsilon) \iso \mathit{HF}^{2n-*}(\phi,\epsilon)^\vee.
\end{equation}
As for the dependence on $\epsilon$, Lemmas \ref{th:no-wall}--\ref{th:wall-2}, with the parts about BV operators omitted, carry over with essentially the same proofs.

\begin{lemma} \label{th:tau-shift}
Suppose that $\partial M$ is a contact circle bundle, and let $\tau_{\partial M}$ be the boundary twist. For any $\phi$ and $\epsilon$, 
\begin{equation} \label{eq:shift-isomorphism}
\mathit{HF}^*(\tau_{\partial M}^{-1} \circ \phi, \epsilon - 1) \iso \mathit{HF}^*(\phi,\epsilon). 
\end{equation}
This isomorphism is compatible with $\bZ$-gradings if we use option (i) from Example \ref{th:boundary-twist} as a grading for $\tau_{\partial M}$; if instead we use option (ii), the right hand side of \eqref{eq:shift-isomorphism} should be replaced with $\mathit{HF}^{*+2-2m}(\phi,\epsilon)$.
\end{lemma}

\begin{proof}
When defining $\tau_{\partial M}$, one can choose the flow $(\phi^t_F)$ to be supported arbitrarily close to $\partial M$, so that it commutes with $\phi$. Consider the diffeomorphism
\begin{equation} \label{eq:change-loop}
\begin{aligned}
& \scrF: \scrL_{\tau_{\partial M}^{-1} \circ \phi} \longrightarrow \scrL_{\phi}, \\
& (\scrF x)(t) = \phi^t_F(x(t)),
\end{aligned}
\end{equation}
which satisfies
\begin{equation}
\begin{aligned}
& \scrF^* A_{\phi,H} = A_{\tau_{\partial M}^{-1} \circ \phi, \tilde{H}}, \\
& \tilde{H}_t(x) = H_t(\phi^t_F(x)) - F(x).
\end{aligned}
\end{equation}
Note that close to $\partial M$, $\tilde{H}_t = \epsilon\rho_M - \rho_M$. Given suitable choices of almost complex structures, \eqref{eq:change-loop} gives rise to an isomorphism of Floer cochain complexes, which induces \eqref{eq:shift-isomorphism}.
\end{proof}

\begin{example} \label{th:grading-of-boundary-twist}
Assuming choice (ii) for the grading, 
$\mathit{HF}^{*+2-2m}(\tau_{\partial M}, \epsilon) \iso \mathit{HF}^*(M,\epsilon-1)$. In particular, using Proposition \ref{th:bv-vanish} and \eqref{eq:id-poincare-duality}, one obtains \eqref{eq:trivial-monodromy} (for $m = 0$) and its generalization \eqref{eq:trivial-monodromy-2} (for arbitrary $m$).
\end{example}

\subsection{Symplectic mapping tori\label{subsec:mapping-tori}}
As usual, the mapping torus construction relates discrete dynamics (symplectic automorphisms) and its continuous counterpart (Hamiltonian flows).

\begin{setup} \label{th:mapping-torus-setup}
(i) Let $M$ be a Liouville domain, and $\mu$ an exact symplectic automorphism (as in Setup \ref{th:auto-setup}, with associated function $G_\mu$). Its symplectic mapping torus is the fibration
\begin{equation} \label{eq:mapping-cone-projection}
E = \frac{\bR^2 \times M}{(p,q,x) \sim (p,q-1,\mu(x))}
\xrightarrow{\pi(p,q,x) = (p,q)}
\bR \times S^1.
\end{equation}
The symplectic form is $\omega_E = dp \wedge dq + \omega_M$. To obtain a primitive, one chooses a function $G = G(q,x)$ such that $G_\mu(x) = G(q,x) - G(q-1,\mu(x))$, and sets $\theta_E = p\,\mathit{dq} + \theta_M + dG$.

(ii) If $M$ has a symplectic Calabi-Yau structure, and $\mu$ is a graded symplectic automorphism, the mapping torus inherits a symplectic Calabi-Yau structure.

(iii) We use the class of functions $H$ on $E$ of the following form. Fix constants $\epsilon$, as in \eqref{eq:no-reeb}, and $\gamma_{\pm} \in \bR \setminus \bZ$. Let $H_M$ be a function on the fibre, which equals $\epsilon \rho_M$ near the boundary, and which is invariant under $\mu$ (such functions can easily be constructed as cutoffs of $\epsilon \rho_M$). Let $H_{\bR \times S^1}$ be a function on the base, such that $H_{\bR \times S^1}(p,q) = \gamma_{\pm} p$ if $\pm p \gg 0$. Then, there should be an open subset $U \subset E$, which contains $\partial E$ and has compact complement, such that
\begin{equation} \label{eq:mapping-torus-h}
H(p,q,x) = H_{\bR \times S^1}(p,q) + H_M(x) \quad \text{for $(p,q,x) \in U$.}
\end{equation}
We stress that if we have several different $H$, the associated $H_{\bR \times S^1}$ and $H_M$ can also be different (they do not have to be fixed once and for all).

(iv) We will use almost complex structures $J$ on $E$ such that: outside a compact subset, $\pi$ is $J$-holomorphic (with respect to the standard complex structure $i$ on the base); and near the boundary, $J = J_{\bR \times S^1} \times J_{M,p,q}$, where $J_{\bR \times S^1}$ is an almost complex structure on the base which is standard outside a compact subset, and the $J_{M,p,q}$ are as in Setup \ref{th:hamiltonian-setup}.
\end{setup}

Given time-dependent $H = (H_t)$ and $J = (J_t)$, one can build a Floer complex $\mathit{CF}^*(E,H)$ as before. Of course, one has to check that solutions of Floer's equation $u: \bR \times S^1 \rightarrow E$ remain inside a compact subset of $E \setminus \partial E$. To see that this is the case, note that on the subset where the $|p|$-coordinate of $u$ is large, $v = \pi(u)$ will itself be a solution of 
\begin{equation}
\partial_s v + i(\partial_t v - \gamma_{\pm} \partial_q) = 0,
\end{equation}
($\partial_q$ stands for the unit vector field in $q$-direction), hence its $p$-component is harmonic. On the other hand, where $u$ is close to $\partial E$, one can consider its fibre component alone, and then apply the maximum principle in the same way as when constructing Floer cohomology inside $M$. The resulting groups $\mathit{HF}^*(E,\gamma_-,\gamma_+,\epsilon)$ depend only on the constants involved (and we can change those, without affecting Floer cohomology, as long as no ``forbidden values'' are crossed, in parallel with Lemma \ref{th:no-wall}). They also carry a BV operator.

\begin{lemma} \label{th:mclean-example}
Suppose that $\gamma_- \in (0,1)$ and $\gamma_+ \in (1,2)$. Then
\begin{equation} \label{eq:s1-floer}
\mathit{HF}^*(E,\gamma_-,\gamma_+,\epsilon) \iso H^*(S^1) \otimes \mathit{HF}^*(\mu,\epsilon).
\end{equation}
The BV operator is given by rotation on $S^1$, tensored with the identity map on $\mathit{HF}^*(\mu,\epsilon)$.
\end{lemma}

The Floer cohomology computation \eqref{eq:s1-floer} is \cite[Theorem 1.3]{mclean12}, which can be proved relatively straightforwardly by a Morse-Bott approach. A similar argument determines the BV operator.

\section{Rotations at infinity\label{subsec:rotate}}
The next step is to apply Hamiltonian Floer cohomology to total spaces of Lefschetz fibrations. Besides making the necessary adjustments to the construction of Floer cohomology, we will consider some specific computations. Those follow \cite{mclean12} fairly closely, but with the BV operator as an added ingredient.

\subsection{Target spaces}
The Lefschetz (complex nondegeneracy) condition is not important yet, and we will allow considerably more freedom for the local geometry. On the other hand, we impose quite strict conditions on the behaviour near infinity.

\begin{setup} \label{th:setup-e}
(i) An {\em exact symplectic fibration with singularities} is a $2n$-dimensional manifold with boundary $E$, together with an exact symplectic form $\omega_E = d\theta_E$, and a proper map
\begin{equation} \label{eq:symplectic-fibration}
\pi: E \longrightarrow \bC,
\end{equation}
subject to the following conditions.

At any $x \in E$, define the horizontal subspace $TE^h_x \subset TE$ as the $\omega_E$-orthogonal complement of $TE^v_x = \mathit{ker}(D\pi_x)$. We ask that there should be an open subset $U \subset E$ which contains $\partial E$ and has compact complement, such that at each $x \in U$, one has: $TE_x^v$ is a codimension $2$ symplectic subspace of $TE$ (hence, $x$ is a regular point of $\pi$, and $D\pi_x|TE^h_x$ is an isomorphism); and
\begin{equation} \label{eq:flat}
(\omega_E-\pi^*\omega_{\bC})\,|\,TE^h_x = 0,
\end{equation}
where $\omega_{\bC} = d\mathrm{re}(y) \wedge d\mathrm{im}(y)$ is the standard symplectic form on the base. 

Next, at each point $x \in \partial E$, $TE_x^h$ should lie inside $T_x(\partial E)$.

Finally, there should be a $y_* > 0$ such that: all fibres $E_y = \pi^{-1}(y)$ with $|y| \geq y_*$ lie inside the previously introduced subset $U$; and $M = E_{y_*}$, with its induced exact symplectic structure, is a Liouville domain (in all subsequent developments, we will assume that such a $y_*$ has been fixed).

(ii) For part of our considerations, we will assume that $E$ comes with a symplectic Calabi-Yau structure. This induces the same kind of structure on $M$.
\end{setup}

Let's consider the implications of these conditions. We have symplectic parallel transport maps defined in a neighbourhood of $\partial E$. More precisely, let $W \subset M$ be a small open neighbourhood of $\partial M$. Then, parallel transport yields a canonical embedding, whose image is an open neighbourhood of $\partial E \subset E$:
\begin{equation} \label{eq:first-trivialization}
\left\{
\begin{aligned}
& \Theta: \bC \times W \longrightarrow E, \\
& \Theta(\bC \times \partial M) = \partial E, \\
& \pi(\Theta(y,x)) = y, \\
& \Theta(y_*,x) = x, \\
& \Theta^*\omega_E = \omega_\bC + \omega_M.
\end{aligned}
\right.
\end{equation}
To see why that is the case, note that \eqref{eq:flat} implies that the symplectic connection is flat near $\partial E$. Hence, the relevant part of parallel transport is independent of the choice of path, and $\Theta^*\omega_E - \omega_M$ must be locally the pullback of some two-form on $\bC$; to determine that two-form, one again appeals to \eqref{eq:flat}.

We also have the same flatness property outside a compact subset, with the following consequence. Let $\mu$ be the monodromy around the circle of radius $y_*$, which is an exact symplectic automorphism of $M$ (and restricts to the identity on $W$). 
Then, for $p_* = \log(y_*)$, there is a unique covering map
\begin{equation} \label{eq:second-trivialization}
\left\{
\begin{aligned}
& \Psi: [p_*,\infty) \times \bR \times M \longrightarrow \{|\pi(x)| \geq y_*\} \subset E, \\
& \pi(\Psi(p,q,x)) = e^{p+iq}, \\
& \Psi(p_*,0,x) = x, \\
& \Psi(p,q+2\pi,x) = \Psi(p,q,\mu(x)), \\
& \Psi^*\,\omega_E = e^{2p} dp \wedge dq + \omega_M.
\end{aligned}
\right.
\end{equation}
The partial trivializations \eqref{eq:first-trivialization} and \eqref{eq:second-trivialization} are compatible, in the sense that
\begin{equation}
\Theta(e^{p+iq},x) = \Psi(p,q,x) \quad \text{for $p \geq p_*$ and $x \in W$.}
\end{equation}
In words, what we have observed is that \eqref{eq:symplectic-fibration} is symplectically trivial near $\partial E$, whereas outside the preimage of a disc, it is like a symplectic mapping torus for $\mu$ (but with a different symplectic structure on the base than before). Finally, in the situation of Setup \ref{th:setup-e}(ii), $\mu$ is a graded symplectic automorphism (in a preferred way, such that the grading is trivial near $\partial M$).

\begin{example} \label{th:anticanonical-lefschetz-pencil}
The most important examples for us are those obtained from anticanonical Lefschetz pencils, or more generally, from Lefschetz pencils satisfying the condition from Remark \ref{th:fractional-cy-1}. In that case, the boundary of the fibre is a contact circle bundle; $E$ carries a symplectic Calabi-Yau structure, such that the induced structure on the fibre is as in Example \ref{th:contact-circle-bundle}(ii); and as already stated in \eqref{eq:mu-tau}, the monodromy is the associated boundary twist.
\end{example}

\begin{setup} \label{th:setup-rotation}
(i) We consider functions $H \in \smooth(E,\bR)$ such that
\begin{equation}
\left\{
\begin{aligned}
& H(\Theta(y,x)) = H_{\bC}(y) + \epsilon \rho_M(x) \quad \text{for $x$ near $\partial M$}, \\
& H(\Psi(p,q,x)) = \textstyle \gamma e^{2p}/2 + H_M(x) \quad \text{for $p \gg 0$.}
\end{aligned}
\right.
\end{equation}
Here, the constant $\epsilon$ is as in \eqref{eq:no-reeb}, and $\gamma \in \bR \setminus 2\pi\bZ$. $H_{\bC}$ is a function on the base such that $H_{\bC}(y) = \gamma |y|^2/2$ for $|y| \gg 0$, and $H_M$ is as in Setup \ref{th:mapping-torus-setup}.

(ii) We use compatible almost complex structures $J$ on $E$ satisfying the following conditions. Outside a compact subset, $\pi$ is $J$-holomorphic, with respect to the standard complex structure $i$ on the base; and at points $(y,x) \in \bC \times M$ with $x$ sufficiently close to $\partial M$,
\begin{equation} \label{eq:theta-star-j}
\Theta^*J = J_{\bC} \times J_{M,y}.
\end{equation}
Here, $J_{\bC}$ is an almost complex structure on $\bC$ which is standard outside a compact subset, and each $J_{M,y}$ is an almost complex structure on $M$ as in Setup \ref{th:hamiltonian-setup}. 
\end{setup}

\subsection{Floer cohomology and its properties}
The construction of Hamiltonian Floer cohomology in this context proceeds pretty much as in Section \ref{subsec:mapping-tori}. One gets chain complexes $\mathit{CF}^*(E,H)$ and cohomology groups $\mathit{HF}^*(E,\gamma,\epsilon)$, together with a BV operator. A suitable analogue of Lemma \ref{th:no-wall} holds. There are also counterparts of Lemma \ref{th:wall} for changing either $\gamma$ or $\epsilon$. We will not consider them in full generality, but one important special case is this:

\begin{lemma} \label{th:gamma-wall}
For $\gamma_- \in (0,2\pi)$, $\gamma_+ \in (2\pi,4\pi)$, and any $\epsilon$, we have a long exact sequence
\begin{equation} \label{eq:2pi-sequence}
\cdots \rightarrow \mathit{HF}^*(E,\gamma_-,\epsilon) \longrightarrow
\mathit{HF}^*(E,\gamma_+,\epsilon) \longrightarrow H^*(S^1) \otimes \mathit{HF}^{*+2}(\mu,\epsilon) \rightarrow \cdots
\end{equation}
where $\mu$ is the monodromy. This is compatible with BV operators, where the operator on the rightmost group is as in Lemma \ref{th:mclean-example}. Moreover, if $E$ has a symplectic Calabi-Yau structure, \eqref{eq:2pi-sequence} is compatible with $\bZ$-gradings.
\end{lemma}

\begin{proof}
To simplify the notation, we assume that one can take $y_* = 1$. Fix a function $H_M$ which is supported in a small neighbourhood of $\partial M$ (hence is invariant under $\mu$), and which equals $\epsilon\rho_M$ close to $\partial M$. Via \eqref{eq:first-trivialization}, there is an obvious extension of this function to all of $E$, which we again denote by $H_M$. Choose a function
\begin{equation} \label{eq:rotation-function}
\left\{
\begin{aligned}
& h_+: [0,\infty) \longrightarrow \bR, \\
& h_+'(a) = 0 \quad \text{if and only if $a \leq 1/2$,} \\
& h_+'(a) = \gamma_- \quad \text{for $a$ close to $2$,} \\
& h_+(a) = \gamma_+ a \quad \text{for $a \gg 0$}, \\
& h_+''(a) \geq 0 \quad \text{everywhere}, \\
& h_+''(a) > 0 \quad \text{whenever $h'_+(a) \in (\gamma_-,\gamma_+)$.}
\end{aligned}
\right.
\end{equation}
Define $H_+: E \rightarrow \bR$ by
\begin{equation}
H_+(x) = h_+(|\pi(x)|^2/2) + H_M(x).
\end{equation}
This has two kinds of $1$-periodic orbits:
\begin{itemize} \itemsep.5em
\item[(i)] The first kind are contained in a fibre $\pi^{-1}(y)$, $|y| \leq 1$. They are either constant, or else nontrivial $1$-periodic orbits of $H_M$ (lying close to $\partial M$).
\item[(ii)] The second kind of $1$-periodic orbit is fibered over the circle $|y|^2/2 = a_*$, where $a_* > 2$ is the unique value such that $h'_+(a_*) = 2\pi$. In each fibre, they correspond to fixed points of $\phi_{H_M}^1 \circ \mu$.
\end{itemize}
Let's perturb $H_+$ slightly (in a time-dependent way, and without changing the notation), so as to make the $1$-periodic points nondegenerate; we will assume that the perturbation is trivial near $\partial E$, as well as outside a compact subset. The outcome lies in the class of Hamiltonians which can be used to construct $\mathit{HF}^*(E,\gamma_+,\epsilon)$. As long as the perturbation is small, one can still maintain the distinction between type (i) and (ii) orbits, even though they no longer have exactly the same properties as before. 

We now consider Floer trajectories. When perturbing $H_+$, we can assume that the perturbation is trivial near the preimage of the circle $Z = \{|y| = 2\}$. We will also use almost complex structures $J = (J_t)$ such that $\pi$ becomes pseudo-holomorphic near that same preimage. As a consequence, if $u$ is a Floer trajectory, then $v = \pi(u)$ satisfies
\begin{equation} \label{eq:v-projection}
\partial_s v + i (\partial_t v - \gamma_- iv) = 0 \quad \text{whenever $v$ is close to $Z$.}
\end{equation}
We exploit that by introducing a degenerate version of the action functional (this is essentially what \cite{seidel12b} called a ``barrier argument''). On the base, choose a two-form $\bar{\omega}_{\bC}$ which is: rotationally invariant; everywhere nonnegative; supported in a small neighbourhood of $Z$; and positive at every point of $Z$. Then,
\begin{equation} \label{eq:barrier}
0 \leq \int_{\bR \times S^1} \bar\omega_{\bC}(\partial_s v, \partial_tv - \gamma_- iv) =
\bar{A}(y_-) - \bar{A}(y_+),
\end{equation}
where $\bar{A}(y_{\pm})$ depends only on the limits $y_{\pm}$ of $v$. To compute it explicitly, fix a primitive $\bar{\omega}_{\bC} = d\bar{\theta}_{\bC}$. Also, consider the function $\bar{H}_{\bC}$ such that $d\bar{H}_{\bC} = \bar{\omega}_{\bC}(\cdot, \gamma_- i y)$, normalized by asking that it should vanish near the origin. Then,
\begin{equation}
\bar{A}(y) = \int_{S^1} -y^*\bar{\theta}_{\bC} + \bar{H}_{\bC}(y(t))\, \mathit{dt}.
\end{equation}
Concretely,
\begin{equation}
\bar{A}(y) = \begin{cases} 0 & \text{if $y = \pi(x)$, with $x$ of type (i) as listed above,} \\
\textstyle (\frac{\gamma_-}{2\pi} - 1) \textstyle \int_{\bC} \bar{\omega}_{\bC} < 0 & \text{if $x$ is of type (ii).}
\end{cases}
\end{equation}
Finally, note that equality in \eqref{eq:barrier} holds only if $v^{-1}(Z) = \emptyset$. It follows that there are only three kinds of Floer trajectories:
\begin{itemize} \itemsep.5em
\item[(i, i)]
If both limits $x_{\pm}$ are of type (i), the Floer trajectory remains entirely on the inside of $\pi^{-1}(Z)$;
\item[(ii, ii)]
If both limits $x_{\pm}$ are of type (ii), the Floer trajectoy remains entirely on the outside of $\pi^{-1}(Z)$, since $\bar{A}(y_-) = \bar{A}(y_+)$ (the same purpose was achieved by the ``minimum principle argument'' in \cite[Lemma 6.1]{mclean12});
\item[(i, ii)]
Finally, $x_-$ could be of type (i), and $x_+$ of type (ii).
\end{itemize}

Take a function $h_-: [0,\infty) \rightarrow \bR$ which agrees with $h$ on $[0,2]$, and satisfies $h'_-(a) = \gamma_-$ for $a \geq 2$. Let $H_-$ be a (time-dependent) function which agrees with $H_+$ over the preimage of $\{|y| \leq 2\}$, and satisfies $H_-(x) = h_-(|\pi(x)|^2/2) + H_M(x)$ elsewhere. This has only the type (i) periodic orbits, and the associated chain complex $\mathit{CF}^*(E,H_-)$ has cohomology $\mathit{HF}^*(E, \gamma_-,\epsilon)$. From the argument above, we get (assuming suitable choices of almost complex structures) a short exact sequence of chain complexes
\begin{equation} \label{eq:gamma-gamma}
0 \rightarrow \mathit{CF}^*(E,H_-) \longrightarrow \mathit{CF}^*(E,H_+) \longrightarrow Q^* \rightarrow 0.
\end{equation}
The differential on the quotient $Q^*$ counts type (ii, ii) trajectories only, hence defines a (slightly modified) version of Floer cohomology inside the mapping torus of $\mu$, which is as computed in Lemma \ref{th:mclean-example}. This shows that \eqref{eq:gamma-gamma} induces the desired long exact sequence \eqref{eq:2pi-sequence}. The argument about BV operators is parallel, since one can arrange that the relevant equation \eqref{eq:bv-equation} has the same property \eqref{eq:v-projection}. The statement about gradings is straightforward (the shift by $2$ reflects the Conley-Zehnder index of the circle $|y|^2/2 = a_*$, as a $1$-periodic orbit of the Hamiltonian $h_+(|y|^2/2)$ on $\bC$).
\end{proof}

\begin{proposition} \label{th:bv-vanish-3}
For $\gamma \in (0,2\pi)$ and sufficiently small $\epsilon>0$, 
\begin{equation} \label{eq:e-pss}
\mathit{HF}^*(E,\gamma,\epsilon) \iso H^*(E). 
\end{equation}
Moreover, the BV operator vanishes.
\end{proposition}


This is an analogue of Proposition \ref{th:bv-vanish}, and has the same proof. We will now dig a little deeper into the consequences. In the situation of Lemma \ref{th:gamma-wall}, choose a class $z \in \mathit{HF}^{*+2}(\mu,\epsilon)$. Using notation as in \eqref{eq:gamma-gamma}, take $[\mathit{point}] \otimes z \in H^1(S^1) \otimes \mathit{HF}^{*+2}(\mu,\epsilon) \subset H^{*+1}(Q)$, represent it by a cocycle in $Q^*$, and lift that cocycle to a cochain
\begin{equation} \label{eq:x-cochain}
x \in \mathit{CF}^{*+1}(E,H_+),  \quad dx \in \mathit{CF}^{*+2}(E,H_-).
\end{equation}
Assume that $\gamma_-$ and $\epsilon$ are small. Using the proof of Proposition \ref{th:bv-vanish} (as applied to our situation in Lemma \ref{th:bv-vanish-3}), one can arrange that the chain level BV operator vanishes on $\mathit{CF}^*(E,H_-) \subset \mathit{CF}^*(E,H_+)$. As a result, we get a cocycle
\begin{equation} \label{eq:bar-x}
\bar{x} = \delta x \in \mathit{CF}^{*}(E,H_+), 
\end{equation}
since $d \bar{x} = -\delta (dx) = 0$. Clearly, $\bar{x}$ is independent of the choice of lift $x$. Moreover, if the original cohomology class was trivial, we would have $x - dy \in \mathit{CF}^{*+1}(E,H_-)$ for some $y$, and hence
\begin{equation}
\bar{x} = \delta(dy) = -d(\delta y).
\end{equation}
This shows that $[\bar{x}] \in \mathit{HF}^*(E,\gamma,\epsilon_+)$ depends only on $z$. Finally, if we project $\bar{x}$ to $Q^*$, it represents $1 \otimes z \in H^0(S^1) \otimes \mathit{HF}^{*+2}(\mu,\epsilon) \subset H^*(Q)$ (by the description of the BV operator in Lemmas \ref{th:mclean-example} and \ref{th:gamma-wall}). The outcome of this argument can be summarized as follows:

\begin{proposition} \label{th:partial-splitting}
Take the situation from Lemma \ref{th:gamma-wall}, with $\epsilon>0$ small. Then the map 
\begin{equation} \label{eq:proj-s1}
\mathit{HF}^*(E,\gamma_+,\epsilon) \longrightarrow H^*(S^1) \otimes \mathit{HF}^{*+2}(\mu,\epsilon)
\end{equation}
from \eqref{eq:2pi-sequence} has a partial splitting, defined on $H^0(S^1) \otimes \mathit{HF}^{*+2}(\mu,\epsilon)$. \qed
\end{proposition}

For the benefit of readers concerned by the apparent ad hoc nature of the previous argument, we can explain how it fits into an appropriate formal framework. 

\begin{setup} \label{th:split-setup}
(i) Take a chain complex $C^*_+$ together with an endomorphism $\delta$ of degree $-1$. Let $C^*_-$ be a subcomplex such that $\delta(C^*_-) \subset C^*-$. Then, a homotopy trivialization $[\chi]$ of $\delta|C^*_-$ induces a lift
\begin{equation} \label{eq:lift-bv}
\xymatrix{
&& H^{*-1}(C_+) \ar[d] \\
H^*(Q) \ar@{-->}[urr] \ar[rr] && H^{*-1}(Q)
}
\end{equation}
where the horizontal arrow is the operation induced by $\delta$ on the quotient complex $Q^* = C^*_+/C^*_-$, and the vertical arrow is the projection map. The lift in \eqref{eq:lift-bv} is defined as follows: given a cocycle in $Q^*$, lift it to a cochain $x$ in $C^*_+$, and then map that to the cocycle 
\begin{equation} \label{eq:corrected-bar}
\bar{x} = \delta x + \chi(dx).
\end{equation}
One sees immediately that this induces a map on cohomology, which depends only on the equivalence class $[\chi]$.

(ii) Take two such complexes $C^*_+$ and $\tilde{C}^*_+$, each with the same structure as in (i) (including subcomplexes $C^*_-$ and $\tilde{C}^*_-$). Suppose that we have maps between them as in Setup \ref{th:null}(ii), which preserve the subcomplexes. If we are given mutually compatible homotopy trivializations $[\chi]$ and $[\tilde{\chi}]$ on those subcomplexes, which are compatible in the sense of Setup \ref{th:null}(ii), then (by a straightforward computation) the lifts from \eqref{eq:lift-bv} fit into a commutative diagram
\begin{equation} \label{eq:lift-lift}
\xymatrix{
H^*(C_+) \ar[rr] && H^*(\tilde{C}_+)
\\
\ar@{-->}[u] H^*(Q) \ar[rr] && H^*(\tilde{Q}). \ar@{-->}[u]
}
\end{equation}
\end{setup}

Mapping this abstract framework onto our specific situation is fairly straightforward. We had $H^*(Q) \iso H^*(S^1) \otimes \mathit{HF}^{*+2}(\mu,\epsilon)$, with the BV operator on the quotient as described in Lemma \ref{th:mclean-example}. We arranged that $\delta|C^*_-$ vanished, and made the trivial choice $\chi = 0$. The partial splitting of \eqref{eq:proj-s1} is the map which fits into the commutative diagram
\begin{equation}
\xymatrix{
\mathit{HF}^*(E,\gamma_+,\epsilon) \\
\ar[u]
H^1(S^1) \otimes \mathit{HF}^{*}(\mu,\epsilon) \ar[rr]^-{\iso} &&
H^0(S^1) \otimes \mathit{HF}^{*}(\mu,\epsilon) \ar@{-->}[ull]
}
\end{equation}
where the vertical arrow is the restriction of the lift \eqref{eq:lift-bv}, and the horizontal arrow is the BV operator. Inspection shows that this indeed recovers the formula \eqref{eq:bar-x}. As in the discussion following Setup \ref{th:null}, one can use the fact that the homotopy trivialization is canonical to show that the partial splitting is independent of all auxiliary choices.

We will also need a modified version of Proposition \ref{th:partial-splitting}, whose proof combines much of the technical material about Floer cohomology that has been mentioned so far:

\begin{proposition} \label{th:partial-splitting-2}
Take the situation from Lemma \ref{th:gamma-wall}. Assume that $\partial M$ is a contact circle bundle, and that $\epsilon \in (1,2)$. Additionally, assume that $E$ has a symplectic Calabi-Yau structure, such that the induced symplectic Calabi-Yau structure on $M$ is as in Example \ref{th:contact-circle-bundle}(ii). Then \eqref{eq:proj-s1} has a partial splitting, defined on the subspace of $H^0(S^1) \otimes \mathit{HF}^{*+2}(\mu,\epsilon)$ where $\ast \leq -2m$.
\end{proposition}

\begin{proof}
One can set up a Morse-Bott chain complex which computes $\mathit{HF}^*(E,\gamma_-,\epsilon)$ exactly as in Proposition \ref{th:bv-vanish-2}. Namely, generators consist of constant orbits, together with a Morse-Bott nondegenerate submanifold $F$ which is a copy of $\partial M$ (located in some fibre of $\pi$). Moreover, the Hamiltonian and almost complex structure can be taken to be time-independent. Then, the argument used there shows that the Morse-Bott analogue of the BV operator vanishes on all cochains of degree $\leq 2-2m$.

What does this mean for our original complex $\mathit{CF}^*(E,H_-)$? It is related to the Morse-Bott version by a chain homotopy equivalence, which is a version of a continuation map, hence comes with a secondary version as in \eqref{eq:cont-maps}, expressing its compatibility with BV operators. Using Setup \ref{th:truncation}(ii), we obtain a nullhomotopy for the BV operator on the correspondingly truncated subcomplex:
\begin{equation} \label{eq:chi-kill}
\chi: \mathit{CF}^{*}(E,H_-)^{\leq 2-2m} \longrightarrow \mathit{CF}^{*-2}(E,H_-)^{\leq 2-2m}.
\end{equation}
Now take \eqref{eq:x-cochain}, where $\ast \leq -2m$, and replace \eqref{eq:bar-x} with \eqref{eq:corrected-bar}, which provides the desired splitting; this makes sense because $dx$ is a cocycle in $\mathit{CF}^*(E,H_-)$ of degree $\ast+2 \leq 2-2m$, hence lies in the domain of definition of \eqref{eq:chi-kill}. 
\end{proof}

As before, one can interpret this as an instance of a more general setup:

\begin{setup} \label{th:final-setup}
(i) Consider a variation on the situation of Setup \ref{th:split-setup}(i), where the homotopy trivialization exists only on $(C^*_-)^{\leq j}$, for some $j$. Then, \eqref{eq:corrected-bar} defines a partial lift, as in \eqref{eq:lift-bv} but defined only on $H^*(Q)$ for $\ast \leq j-1$.

(ii) There is a corresponding variation of Setup \ref{th:split-setup}(ii), with two homotopy trivializations defined on $(C^*_-)^{\leq j}$ respectively $(\tilde{C}^*_-)^{\leq j}$. Then, one gets a diagram \eqref{eq:lift-lift} in degrees $\ast \leq j-1$.
\end{setup}

To apply this to our case, we first used a Morse-Bott formalism to put ourselves in a situation where part of the BV operator can be explicitly computed (allowing an obvious choice of partial homotopy trivialization), and then used continuation maps to carry over that trivialization to the general case. In principle, part (ii) of Setup \ref{th:final-setup} can be used to show that the partial splitting from Proposition \ref{th:partial-splitting} is canonical, but that would require additional arguments not carried out here.

\begin{example} \label{th:lift-1}
In the situation of Example \ref{th:anticanonical-lefschetz-pencil}, we can use \eqref{eq:trivial-monodromy-2} for $\epsilon \in (1,2)$. The splitting from Proposition \ref{th:partial-splitting-2} then takes $1 \in H^0(M) \iso \mathit{HF}^{2-2m}(\mu,\epsilon)$ to an element of $\mathit{HF}^{-2m}(E,\gamma_+,\epsilon)$.
\end{example}

\section{Translations at infinity\label{subsec:translations}}
We remain in the same class of manifolds (Setup \ref{th:setup-e}), but now set up Floer cohomology differently. First of all, we follow the version for symplectic automorphisms, rather than the Hamiltonian one. More importantly, the perturbations involved have different behaviour at infinity, following \cite{seidel12b}.

\subsection{Geometric data}
As usual, we begin by assembling the basic geometric ingredients, this time starting with the class of symplectic automorphisms that are allowed.

\begin{setup} \label{th:setup-t}
(i) We consider exact symplectic automorphisms $\phi: E \rightarrow E$ with the following properties. There is an open subset $U \subset E$ of the same kind as in Setup \ref{th:setup-e}, such that
\begin{equation} \label{eq:fibrewise-phi}
\pi(\phi(x)) = \phi_{\bC}(\pi(x)) \quad \text{for $x \in U$.}
\end{equation}
Here, $\phi_{\bC}$ is a symplectic automorphism of the base, which is the identity outside a compact subset. Moreover, if we restrict $\phi$ to a fibre $E_y$, $|y| \gg 0$, then it should be the identity near $\partial E_y$.

(ii) As usual, if $E$ has a symplectic Calabi-Yau structure, we assume that $\phi$ is a graded symplectic automorphism (then, the same holds for its restriction to any fibre $E_y$, $|y| \gg 0$).

(iii) We use functions $H \in \smooth(E,\bR)$ such that
\begin{equation} \label{eq:h-translation}
\left\{
\begin{aligned}
& H(\Theta(y,x)) = H_{\bC}(y) + \epsilon\rho_M(x) \quad \text{for $x$ near $\partial M$,} \\
& H(\Psi(p,q,x)) = \delta \,\mathrm{re}(e^{p+iq}) + H_M(x) \quad \text{for $p \gg 0$.}
\end{aligned}
\right.
\end{equation}
Here, $\epsilon$ is as in \eqref{eq:no-reeb}, and $\delta \in \bR \setminus \{0\}$. $H_{\bC}$ is a function on the base such that $H_{\bC}(y) = \delta\, \mathrm{re}(y)$ outside a compact subset, and $H_M$ is as in Setup \ref{th:mapping-torus-setup}.

(iv) We use compatible almost complex structures $J$ on $E$ of the following kind. Near the boundary, we impose the same condition \eqref{eq:theta-star-j} as before. At points $(p,q,x)$, where $p \gg 0$ and $-\pi \leq q \leq \pi$,
\begin{equation} \label{eq:two-families}
\Psi^*J = \begin{cases} 
i \times J_{M,\mathrm{re}(e^{p+iq})}^+ & \text{for $0 \leq q \leq \pi$}, \\
i \times J_{M,\mathrm{re}(e^{p+iq})}^- & \text{for $-\pi \leq q \leq 0$.}
\end{cases}
\end{equation}
Here, $J_{M,r}^\pm$ are two auxiliary families of compatible almost complex structures on the fibre, both parametrized by $r \in \bR$ and belonging to the class \eqref{eq:j-convex}, and additionally satisfying
\begin{equation}
\begin{cases}
J_{M,r}^- = J_{M,r}^+ & \text{if $r \gg 0$,} \\
J_{M,r}^- = \mu_*J_{M,r}^+ & \text{if $r \ll 0$.}
\end{cases}
\end{equation}
\end{setup}

The assumption \eqref{eq:fibrewise-phi} is quite restrictive, since it implies that $\phi$ must commute with symplectic parallel transport between the fibres. In combination with the other condition, this implies that
\begin{equation} \label{eq:phi-first-trivialization}
\phi(\Theta(y,x)) = \Theta(\phi_{\bC}(y),x) \quad \text{for $x$ close to $\partial M$.}
\end{equation}
Similarly, we have
\begin{equation} \label{eq:conjugate-mu}
\phi(\Psi(p,q,x)) = \Psi(p,q,\phi_M(x)) \quad \text{if $p \gg 0$},
\end{equation}
where $\phi_M$ is an exact symplectic automorphism of $M$, which is the identity near $\partial M$, and which must commute with $\mu$.

\begin{example} \label{th:global-monodromy}
Fix a function on $\bC$, which equals $\pi |y|^2$ outside a compact subset, and pull it back to $E$. Then, its flow for time $1$ gives a symplectic automorphism within the class defined above. We denote it by $\nu$, and call it the {\em global monodromy}. The associated automorphism of the fibre is the monodromy $\mu$.
\end{example}

The class of almost complex structures introduced above has the property that $\pi$ is $J$-holomorphic outside a compact subset. Moreover, because of \eqref{eq:two-families}, the induced almost complex structures on the fibres $E_y$, $|y| \gg 0$, are locally constant under parallel transport in imaginary direction (this is sketched in Figure \ref{fig:j}).
\begin{figure}
\begin{centering}
\begin{picture}(0,0)%
\includegraphics{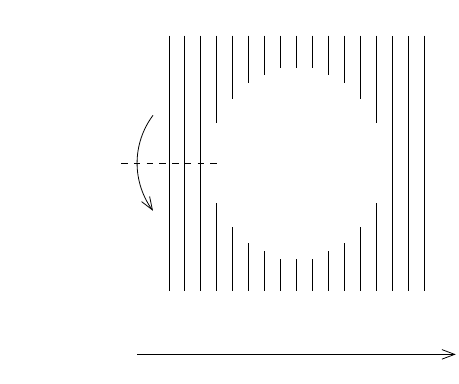}%
\end{picture}%
\setlength{\unitlength}{3355sp}%
\begingroup\makeatletter\ifx\SetFigFont\undefined%
\gdef\SetFigFont#1#2#3#4#5{%
  \reset@font\fontsize{#1}{#2pt}%
  \fontfamily{#3}\fontseries{#4}\fontshape{#5}%
  \selectfont}%
\fi\endgroup%
\begin{picture}(4302,3594)(-389,-3634)
\put(3001,-3561){\makebox(0,0)[lb]{\smash{{\SetFigFont{10}{12.0}{\rmdefault}{\mddefault}{\updefault}{\color[rgb]{0,0,0}$r = \mathrm{re}(y)$}%
}}}}
\put(2251,-211){\makebox(0,0)[lb]{\smash{{\SetFigFont{10}{12.0}{\rmdefault}{\mddefault}{\updefault}{\color[rgb]{0,0,0}$J_{M,r}^+$}%
}}}}
\put(2251,-3061){\makebox(0,0)[lb]{\smash{{\SetFigFont{10}{12.0}{\rmdefault}{\mddefault}{\updefault}{\color[rgb]{0,0,0}$J_{M,r}^-$}%
}}}}
\put(-374,-1411){\makebox(0,0)[lb]{\smash{{\SetFigFont{10}{12.0}{\rmdefault}{\mddefault}{\updefault}{\color[rgb]{0,0,0}monodromy $\mu$}%
}}}}
\end{picture}%
\caption{\label{fig:j}}
\end{centering}
\end{figure}%

\subsection{The compactness argument}
Given $\phi$, choose $H = (H_t)$ and $J = (J_t)$, with the same periodicity condition as in \eqref{eq:phi-periodicity} (this makes sense because the relevant classes of Hamiltonians and almost complex structures are invariant under $\phi$). The solutions of \eqref{eq:periodic-y} in $\scrL_\phi$, or equivalently the fixed points of $\phi_H^1 \circ \phi$, will always be contained in a compact subset of $E \setminus \partial E$. For generic choice of $(H_t)$, these fixed points will also be nondegenerate; assume from now on that this is the case. Consider solutions $u: \bR^2 \rightarrow E$ of Floer's equation \eqref{eq:floer} with periodicity conditions \eqref{eq:floer-2}, and with limits $x_{\pm}$ as usual.

\begin{lemma} \label{th:b-energy}
There is a compact subset of $E$, such that for any solution $u$ and any $s$ such that $u(s,1)$ lies outside that subset,
\begin{equation}
\int_{[0,1]} |\partial_s u(s,t)|^2 \, \mathit{dt} \geq \delta^2.
\end{equation}
\end{lemma}

\begin{proof}
For $x$ outside a compact subset, 
\begin{equation} \label{eq:no-asymptotic-fixed-points}
\mathrm{dist}(x, (\phi_H^1 \circ \phi)(x)) \geq \delta, 
\end{equation}
where the distance is with respect to the metric associated to any almost complex structure in our class (recall that the fibres of $\pi$ are compact, so ``outside a compact subset'' means going to infinity in base direction). For essentially the same reason, $\int_{[0,1]} |\partial_s u(s,t)| \mathit{dt} \geq \delta$.
\end{proof}

\begin{lemma} \label{th:base-convexity}
There is a constant $B>0$, such that $|\mathrm{re}(\pi(u))| \leq B$ for all solutions $u$.
\end{lemma}

\begin{proof}
Write $v = \pi(u)$. On the subset where $\mathrm{re}(v)$ is large, it satisfies the equation
\begin{equation}
\left\{
\begin{aligned}
& \partial_s v + i(\partial_t v - i \delta) = 0, \\
& v(s,t+1) = v(s,t) + i\delta.
\end{aligned}
\right.
\end{equation}
Hence, $\mathrm{re}(v)$ is harmonic, so that we can apply the maximum principle.
\end{proof}

\begin{lemma} \label{th:fibre-convexity}
There is a neighbourhood $V \subset M$ of $\partial M$, such that the image of any solution $u$ is disjoint from $\Theta(\bC \times V)$.
\end{lemma}

This is again a maximum principle argument, but now used in fibre direction.

\begin{lemma} \label{th:gromov}
There is a constant $C$ such that $\|du\|_{\infty} \leq C$ for any Floer trajectory.
\end{lemma}

\begin{proof}
This is a simple Gromov compactness argument, following \cite[Proposition 5.1]{seidel12b}, and we will only give limited details. Suppose that the result is false. After passing to a subsequence, $\|\partial_s u_k\|$ goes to infinity. Take points $z_k = (s_k,t_k)$ where $|\partial_s u_k|$ reaches its maximum. 

{\em If the $u_k(z_k)$ remain inside a compact subset of $E$}, one can rescale locally near $z_k$, and get a non-constant pseudo-holomorphic plane as a limit. Since we have a priori bounds on the energy, this plane has finite energy, hence extends to a pseudo-holomorphic sphere, in contradiction to exactness. {\em Suppose on the other hand that the $u_k(z_k)$ do not remain inside a compact subset of $E$.} In view of Lemma \ref{th:base-convexity}, one can pass to a subsequence and then assume that $\mathrm{im}(\pi(u_k(z_k)))$ converges to either $+\infty$ or $-\infty$. Then, after a translation in imaginary direction over the base, the same argument as before applies, except that the limit lies in $\bC \times M$, and is pseudo-holomorphic for an almost complex structure $i \times J_{M,\mathrm{re}(y)}^{\pm}$, where $y$ is the coordinate on $\bC$ (it is here that we use the specific property of our almost complex structures indicated in Figure \ref{fig:j}).
\end{proof}

\begin{proposition} \label{th:floer-bound}
There is a compact subset of $E \setminus \partial E$ which contains all Floer trajectories.
\end{proposition}

\begin{proof}
Assume that this is not true. By Lemmas \ref{th:base-convexity} and \ref{th:fibre-convexity}, there must be a sequence $(u_k)$ of solutions, such that the maximal value of $|\mathrm{im}(\pi(u_k))|$ goes to infinity. Because we have an absolute bound on $|\partial_su_k|$ from Lemma \ref{th:gromov}, $u_k(s,1)$ has to spend an increasingly large interval (in $s$) inside the region where Lemma \ref{th:b-energy} applies. But that contradicts the a priori bound on the energy.
\end{proof}

\subsection{Floer cohomology and its properties}
Having obtained Proposition \ref{th:floer-bound}, it is now a familiar process to set up Floer complexes $\mathit{CF}^*(\phi,H)$, whose cohomology we denote by $\mathit{HF}^*(\phi,\delta,\epsilon)$. The appropriate version of \eqref{eq:phi-poincare-duality} is 
\begin{equation} \label{eq:translation-duality}
\mathit{HF}^*(\phi^{-1},-\delta,-\epsilon) \iso \mathit{HF}^{2n-*}(\phi,\delta,\epsilon)^\vee.
\end{equation}

\begin{lemma} \label{th:vanishing-homology}
For small $\delta > 0$ and $\epsilon > 0$, $\mathit{HF}^*(\mathit{id},\delta,\epsilon) \iso H^*(E,\{\mathrm{re}(\pi) \ll 0\})$.
\end{lemma}

We will not explain the proof of this, which can be done by reduction to Morse theory as in Proposition \ref{th:bv-vanish} (if $\pi$ is a Lefschetz fibration, $H^*(E,\{\mathrm{re}(\pi) \ll 0\})$ is concentrated in degree $n$, and has one generator for each critical point; one can in fact arrange that the underlying chain complex has the same property, and thereby give an elementary proof of Lemma \ref{th:vanishing-homology} for that special case).

The dependence of Floer cohomology on the parameters $(\delta,\epsilon)$ is a more interesting issue than before. One can show (for instance, using parametrized moduli spaces) that only the sign of $\delta$ matters. In fact, one also has isomorphisms
\begin{equation} \label{eq:reverse-delta}
\mathit{HF}^*(\phi,-\delta,\epsilon) \iso \mathit{HF}^*(\phi,\delta,\epsilon),
\end{equation}
but not canonical ones. To see that, note that the sign of $\delta$ depends on our identification of the base with the standard complex plane $\bC$. Reversing that identification ($y \mapsto -y$) takes $\delta$ to $-\delta$. To construct \eqref{eq:reverse-delta}, one has to rotate the plane by some amount in $\pi + 2\pi\bZ$, and different choices yield different maps \eqref{eq:reverse-delta}. 

\begin{remark} \label{th:canonical-automorphism}
One way to think of the ambiguity in \eqref{eq:reverse-delta} is as follows. $\mathit{HF}^*(\phi,\delta,\epsilon)$ carries a canonical automorphism, induced by a full rotation of the plane (whose angle is thought of as an additional parameter), or equivalently by conjugation with the global monodromy. Then, \eqref{eq:reverse-delta} is unique up to composition with powers of that automorphism.

Another version of the same explanation goes as follows: on can allow translations over the base in any direction, corresponding to a parameter $\delta \in \bC^*$ (for compatibility with our previous notation, the translation would have to be by $i\delta$). For fixed $\epsilon$, these more general Floer cohomology groups $\mathit{HF}^*(\phi,\delta,\epsilon)$ would be canonical locally trivial in $\delta$, which means that they would form a local system over $\bC^*$. The previously mentioned automorphism is just the holonomy of the local system, and \eqref{eq:reverse-delta} would be a parallel transport map between two different fibres.
\end{remark}

\begin{lemma} \label{th:independence-of-epsilon}
For any $\epsilon_- < \epsilon_+$ which satisfy \eqref{eq:no-reeb}, one has $\mathit{HF}^*(\phi,\delta,\epsilon_-) \iso \mathit{HF}^*(\phi,\delta,\epsilon_+)$.
\end{lemma}

\begin{proof}
%
Our strategy follows that of Lemma \ref{th:no-wall}. We enlarge the given $E$ by attaching a conical piece to the boundary of each fibre, as in \eqref{eq:attach-cone}:
\begin{equation} \label{eq:attach-cone-2}
\hat{E} = E \cup_{\partial E} (\bC \times [1,C] \times \partial M),
\end{equation}
where $\partial E$ is identified with $\bC \times \partial M$ using parallel transport, which means \eqref{eq:first-trivialization}. Extend the given $\phi$ to an automorphism $\hat\phi$ of $\hat{E}$ by setting it equal to $\phi_{\bC} \times \mathit{id}_{[1,C] \times \partial M}$ on the conical part, where $\phi_{\bC}$ is as in \eqref{eq:phi-first-trivialization}.

Suppose that $H_-$ is the function used to define the chain complex $\mathit{CF}^*(\phi,H_-)$ underlying $\mathit{HF}^*(\phi,\delta,\epsilon_-)$. We extend it to a function $\hat{H}_+$ on $\hat{E}$ by a fibrewise version of \eqref{eq:extend-hamiltonian}:
\begin{equation}
\hat{H}_{t,+}(r,y,x) = H_{\bC,t}(y) + C h(C^{-1}r)
\end{equation}
where $h$ is as in \eqref{eq:turn}, and the $H_{\bC,t}$ are functions as in \eqref{eq:h-translation}. It is unproblematic to show that $\mathit{CF}^*(\hat\phi,\hat{H}_+)$ computes $\mathit{HF}^*(\phi,\delta,\epsilon_+)$. At this point, we want to be more specific about the choice of function for the original Floer cohomology group. Using the fact that the group of compactly supported symplectic automorphisms of $\bC$ is connected, one can find a time-dependent $H_{\bC}$ such that 
\begin{equation}
(\phi_{H_{\bC}}^1 \circ \phi_{\bC})(y) = y + i\delta
\end{equation}
is simply a translation. In that case, $\mathit{CF}^*(\hat\phi,\hat{H}_+)$ has the same generators as $\mathit{CF}^*(\phi,H_-)$, and an easy maximum principle argument (in fibre direction) shows that the differentials also coincide.
\end{proof} 

\begin{remark}
Alternatively, one can avoid the use of specific choices of $H_{\bC}$, and argue as in Lemma \ref{th:wall}. Namely, suppose that there is only one $\epsilon \in (\epsilon_-,\epsilon_+)$ such that $\epsilon R_{\partial M}$ has $1$-periodic orbits. Then, there is a long exact sequence
\begin{equation} \label{eq:les-sequence-2}
\cdots \rightarrow \mathit{HF}^*(\phi,\delta,\epsilon_-) \longrightarrow \mathit{HF}^*(\phi,\delta,\epsilon_+) \longrightarrow H(Q^*) \rightarrow \cdots
\end{equation}
One can arrange that $Q^* \iso Q^*_{\bC} \otimes Q^*_{M}$, where the first factor is a version of the Floer chain complex for $\phi^1_{H_{\bC}} \circ \phi_{\bC}$, which is then shown to be acyclic (this last step would again use the connectedness of the group of compactly supported symplectic automorphisms of $\bC$).
\end{remark}

Our final topic is the relation between the two versions of Floer cohomology on $E$. Because of the different classes of perturbations used, this is not quite straightforward. We only need a partial result:

\begin{lemma} \label{th:rotation-translation}
Take $\gamma_- \in (2\pi(k-1),2\pi k)$, and arbitrary $\delta,\epsilon$. Let $\nu$ be the global monodromy. Then there is a canonical map
\begin{equation} \label{eq:rotation-translation}
\mathit{HF}^*(E,\gamma_-,\epsilon) \longrightarrow \mathit{HF}^*(\nu^k,\delta,\epsilon).
\end{equation}
\end{lemma}

To clarify the grading conventions: the source in \eqref{eq:rotation-translation} carries its natural $\bZ$-grading (as a form of Hamiltonian Floer cohomology, on a manifold with a symplectic Calabi-Yau structure). For the target, we use the structure of the global monodromy $\nu$ as a graded symplectic automorphism which comes from its original construction by a flow (see Example \ref{th:global-monodromy}; this is parallel to (i) in Example \ref{th:boundary-twist}).

\begin{proof} 
The argument follows Lemma \ref{th:gamma-wall} closely, and we will only give a few details. Set $\gamma_+ = 2\pi k$, and choose a function $h_+$ as in \eqref{eq:rotation-function}. We can assume that $\nu^k$ is defined as the time-one map of the function $F(x) = h_+(|\pi(x)|^2/2)$. When choosing a perturbation $H_+$ to be used to define $\mathit{CF}^*(\nu^k,H_+)$, we can assume that $H_{t,+}(x) = H_M(x)$ close to $Z = \{|\pi(x)| = 2\}$.

A suitable ``barrier'' argument shows that (for small $|\delta|$) the generators corresponding to fixed points lying on the inside of $Z$ form a subcomplex of $\mathit{CF}^*(\nu^k,H_+)$. Moreover, that subcomplex can be identified with $\mathit{CF}^*(E,H_-)$, where
\begin{equation} \label{eq:minusham}
H_{t,-}(x) = 
\begin{cases} 
H_{t,+}(\phi_F^{-t}(x)) + F(x) & |\pi(x)| \leq 2, \\
H_M(x) + h_-(|\pi(x)|^2/2) & |\pi(x)| \geq 2. 
\end{cases}
\end{equation}
Here, $h_-$ is the function that agrees with $h_+$ on $[0,2]$, and satisfies $h_-'(a) = \gamma_-$ for all $a \geq 2$. But \eqref{eq:minusham} defines $\mathit{HF}^*(E,\gamma_-,\epsilon)$.
\end{proof}

\subsection{Additional remarks}
One can show that the map \eqref{eq:rotation-translation} fits into a long exact sequence
\begin{equation} \label{eq:r-t-1}
\cdots \rightarrow \mathit{HF}^*(E,\gamma_-,\epsilon) \longrightarrow \mathit{HF}^*(\nu^k,\delta,\epsilon) \longrightarrow \mathit{HF}^{*+2k}(\mu^{k+1},\epsilon) \rightarrow \cdots
\end{equation}
It seems likely (but we have not checked the details) that there is a similar long exact sequence
\begin{equation} \label{eq:r-t-2}
\cdots \rightarrow \mathit{HF}^*(\nu^k,\delta,\epsilon) \longrightarrow \mathit{HF}^*(E,\gamma_+,\epsilon) \longrightarrow \mathit{HF}^{*+2k+1}(\mu^{k+1},\epsilon) \rightarrow \cdots
\end{equation}
where $\gamma_+ \in (2\pi k, 2\pi (k+1))$. One piece of supporting evidence is that the combination of the two sequences above (in the appropriate order) is compatible with \eqref{eq:2pi-sequence}. Repeated use of those two sequences gives a step-by-step ``decomposition'' of all the Floer cohomology groups $\mathit{HF}^*(E,\gamma,\epsilon)$ and $\mathit{HF}^*(\nu^k,\delta,\epsilon)$. Alternatively, one can approach the same idea through spectral sequences:

\begin{lemma} \label{th:rotate-ss}
For any $\epsilon$ and any $\gamma \in (2\pi k, 2\pi(k+1))$, $k \geq 0$, there is a spectral sequence converging to $\mathit{HF}^*(E,\gamma,\epsilon)$, whose starting page is
\begin{equation} \label{eq:rotate-ss}
E_1^{pq} = \begin{cases} 
H^{q+1}(E,\{\mathrm{re}(\pi) \ll 0\}) & p = 1, \\
\mathit{HF}^q(\mu^{-\lfloor p/2 \rfloor},\epsilon) & -2k \leq p \leq 0, \\
0 & \text{otherwise.}
\end{cases}
\end{equation}
\end{lemma}

\begin{lemma} \label{th:nu-ss}
For any $k \geq 0$, $\delta>0$, and $\epsilon$, there is a spectral sequence converging to $\mathit{HF}^*(\nu^k,\delta,\epsilon)$, whose starting page is
\begin{equation} \label{eq:nu-ss}
E_1^{pq} = \begin{cases}
H^{q+1}(E,\{\mathrm{re}(\pi) \ll 0\}) & p = 1, \\
\mathit{HF}^q(\mu^{-\lfloor p/2 \rfloor},\epsilon) & -2k+1 \leq p \leq 0, \\
0 & \text{otherwise.}
\end{cases}
\end{equation}
\end{lemma}

In fact, these spectral sequences are more straightforward to prove than \eqref{eq:r-t-1} or \eqref{eq:r-t-2}. The first one is a weaker version of that in \cite{mclean12}, and the second one is not substantially different. We will not explain them further, but we do want to show one application.

\begin{example} \label{th:negative-degrees}
Take the situation arising from an anticanonical Lefschetz pencil (see Examples \ref{th:contact-circle-bundle}(ii) and \ref{th:anticanonical-lefschetz-pencil}, but where $m = 0$). By repeatedly applying Lemma \ref{th:tau-shift}, we get
\begin{equation} \label{eq:r-shift}
\mathit{HF}^*(\mu^j, \epsilon) \iso \mathit{HF}^{*-2j}(M,\epsilon-j).
\end{equation}
Moreover, for any $\epsilon>0$, there is a Morse-Bott spectral sequence converging to $\mathit{HF}^*(M,\epsilon)$ (generalizing Proposition \ref{th:bv-vanish-2}), with
\begin{equation} \label{eq:bottish}
E_1^{pq} = \begin{cases} H^{p+q}(M) & p = 0, \\
H^{q-p}(\partial M) & -\lfloor \epsilon \rfloor \leq p < 0, \\
0 & \text{otherwise.}
\end{cases}
\end{equation}
In particular, $\mathit{HF}^*(M,\epsilon)$ is always concentrated in degrees $\ast \geq 0$. 

Choose some $k>0$, and take $\epsilon > k$. From \eqref{eq:r-shift} and \eqref{eq:bottish}, it follows that
\begin{equation}
\left\{
\begin{aligned}
& \text{$\mathit{HF}^*(M, \epsilon)$ is concentrated in degrees $\ast \geq 0$,} \\
& \text{$\mathit{HF}^*(\mu,\epsilon) \iso \mathit{HF}^{*-2}(M, \epsilon-1)$ is concentrated
in degrees $\ast \geq 2$,} \\
& \dots \\
& \text{$\mathit{HF}^*(\mu^k,\epsilon) \iso \mathit{HF}^{*-2k}(M,\epsilon-k)$ is concentrated in degrees $\ast \geq 2k$.}
\end{aligned}
\right.
\end{equation}
By feeding that into \eqref{eq:rotate-ss} and \eqref{eq:nu-ss}, it follows that $\mathit{HF}^*(E,\gamma,\epsilon)$ and $\mathit{HF}^*(\nu^k,\delta,\epsilon)$ are all concentrated in nonnegative degrees 
(moreover, all the contributions coming from closed Reeb orbits on $\partial M$ land in degrees $\geq 2$).
\end{example}

To summarize, we have two infinite sequences of ``closed string'' Floer cohomology groups associated to any Lefschetz fibration. The first of these sequences, $\mathit{HF}^*(E,\gamma,\epsilon)$ for $\gamma \in (2\pi k, 2\pi(k+1))$, comes with an additional dependence on $\epsilon$ (one can remove that dependence by passing to the direct limit $\epsilon \rightarrow \infty$, as in the definition of symplectic cohomology; the price to pay is that the resulting groups will typically be infinite-dimensional). Moreover, these groups carry BV operators. The second sequence is $\mathit{HF}^*(\nu^k,\delta,\epsilon)$. These groups are independent of $\epsilon$, do not have BV operators, but come with canonical automorphisms (mentioned in Remark \ref{th:canonical-automorphism}).

\begin{remark}
It is possible to interpret this situation in terms of mirror symmetry. Consider a smooth projective variety $A$, together with a section $r$ of its anticanonical bundle, which gives rise to a smooth divisor $B = r^{-1}(0)$. We'll discuss the analogues of the two Floer cohomology groups mentioned above, but in reverse order.

Take the sheaves $\Omega^i_A(kB)$ of algebraic $i$-forms with poles of order at most $k$ along $B$, and form
\begin{equation} \label{eq:hh-b}
\bigoplus_i H^{*+i}(A,\Omega^i_A(kB))
\end{equation}
For $k = 0$, this is the Hochschild homology of $A$. It admits a BV type operator, induced by the de Rham differential, but that is known to vanish. Derived autoequivalences act on Hochschild homology, and in particular, one gets a distinguished automorphism from the action of the Serre functor. More concretely, this automorphism is given by multiplying with the exponential of the class of the canonical bundle in $H^1(A,\Omega^1_A)$. For $k>0$, \eqref{eq:hh-b} does not carry a natural BV type operator, since the de Rham differential increases pole order (but one can still define a canonical automorphism, as before).

Instead, consider the subsheaves
\begin{equation} \label{eq:log}
\Omega^i_A((k-1)B +\log B) \subset \Omega^i_A(kB)
\end{equation}
of those differential forms $\alpha$ such that both $\alpha$ and $d\alpha$ have poles of order $\leq k$ along $B$. This gives rise to another sequence of graded groups,
\begin{equation} \label{eq:coherent-or-not}
\bigoplus_i H^{*+i}(A,\Omega^i_A((k-1)B + \log B)),
\end{equation}
which do carry BV operators. To be more precise, the situation we have just considered (with a smooth $B$) is mirror to working with a Lefschetz fibration which has closed fibres; hence, there is no parameter corresponding to our $\epsilon$ here.
\end{remark}

We also would like to consider a variant of Lemma \ref{th:nu-ss}, where some of the columns in the $E_1$ page have already been combined, as in \eqref{eq:r-t-1}. Again, the proof is omitted.

\begin{lemma} \label{th:3e}
There is a spectral sequence converging to $\mathit{HF}^*(\nu^2,\delta,\epsilon)$, with
\begin{equation} \label{eq:3cols}
E_1^{pq} = \begin{cases} \mathit{HF}^q(E,\gamma,\epsilon) & p = 0, \text{ where $\gamma>0$ is small,} \\
\mathit{HF}^{q+1}(\mu,\epsilon) \oplus \mathit{HF}^{q+2}(\mu,\epsilon) & p = -1, \\
\mathit{HF}^{q+1}(\mu^2,\epsilon) & p = -2, \\
0 & \text{otherwise.}
\end{cases}
\end{equation}
This spectral sequence is compatible with the $\bZ/2$-action on $\mathit{HF}^*(\nu^2,\delta,\epsilon)$ from Remark \ref{th:circle-and-discrete-actions}: the induced $\bZ/2$-action on \eqref{eq:3cols} is trivial except in the $p = -2$ column, where it is the corresponding action on $\mathit{HF}^*(\mu^2,\epsilon)$.
\end{lemma}

\begin{example} \label{th:trivial-involution}
Take the situation from Example \ref{th:negative-degrees} with $\epsilon>2$. Then the $p = -2$ column contributes only in positive degrees. Hence, 
the $\bZ/2$-action on $\mathit{HF}^0(\nu^2,\delta,\epsilon)$ is trivial, at least if we assume that $\mathrm{char}(\bK) \neq 2$.
\end{example}

\section{Open-closed string maps}

We now combine fixed point Floer cohomology, in the version considered in Section \ref{subsec:translations}, with its counterpart for Lagrangian submanifolds. The relation between the two theories, together with the preceding Floer cohomology computations, leads directly to our main results (Theorems \ref{th:main} and \ref{th:fano}).

\subsection{Lagrangian Floer cohomology}
We will continue to work in the situation of Setups \ref{th:setup-e} and \ref{th:setup-t}, with the following additional geometric ingredient.

\begin{setup} \label{th:setup-l}
(i) We consider oriented exact Lagrangian submanifolds $L \subset E \setminus \partial E$ such that $\pi|L$ is proper, and
\begin{equation} \label{eq:lambda}
\pi(L) = \{\text{compact subset}\} \cup \{\mathrm{re}(y) \gg 0, \;\; \mathrm{im}(y) = o\} \subset \bC
\quad\text{for some $o \in \bR$.}
\end{equation}

(ii) If $E$ comes with a Calabi-Yau structure, we will assume that $L$ is a graded Lagrangian submanifold. 

(iii) Independently, one may want to assume that $L$ comes with a {\em Spin} structure.
\end{setup}

At any point $x \in L$ outside a compact subset, we know that $TE_x^v$ is a symplectic subspace, and hence that 
\begin{equation} \label{eq:n-1}
\mathrm{rank}(D\pi_x|TL_x) = n - \mathrm{dim}(TL_x \cap \mathit{TE}_x^v) \geq 1.
\end{equation}
By combining this with \eqref{eq:lambda}, one sees that equality holds in \eqref{eq:n-1}, hence that $D\pi_x|TL_x: TL_x \rightarrow \bR$ is onto. This implies that there is a unique closed Lagrangian submanifold $L_M \subset M \setminus \partial M$ such that, if $y = e^{p+iq}$ with $q \in (-\pi/2,\pi/2)$, $\mathrm{re}(y) \gg 0$ and $\mathrm{im}(y) = o$, then
\begin{equation} \label{eq:l-m}
L \cap \pi^{-1}(y) = \Psi((p,q) \times L_M).
\end{equation}
In other words, at infinity $L$ is fibered over a horizontal half-infinite path, with each fibre being equal to $L_M$. Note that $L_M$ is exact, and inherits an orientation (as well as the other structure mentioned in (ii) and (iii) above, whenever that exists on $L$).

Given two such submanifolds $L_0,L_1$, such that the corresponding numbers \eqref{eq:lambda} satisfy $o_1 - o_0 \neq 0$, there is a well-defined Floer cohomology $\mathit{HF}^*(L_0,L_1)$. To make the setup formally parallel to fixed point Floer cohomology, we will also introduce a perturbed version $\mathit{HF}^*(L_0,L_1,\delta,\epsilon)$, which reduces to the previous one for $\delta = \epsilon = 0$, and is defined under the assumption that
\begin{equation} \label{eq:delta-inequality}
o_1-o_0 - \delta \neq 0.
\end{equation}
Floer cohomology is invariant under automorphisms of $E$,
\begin{equation}
\mathit{HF}^*(\phi(L_0),\phi(L_1),\delta,\epsilon) \iso \mathit{HF}^*(L_0,L_1,\delta,\epsilon).
\end{equation}
The analogue of \eqref{eq:phi-poincare-duality} says that
\begin{equation} \label{eq:lagrangian-floer-duality}
\mathit{HF}^*(L_1,L_0,-\delta,-\epsilon) \iso \mathit{HF}^{n-*}(L_0,L_1,\delta,\epsilon)^\vee.
\end{equation}

\begin{lemma} \label{th:lag-1}
Suppose that $\delta_{\pm}$ are such that $o_1 - o_0 - \delta_{\pm}$ have the same sign. Then $\mathit{HF}^*(L_0,L_1,\delta_-,\epsilon) \iso \mathit{HF}^*(L_0,L_1,\delta_+,\epsilon)$.
\end{lemma}

\begin{lemma} \label{th:lag-2}
For any $\epsilon_{\pm}$, $\mathit{HF}^*(L_0,L_1,\delta,\epsilon_-) \iso \mathit{HF}^*(L_0,L_1,\delta,\epsilon_+)$.
\end{lemma}

The two Lemmas above are analogues of Lemma \ref{th:no-wall}. Note that this time, there are no ``forbidden values'' of $\epsilon$. The passage through the unique ``forbidden value'' of $\delta$ is described by the following result, whose proof we omit:

\begin{lemma} \label{th:lag-3}
Take $\delta_- < o_1-o_0 < \delta_+$. Then there is a long exact sequence
\begin{equation}
\cdots \rightarrow \mathit{HF}^*(L_0,L_1,\delta_-,\epsilon) \longrightarrow
\mathit{HF}^*(L_0,L_1,\delta_+,\epsilon) \longrightarrow \mathit{HF}^*(L_{0,M},L_{1,M}) \rightarrow \cdots
\end{equation}
where the third group is the Floer cohomology of the associated closed Lagrangian submanifolds \eqref{eq:l-m} in the fibre $M$.
\end{lemma}

Finally, we have an analogue of Proposition \ref{th:bv-vanish}, also given here without proof:

\begin{lemma} \label{th:albers}
For $\delta>0$, $\mathit{HF}^*(L,L,\delta,\epsilon) \iso H^*(L)$.
\end{lemma}

\begin{remark} \label{th:spin}
In general, $\mathit{HF}^*(L_0,L_1,\delta,\epsilon)$ is a $\bZ/2$-graded space over a coefficient field $\bK$ of characteristic $2$. In the situation from Setup \ref{th:setup-l}(ii), one can obtain $\bZ$-graded Floer cohomology groups; and in that of Setup \ref{th:setup-l}(iii), an arbitrary coefficient field $\bK$ can be allowed (see \cite[Chapter 8]{fooo} or \cite[Section 12]{seidel04}).
\end{remark}

To define the chain complex $\mathit{CF}^*(L_0,L_1,H)$ underlying $\mathit{HF}^*(L_0,L_1,\delta,\epsilon)$, one chooses functions $H = (H_t)$ and almost complex structures $J = (J_t)$ as in Setup \ref{th:setup-t}, but this time parametrized by $t \in [0,1]$. On the path space 
\begin{equation} \label{eq:path-space}
\scrL_{L_0,L_1} = \{x: [0,1] \rightarrow E \,:\, x(0) \in L_0, \; x(1) \in L_1\},
\end{equation}
one has a counterpart of \eqref{eq:phi-action}:
\begin{equation} \label{eq:lagrangian-action}
A_{L_0,L_1,H}(x) = \Big( \int_0^1 -x^*\theta_E + H_t(x(t)) \, dt \Big) + G_{L_1}(x(1)) - G_{L_0}(x(0)),
\end{equation}
where the $G_{L_k}$ are functions such that $dG_{L_k} = \theta_E|L_k$. The critical points are solutions $x$ of \eqref{eq:periodic-y} in \eqref{eq:path-space}, hence correspond to points of $\phi_H^1(L_0) \cap L_1$. The condition \eqref{eq:delta-inequality} implies that all such $x$ are contained in a compact subset of $E$. Moreover, for generic choice of $H$, the $x$ will be nondegenerate. Assuming this to be the case, one considers solutions $u: \bR \times [0,1] \rightarrow E$ of \eqref{eq:floer}, with boundary conditions
\begin{equation} \label{eq:floer-boundary}
u(s,0) \in L_0, \;\; u(s,1) \in L_1.
\end{equation}

There is an analogue of Proposition \ref{th:floer-bound} in this context, with similar strategy of proof. The rest of the construction of $\mathit{HF}^*(L_0,L_1,\delta,\epsilon)$ follows the classical theory for closed Lagrangian submanifolds \cite{floer88c}.

\begin{remark} \label{th:referee}
It may make sense to mention the one point where the situation here differs from that in Proposition \ref{th:floer-bound} (by being easier, in fact). There is no analogue of \eqref{eq:no-asymptotic-fixed-points} in the present context, but that is in fact unnecessary: because of \eqref{eq:floer-boundary}, a bound on $\|du_\infty\|$ (obtained as in Lemma \ref{th:gromov}) implies that the image of $u$ is contained in a region $\{\mathrm{im}(\pi) \leq D\}$. (The reader may also want to consult the proof of \cite[Proposition 5.1]{seidel12b}, even though that result itself does not apply here, because of the different choice of inhomogeneous terms.)
\end{remark}

\subsection{The formalism\label{subsec:formalism}}
We will now introduce the additional structures on Floer cohomology which underlie our main argument (for the moment, only the formal aspects will be considered; discussion of their actual construction will take place later on, in Section \ref{subsec:cr}).

Given Lagrangian submanifolds $(L_0,L_1,L_2)$ and constants $(\delta_0,\delta_1)$, $(\epsilon_0,\epsilon_1)$ such that all Floer cohomology groups involved are well-defined, the triangle product (due to Donaldson) is a map
\begin{equation} \label{eq:triangle-product}
\mathit{HF}^*(L_1,L_2,\delta_1,\epsilon_1) \otimes \mathit{HF}^*(L_0,L_1,\delta_0,\epsilon_0) \longrightarrow \mathit{HF}^*(L_0,L_2,\delta_0+\delta_1,\epsilon_0+\epsilon_1).
\end{equation}
This satisfies an appropriate associativity condition. As an application, fix Lagrangian submanifolds $L_1,\dots,L_m$ with pairwise different constants $o_1, \dots, o_m \in \bR$. To these, one can associate an algebra $A$ over \eqref{eq:semisimple}, namely
\begin{equation} \label{eq:directed-1}
A = R \oplus \bigoplus_{i<j} \mathit{HF}^*(L_i,L_j),
\end{equation}
or equivalently
\begin{equation} \label{eq:directed-2}
e_j A e_i = \begin{cases} \mathit{HF}^*(L_i,L_j) & i<j, \\ \bK & i=j, \\ 0 & i>j.
\end{cases}
\end{equation}
Recall that here, the Floer cohomology groups under discussion are $\mathit{HF}^*(L_i,L_j,\delta,\epsilon)$ with $\delta = \epsilon = 0$ (in view of Lemmas \ref{th:lag-1} and \ref{th:lag-2}, one could equivalently use any $\delta$ which lies on the same side of $o_j-o_i$ as the origin, and any $\epsilon$, but that it not helpful for thinking about the products). The nontrivial part of the algebra structure of $A$ consists of the products $\mathit{HF}^*(L_j,L_k) \otimes \mathit{HF}^*(L_i,L_j) \rightarrow \mathit{HF}^*(L_i,L_k)$ for $i<j<k$, which are special cases of \eqref{eq:triangle-product}. Next, take an automorphism $\phi$ as in Setup \ref{th:setup-t}, and constants $\delta$, $\epsilon$ satisfying
\begin{equation} \label{eq:lambda-difference}
\delta \neq o_i - o_j \quad \text{for all $i,j \in \{1,\dots,m\}$  (including $i = j$, which means $\delta \neq 0$).}
\end{equation}
Generalizing \cite[Section 6.3]{seidel12b}, one associates to this a bimodule $P_{\phi,\delta,\epsilon}$ over $A$, namely:
\begin{equation}
P_{\phi,\delta,\epsilon} = \bigoplus_{i,j} \mathit{HF}^*(\phi(L_i),L_j,\delta,\epsilon).
\end{equation}
These structures have cochain level refinements: an $A_\infty$-algebra $\scrA$ (following Fukaya), and $A_\infty$-bimodules $\scrP_{\phi,\delta,\epsilon}$ over $\scrA$. There are corresponding refinements of the properties of Lagrangian Floer cohomology mentioned above. Namely, there is a quasi-isomorphism (in fact, if the choices are suitably coordinated, an isomorphism)
\begin{equation} \label{eq:dual-bimodule}
\scrP_{\phi,\delta,\epsilon} \htp \scrP_{\phi^{-1},-\delta,-\epsilon}^\vee[-n].
\end{equation}

\begin{lemma}
The quasi-isomorphism type of $\scrP_{\phi,\delta,\epsilon}$ is independent of $\epsilon$. It is also locally constant in $\delta$, within the allowed range \eqref{eq:lambda-difference}.
\end{lemma}

This is a refinement of the previously mentioned invariance properties of Lagrangian Floer cohomology (Lemmas \ref{th:lag-1}, \ref{th:lag-2}). We omit the proof.

%
\begin{example}
In our eventual application, $(L_1,\dots,L_m)$ will be a basis of Lefschetz thimbles. In that case, if one takes $\delta \gg 0$, then $\scrP_{\mathit{id},\delta,\epsilon}$ is quasi-isomorphic to the diagonal bimodule $\scrA$ \cite[Corollary 6.1]{seidel12b}; and if one takes $\delta \ll 0$, $\scrP_{\mathit{id},\delta,\epsilon}$ is quasi-isomorphic to the shifted dual diagonal bimodule $\scrA^\vee[-n]$ \cite[Corollary 6.2]{seidel12b}.
\end{example}

For $\phi$ as before, take $(\delta,\epsilon)$ so that $\mathit{HF}^*(\phi,\delta,\epsilon)$ is well-defined. Given any $L$, one has a canonical map relating the two kinds of Floer cohomology groups:
\begin{equation} \label{eq:oc0}
\mathit{HF}^*(\phi(L_i),L_i,\delta,\epsilon) \longrightarrow \mathit{HF}^{*+n}(\phi,\delta,\epsilon).
\end{equation}
After applying \eqref{eq:translation-duality}, \eqref{eq:lagrangian-floer-duality} and changing notation from $(\phi,\delta,\epsilon)$ to $(\phi^{-1},-\delta,\epsilon)$, the map dual to \eqref{eq:oc0} can be written as
\begin{equation} \label{eq:dual-oc0}
\mathit{HF}^*(\phi,\delta,\epsilon) \longrightarrow \mathit{HF}^*(\phi(L_i),L_i,\delta,\epsilon).
\end{equation}

\begin{example}
Let's specialize to $\phi = \mathit{id}$ and $\delta >0$. Applying Lemmas \ref{th:albers} and \ref{th:vanishing-homology}, one sees that \eqref{eq:oc0} reduces to a map
\begin{equation} \label{eq:i-1}
H^*(L_i) \longrightarrow H^{*+n}(E,\{\mathrm{re}(\pi) \ll 0\}),
\end{equation}
Let's instead take $\delta < 0$. Then, because of \eqref{eq:lagrangian-floer-duality} and \eqref{eq:translation-duality}, the map \eqref{eq:oc0} looks like this:
\begin{equation} \label{eq:i-2}
H_*(L_i) \longrightarrow H_*(E,\{\mathrm{re}(\pi) \ll 0\}),
\end{equation}
and its dual \eqref{eq:dual-oc0} is correspondingly
\begin{equation} \label{eq:i-3}
H^*(E,\{\mathrm{re}(\pi) \ll 0\}) \longrightarrow H^*(L_i).
\end{equation}
Unsurprisingly, \eqref{eq:i-1} and \eqref{eq:i-3} can be identified with the ordinary pushforward and restriction maps in cohomology; in particular, \eqref{eq:i-3} factors through $H^*(E)$.
\end{example}

As before, \eqref{eq:oc0} is part of a more complicated structure, the {\em open-closed string map}
\begin{equation} \label{eq:open-closed-string-map}
\mathit{HH}_*(\scrA,\scrP_{\phi,\delta,\epsilon}) \longrightarrow \mathit{HF}^{*+n}(\phi,\delta,\epsilon),
\end{equation}
which is defined for all $(\delta,\epsilon)$ such that both sides make sense. Let's use the same Floer-theoretic duality as before, as well as \eqref{eq:second-hochschild} and \eqref{eq:dual-bimodule}. After applying that, and suitably adjusting notation, one can write the dual of \eqref{eq:open-closed-string-map} as
\begin{equation} \label{eq:dual-open-closed-string-map}
\mathit{HF}^*(\phi,\delta,\epsilon) \longrightarrow
\mathit{HH}^*(\scrA,\scrP_{\phi,\delta,\epsilon}) = H^*(\hom_{[\scrA,\scrA]}(\scrA, \scrP_{\phi,\delta,\epsilon})).
\end{equation}
%

\begin{remark}
The composition of \eqref{eq:open-closed-string-map} and \eqref{eq:dual-open-closed-string-map} yields a map from Hochschild homology to Hochschild cohomology (of degree $n$), for any bimodule $\scrP_{\phi,\delta,\epsilon}$. Taking into account \eqref{eq:serre-hochschild} (since $\scrA$ is directed, it is homologically smooth), one can write that as
\begin{equation} \label{eq:our-map}
H^*(\hom_{[\scrA,\scrA]}((\scrA^\vee)^{-1}, \scrP_{\phi,\delta,\epsilon}))
\longrightarrow
H^{*+n}(\hom_{[\scrA,\scrA]}(\scrA,\scrP_{\phi,\delta,\epsilon})).
\end{equation}
Generally speaking, one natural source of such homomorphisms are bimodule maps $\scrA \rightarrow (\scrA^\vee)^{-1}$ of degree $n$, or equivalently bimodule maps $\scrA^\vee[-n] \rightarrow \scrA$. At least in the case of Lefschetz fibrations, to be discussed later (and assuming $\delta \gg 0$), it seems likely that \eqref{eq:our-map} is induced by the bimodule map we have called $\rho$ (Lemma \ref{th:old}).
\end{remark}

We will also need a variation of \eqref{eq:open-closed-string-map} which involves two automorphisms $\phi_k$ ($k = 0,1$), with their associated bimodules $\scrP_{\phi_k,\delta_k,\epsilon_k}$. For simplicity, we assume that $\epsilon_0/\delta_0 = \epsilon_1/\delta_1$. Set
\begin{equation} \label{eq:product-phi}
\left\{
\begin{aligned}
& \phi = \phi_1\,\phi_0, \\
& \delta = \delta_0 + \delta_1, \\
& \epsilon = \epsilon_0 + \epsilon_1.
\end{aligned}
\right.
\end{equation}
We suppose that $\mathit{HF}^*(\phi,\delta,\epsilon)$ is defined. Then, the new version of the open-closed string map has the form
\begin{equation} \label{eq:double-open-closed}
\mathit{HH}_*(\scrA, \scrP_{\phi_1,\delta_1,\epsilon_1} \otimes_{\scrA} \scrP_{\phi_0,\delta_0,\epsilon_0}) \longrightarrow \mathit{HF}^{*+n}(\phi,\delta,\epsilon).
\end{equation}
This is the same kind of construction as the ``two-pointed open-closed string maps'' of \cite[Section 5.6]{ganatra13}. In parallel with \eqref{eq:dual-open-closed-string-map}, but this time using \eqref{eq:general-hochschild}, one can write the dual of \eqref{eq:double-open-closed} as
\begin{equation} \label{eq:dual-double}
\mathit{HF}^*(\phi,\delta,\epsilon) \longrightarrow H^*\big(\mathit{hom}_{[\scrA,\scrA]}(\scrP_{\phi_0^{-1},-\delta_0,-\epsilon_0}, \scrP_{\phi_1,\delta_1,\epsilon_1})\big).
\end{equation}

\begin{remark} \label{th:consider-double}
One can think of \eqref{eq:double-open-closed} as follows. Write $\scrP_k = \scrP_{\phi_k,\delta_k,\epsilon_k}$ and $\scrP = \scrP_{\phi,\delta,\epsilon}$, assuming that the latter is defined. The triangle products
\begin{equation}
\begin{aligned} &
\mathit{HF}^*(\phi_1(L_j),L_k,\delta_1,\epsilon_1) \otimes \mathit{HF}^*(\phi_0(L_i),L_j,\delta_0,\epsilon_0) \\
& \quad \iso \mathit{HF}^*(\phi_1(L_j),L_k,\delta_1,\epsilon_1) \otimes
\mathit{HF}^*(\phi(L_i),\phi_1(L_j),\delta_0,\epsilon_0)
\longrightarrow \mathit{HF}^*(\phi(L_i),L_k,\delta,\epsilon)
\end{aligned}
\end{equation}
can be lifted to an $A_\infty$-bimodule homomorphism
\begin{equation} \label{eq:bimodule-tensor}
\scrP_1 \otimes_{\scrA} \scrP_0 \longrightarrow \scrP.
\end{equation}
One can reduce \eqref{eq:double-open-closed} to \eqref{eq:open-closed-string-map} for $\phi$, by writing it as the composition
\begin{equation}
\mathit{HH}_*(\scrA, \scrP_1 \otimes_{\scrA} \scrP_0) \longrightarrow 
\mathit{HH}_*(\scrA,\scrP) \longrightarrow \mathit{HF}^{*+n}(\phi,\delta,\epsilon),
\end{equation}
where the first map is induced by \eqref{eq:bimodule-tensor} (\cite{ganatra13} uses the same idea to relate ``two-pointed open-closed string maps'' to ordinary ``one-pointed'' ones). However, we will prefer a direct definition of \eqref{eq:double-open-closed}, since that is a little simpler, and we have no use for the maps \eqref{eq:bimodule-tensor} by themselves.
\end{remark}
 
\begin{remark} \label{th:swap-factors}
Let's slightly change notation, and assume that $(\phi,\delta,\epsilon)$ is given, with the property that $\scrP_{\phi,\delta,\epsilon}$ and $\mathit{HF}^*(\phi^2,2\delta,2\epsilon)$ are well-defined. Then, \eqref{eq:double-open-closed} specializes to yield a map
\begin{equation} \label{eq:square-map}
\mathit{HH}_*(\scrA,\scrP_{\phi,\delta,\epsilon}^{\otimes_{\scrA} 2}) \longrightarrow \mathit{HF}^{*+n}(\phi^2,2\delta,2\epsilon).
\end{equation}
The left hand side carries a natural $\bZ/2$-action, as discussed in Section \ref{subsec:hochschild-homology}; so does the right hand side, because it is the Floer cohomology of a square (Remark \ref{th:circle-and-discrete-actions}); and \eqref{eq:square-map} is compatible with those actions (essentially, because Figure \ref{fig:disc2} is rotationally symmetric).
\end{remark}

\subsection{Cauchy-Riemann equations\label{subsec:cr}}
The previously mentioned algebraic structures belong to a kind of TCFT (Topological Conformal Field Theory), which means a general framework of algebraic operations parametrized by Riemann surfaces. We will now outline some of the ingredients of this construction. Compared to similar ideas in the literature, the main difference is that we allow nontrivial monodromy around interior punctures on the Riemann surfaces. In other respects, the exposition here is actually somewhat restrictive: for instance, unlike \cite{abouzaid-seidel07}, we only allow closed one-forms on our Riemann surfaces.

\begin{setup} \label{th:setup-s}
(i) Fix constants $(\epsilon,\delta)$. A {\em worldsheet} is a connected non-compact Riemann surface $S$, possibly with boundary, with the following properties and additional data.

We assume that there is a compactification $\bar{S} = S \sqcup \Sigma$, obtained by adding a finite set of points. We divide the points of this finite set $\Sigma$ into closed string (interior) and open string (boundary) ones, and (independently, and arbitrarily) into inputs and outputs, denoting the respective subsets by $\Sigma^{\mathit{cl}/\mathit{op},\mathit{in}/\mathit{out}}$. These should come with local holomorphic coordinates (tubular and strip-like ends), of the form
\begin{equation} \label{eq:ends}
\begin{cases}
\epsilon_\zeta: (-\infty,0] \times S^1 \longrightarrow S, & \zeta \in \Sigma^{\mathit{cl},\mathit{out}}, \\
\epsilon_\zeta: [0,\infty) \times S^1 \longrightarrow S, & \zeta \in \Sigma^{\mathit{cl},\mathit{in}}, \\
\epsilon_\zeta: (-\infty,0] \times [0,1] \longrightarrow S, & \zeta \in \Sigma^{\mathit{op},\mathit{out}}, \\
\epsilon_\zeta: [0,\infty) \times [0,1] \longrightarrow S, & 
\zeta \in \Sigma^{\mathit{op},\mathit{in}}.
\end{cases}
\end{equation}
Let $\tilde{S} \rightarrow S$ be the universal covering. We fix lifts of \eqref{eq:ends}, of the form
\begin{equation}
\begin{cases}
\tilde\epsilon_\zeta: (-\infty,0] \times \bR \longrightarrow \tilde{S}, & \zeta \in \Sigma^{\mathit{cl},\mathit{out}}, \\
\tilde\epsilon_\zeta: [0,\infty) \times \bR \longrightarrow \tilde{S}, & \zeta \in \Sigma^{\mathit{cl},\mathit{in}}, \\
\tilde\epsilon_\zeta: (-\infty,0] \times [0,1] \longrightarrow \tilde{S}, & \zeta \in \Sigma^{\mathit{op},\mathit{out}}, \\
\tilde\epsilon_\zeta: [0,\infty) \times [0,1] \longrightarrow \tilde{S}, & 
\zeta \in \Sigma^{\mathit{op},\mathit{in}}.
\end{cases}
\end{equation}

Write $\Gamma \iso \pi_1(S)$ for the covering group of $\tilde{S} \rightarrow S$. We want to have a homomorphism 
\begin{equation} \label{eq:monodromy}
\Phi: \Gamma \longrightarrow \mathit{Symp}(E),
\end{equation}
which takes values in the subgroup of those symplectic automorphisms of $E$ described in Setup \ref{th:setup-t}. In particular, given $\zeta \in \Sigma^{\mathit{cl}}$, passing from $\tilde\epsilon_\zeta(s,t+1)$ to $\tilde\epsilon_\zeta(s,t)$ amounts to acting by an element of $\Gamma$, which determines an automorphism $\phi_\zeta$ by \eqref{eq:monodromy}.

To each point $\tilde{z} \in \partial \tilde{S}$ we want to associate a Lagrangian submanifold $L_{\tilde{z}}$ as in Setup \ref{th:setup-l}, in a way which is locally constant in $z$, and compatible with \eqref{eq:monodromy}:
\begin{equation} \label{eq:equivariant-lagrangian}
L_{\gamma(\tilde{z})} = \Phi(\gamma)(L_{\tilde{z}}).
\end{equation}
In particular, given $\zeta \in \Sigma^{\mathit{op}}$, we have distinguished Lagrangian submanifolds $(L_{\zeta,0}, L_{\zeta,1})$ which are associated to the points $(\tilde{\epsilon}_{\zeta}(s,0), \tilde\epsilon_\zeta(s,1))$. By definition, these come with real numbers \eqref{eq:lambda}, which we denote by $(o_{\zeta,0},o_{\zeta,1})$.

Finally, we want $S$ to carry a closed one-form $\beta_S$ with $\beta_S|\partial S = 0$, and which over each end \eqref{eq:ends} satisfies 
\begin{equation}
\epsilon_\zeta^* \beta_S = \beta_\zeta \mathit{dt},
\end{equation}
where the $\beta_\zeta$ are constants. Set $\delta_\zeta = \beta_\zeta \delta$, $\epsilon_\zeta = \beta_\zeta \epsilon$. We require that for $\zeta \in \Sigma^{\mathit{cl}}$, $\epsilon_\zeta$ should satisfy \eqref{eq:no-reeb}, and $\delta_\zeta \neq 0$; while for $\zeta \in \Sigma^{\mathit{op}}$, the pair $(L_{\zeta,0},L_{\zeta,1})$ and the number $\delta_\zeta$ must satisfy \eqref{eq:delta-inequality}.

(ii) Given $S$, we consider families $J_S = (J_{S,\tilde{z}})$ of almost complex structures in the class from Setup \ref{th:setup-t}, para\-me\-trized by $\tilde{z} \in \tilde{S}$. These should be equivariant with respect to \eqref{eq:monodromy}, and over the strip-like ends they should be invariant under translation in the first variable. Concretely, this means that there are families $J_\zeta = (J_{\zeta,t})$ such that
\begin{equation} \label{eq:j-zeta}
J_{S,\tilde\epsilon_\zeta(s,t)} = J_{\zeta,t};
\end{equation}
and if $\zeta \in \Sigma^{\mathit{cl}}$, the associated family $J_\zeta$ satisfies a periodicity condition as in \eqref{eq:phi-periodicity} with respect to $\phi_\zeta$.

(iii) We will equip $S$ with an {\em inhomogeneous term} $K_S$, which is a one-form on $\tilde{S}$ with values in the space of functions on $E$. More precisely, for any tangent vector $\xi$, $K_S(\xi) \in \smooth(E,\bR)$ belongs to the class from Setup \ref{th:setup-t} for the constants $(\beta_S(\xi) \delta, \beta_S(\xi)\epsilon )$. This family must be equivariant with respect to \eqref{eq:monodromy}, and if $\xi$ is tangent to $\partial S$, then $K_S(\xi)|L_{\tilde{z}} = 0$. Finally, over each end, $\epsilon_\zeta^*K_S$ is invariant under translation in the first variable, and satisfies $(\epsilon_\zeta^* K_S)(\partial_s) = 0$. In analogy with \eqref{eq:j-zeta}, one can therefore write
\begin{equation} \label{eq:k-zeta}
\tilde\epsilon_\zeta^* K_S = H_{\zeta,t} \, \mathit{dt};
\end{equation}
and if $\zeta \in \Sigma^{\mathit{cl}}$, $H_\zeta$ again satisfies \eqref{eq:phi-periodicity}.
\end{setup}

We usually refer to a pair $(J_S,K_S)$ as a {\em perturbation datum}, since it specifies a particular pertubed Cauchy-Riemann equation on $S$. Concretely, this is an equation for maps
\begin{equation} \label{eq:tilde-u}
\left\{\begin{aligned}
& \tilde{u}: \tilde{S} \longrightarrow E, \\
& \tilde{u}(\gamma(\tilde{z})) = \Phi(\gamma)(\tilde{u}(\tilde{z})), \\
& \tilde{u}(\tilde{z}) \in L_{\tilde{z}} \quad \text{for $\tilde{z} \in \partial \tilde{S}$},
\end{aligned}
\right.
\end{equation}
namely:
\begin{equation} \label{eq:generalized-floer}
(d\tilde{u} - Y_{K_S})^{0,1} = 0.
\end{equation}
Here, $Y_{K_S}$ is the one-form on $\tilde{S}$ with values in $\smooth(TE)$ associated to $K_S$ (by passing from functions to Hamiltonian vector fields). We use the complex structure on $S$, and the almost complex structures $J_{S,\tilde{z}}$, to form the $(0,1)$-part of the linear map $d\tilde{u} - Y_{K_S}: TS_{\tilde{z}} \longrightarrow TE_{\tilde{u}(\tilde{z})}$. One checks readily that \eqref{eq:generalized-floer} is invariant under the twisted $\Gamma$-periodicity condition from \eqref{eq:tilde-u}. Over the ends, which means for $u_\zeta(s,t) = \tilde{u}(\tilde{\epsilon}_\zeta(s,t))$, \eqref{eq:generalized-floer} reduces to equations for suitable Floer trajectories:
\begin{equation} \label{eq:end-floer}
\left\{
\begin{aligned}
& u_\zeta(s,k) \in L_{\zeta,k} \quad \text{for $k = 0,1$, if $\zeta \in \Sigma^{\mathit{op}}$}, \\
& u_\zeta(s,t) = \phi_\zeta(u_\zeta(s,t+1)) \quad \text{if $\zeta \in \Sigma^{\mathit{cl}}$}, \\
& \partial_s u_\zeta + J_{\zeta,t}(\partial_t u_\zeta - X_{H_\zeta,t}) = 0.
\end{aligned}
\right.
\end{equation}
It therefore makes sense to impose convergence conditions
\begin{equation} \label{eq:generalized-limit}
\textstyle \lim_{s \rightarrow \pm\infty} u_\zeta(s,\cdot) = x_\zeta.
\end{equation}
The necessary compactness argument for the moduli space of solutions of \eqref{eq:generalized-floer} combines: an overall energy bound; a maximum principle argument in horizontal direction over the base (as in Lemma \ref{th:base-convexity}); and the following idea. Let's fix a metric on $S$ which is standard over the ends. Then there is a constant $c$ such that each point of $S$ either (i) lies within distance $c$ of $\partial S$, or else (ii) lies on a tubular end (the image of $\epsilon_\zeta$ for some $\zeta \in \Sigma^{\mathit{cl}}$). For preimages $\tilde{z}$ of points as in (i), one can bound $|\pi(\tilde{u}(\tilde{z}))|$ as in Lemma \ref{th:gromov}; whereas in case (ii), one uses an argument as in Lemma \ref{th:b-energy}. Taken together, these ingredients yield an analogue of Proposition \ref{th:floer-bound}.  Transversality is straightforward, given the freedom to choose $J_S$ and $K_S$.

Counting solutions of \eqref{eq:generalized-floer}, \eqref{eq:generalized-limit} for a single Riemann surface $S$ yields a map between Floer cochain spaces, which induces a cohomology level map of degree $n(-\chi(\bar{S}) + |\Sigma^{\mathit{out,op}}| + 2|\Sigma^{\mathit{out,cl}}|)$:
\begin{equation} \label{eq:tqft}
\xymatrix{
\displaystyle
\bigotimes_{\zeta \in \Sigma^{\mathit{in},\mathit{op}}} \mathit{HF}^*(L_{\zeta,0},L_{\zeta,1},\delta_\zeta,\epsilon_\zeta) \otimes \bigotimes_{\zeta \in \Sigma^{\mathit{in},\mathit{cl}}} \mathit{HF}^*(\phi_\zeta,\delta_\zeta,\epsilon_\zeta) 
\ar[d] \\ 
\displaystyle
\bigotimes_{\zeta \in \Sigma^{\mathit{out},\mathit{op}}} \mathit{HF}^*(L_{\zeta,0},L_{\zeta,1},\delta_\zeta,\epsilon_\zeta) \otimes
\bigotimes_{\zeta \in \Sigma^{\mathit{out},\mathit{cl}}}
\mathit{HF}^*(\phi_\zeta,\delta_\zeta,\epsilon_\zeta).
}
\end{equation}

\begin{remark} 
A suitable analogue of Remarks \ref{th:signs-and-grading} and \ref{th:spin} applies. In general, \eqref{eq:tqft} will be a $\bZ/2$-graded map, and defined over a coefficient field with $\mathrm{char}(\bK) = 2$. However, if one assumes that: $E$ carries a symplectic Calabi-Yau structure; that \eqref{eq:monodromy} comes with a lift to the graded symplectic automorphism group; and that all boundary conditions are graded Lagrangian submanifolds; then \eqref{eq:tqft} will have the specified degree with respect to the $\bZ$-gradings of the Floer cohomology groups involved. Moreover, if the Lagrangian submanifolds are {\em Spin}, arbitrary $\bK$ are allowed.
\end{remark}

So far, what we have explained is a version of the TQFT formalism for Floer theory. For the extension to a TCFT, which is relevant for us, one has to include families of Riemann surfaces. Strictly speaking, a TCFT framework would have to allow a class of such families which is large enough to form a chain level model for the homology of the relevant compactified moduli spaces of Riemann surfaces, and to allow appropriate composition (operad) structures. For our purpose, only certain specific families are needed; there, the underlying analysis remains the same as for the TQFT case (except that the gluing theory has to be carried out in a parametrized sense, which is something that can be regarded as well-understood).

\subsection{Moduli spaces}
As before, we start with $(L_1,\dots,L_m)$ having pairwise different constants $(o_1,\dots,o_m)$. For $1 \leq i < j \leq m$, fix the additional data (functions $H_{ij}$, and almost complex structures $J_{ij}$) needed to define the Floer cochain complexes $\mathit{CF}^*(L_i,L_j,H_{ij})$. To obtain an $A_\infty$-algebra structure, consider worldsheets $S$ which are discs with $d+1 \geq 3$ boundary punctures. We equip the components of $\partial S$ with boundary conditions $L_{i_0},\dots,L_{i_d}$, for $i_0 < \cdots < i_d$ (in positive order around $\partial \bar{S}$). The one-forms $\beta_S$ are taken to be zero throughout. The perturbation data $(J_S,K_S)$ need to be chosen so that their restriction to the strip-like ends, as in \eqref{eq:j-zeta} and \eqref{eq:k-zeta}, reduces to the choices previously made to define Floer theory. Finally, the choices need to depend smoothly on the moduli of $S$, and there are conditions about their behaviour as this surface degenerates. We will not give any details, referring instead to \cite[Section 9]{seidel04}. The outcome are maps
\begin{equation}
\mathit{CF}^*(L_{i_{d-1}},L_{i_d},H_{i_{d-1}i_d}) \otimes \cdots
\otimes \mathit{CF}^*(L_{i_0},L_{i_1},H_{i_0i_1}) \longrightarrow
\mathit{CF}^{*+2-d}(L_{i_0},L_{i_d},H_{i_0i_d}),
\end{equation}
which satisfy the $A_\infty$-associativity equations. To define the cochain level structure underlying \eqref{eq:directed-1}, one extends these operations in the unique way to a strictly unital $A_\infty$-structure on 
\begin{equation}
\scrA = R \oplus \bigoplus_{i<j} \mathit{CF}^*(L_i,L_j,H_{ij}).
\end{equation}

Fix an automorphism $\phi$. To define $\scrP = \scrP_{\phi,\delta,\epsilon}$, one chooses Hamiltonians $H_{\phi,ij}$ and almost complex structures $J_{\phi,ij}$ for all $(i,j)$, and then sets
\begin{equation} \label{eq:p-space}
\scrP = \bigoplus_{ij} \mathit{CF}^*(\phi(L_i),L_j,H_{\phi,ij}),
\end{equation}
with $\mu^{0;1;0}_{\scrP}$ the direct sum of Floer differentials. Fundamentally, the moduli spaces of Riemann surfaces which define the higher operations $\mu^{s;1;r}_{\scrP}$ are the same (Stasheff polyhedra) as for $\mu^{r+1+s}_{\scrA}$. However, the other data they carry are different, and we find it convenient to think of the Riemann surfaces themselves in a slightly different way, namely to write them as
\begin{equation} \label{eq:s-strip}
S = T \setminus \{\zeta_1,\dots,\zeta_r,\zeta_1',\dots,\zeta_s'\},
\end{equation}
where: $T = \bR \times [0,1]$; the $\zeta_i \in \bR \times \{0\} \subset \partial T$ are increasing; and the $\zeta_i' \in \bR \times \{1\} \subset \partial T$ are decreasing. The one-forms $\beta_S$ should be equal to $\mathit{dt}$ on the region where $|s| \gg 0$, and should vanish sufficiently close to the $\zeta_i$ and $\zeta_i'$. The boundary conditions are 
\begin{equation} \label{eq:boundary-0}
(\phi(L_{i_0}),\dots,\phi(L_{i_r})) \quad \text{along $\bR \times \{0\}$}
\end{equation}
 (as $s$ increases, and where $i_0 < \cdots < i_r$), respectively 
\begin{equation} \label{eq:boundary-1}
(L_{i_0'},\dots,L_{i_s'}) \quad \text{along $\bR \times \{1\}$} 
\end{equation}
(as $s$ decreases, and where again $i_0' < \cdots < i_s'$). The inhomogeneous terms and almost complex structures are determined by those chosen for \eqref{eq:p-space} over the ends where $|s| \gg 0$, and by the choices made for $\scrA$ on the remaining ends. As we vary over all possible $S$, the outcome are operations
\begin{equation}
\begin{CD}
\mathit{CF}^*(L_{i_{s-1}'},L_{i_s'},H_{i_{s-1}'i_s'}) \otimes \cdots \otimes \mathit{CF}^*(L_{i_0'},L_{i_1'},H_{i_0'i_1'}) \\ 
\otimes \, \mathit{CF}^*(\phi(L_{i_r}),L_{i_0'},H_{\phi,i_r i_0'})\, \otimes \\ 
\mathit{CF}^*(L_{i_{r-1}},L_{i_r},H_{i_{r-1}i_r}) \otimes \cdots \otimes
\mathit{CF}^*(L_{i_0},L_{i_1},H_{i_0i_1}) 
\\
@VVV
\\
\mathit{CF}^{*+1-r-s}(\phi(L_{i_0}),L_{i_s'},H_{\phi,i_0 i_s'}),
\end{CD}
\end{equation}
which (extended over the units in $\scrA$ in the obvious way) constitute $\mu_{\scrP}^{s;1;r}$. It may be convenient to think of \eqref{eq:s-strip} as glued together from two half-strips $S_{\pm} = \{(s,t) \in S \;:\; \pm t \leq \half\}$, each of which carries boundary conditions taken from the $L_k$, but where the gluing identifies the two target spaces $E$ using $\phi$. This amounts to rewriting the relevant equation \eqref{eq:generalized-floer} as two parts $u_\pm: S_\pm \rightarrow E$, joined by a ``seam'' $u_+(s,\half) = \phi(u_-(s,\half))$. Mathematically, this does not change anything, but it can be useful as an aid to the intuition, since it makes the connection with quilted Floer cohomology \cite{wehrheim-woodward10} (Figure \ref{fig:half-strips}).
\begin{figure}
\begin{centering}
\begin{picture}(0,0)%
\includegraphics{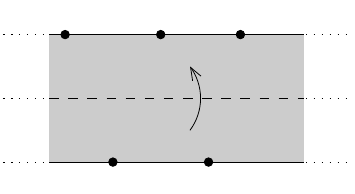}%
\end{picture}%
\setlength{\unitlength}{3355sp}%
\begingroup\makeatletter\ifx\SetFigFont\undefined%
\gdef\SetFigFont#1#2#3#4#5{%
  \reset@font\fontsize{#1}{#2pt}%
  \fontfamily{#3}\fontseries{#4}\fontshape{#5}%
  \selectfont}%
\fi\endgroup%
\begin{picture}(3324,1819)(3739,-2459)
\put(5701,-1411){\makebox(0,0)[lb]{\smash{{\SetFigFont{10}{12.0}{\rmdefault}{\mddefault}{\updefault}{\color[rgb]{0,0,0}$\phi$}%
}}}}
\put(4201,-2386){\makebox(0,0)[lb]{\smash{{\SetFigFont{10}{12.0}{\rmdefault}{\mddefault}{\updefault}{\color[rgb]{0,0,0}$L_{i_0}$}%
}}}}
\put(6151,-811){\makebox(0,0)[lb]{\smash{{\SetFigFont{10}{12.0}{\rmdefault}{\mddefault}{\updefault}{\color[rgb]{0,0,0}$L_{i_0'}$}%
}}}}
\put(5101,-2386){\makebox(0,0)[lb]{\smash{{\SetFigFont{10}{12.0}{\rmdefault}{\mddefault}{\updefault}{\color[rgb]{0,0,0}$L_{i_1}$}%
}}}}
\put(6001,-2386){\makebox(0,0)[lb]{\smash{{\SetFigFont{10}{12.0}{\rmdefault}{\mddefault}{\updefault}{\color[rgb]{0,0,0}$L_{i_2}$}%
}}}}
\put(3901,-811){\makebox(0,0)[lb]{\smash{{\SetFigFont{10}{12.0}{\rmdefault}{\mddefault}{\updefault}{\color[rgb]{0,0,0}$L_{i_3'}$}%
}}}}
\put(4576,-811){\makebox(0,0)[lb]{\smash{{\SetFigFont{10}{12.0}{\rmdefault}{\mddefault}{\updefault}{\color[rgb]{0,0,0}$L_{i_2'}$}%
}}}}
\put(5476,-811){\makebox(0,0)[lb]{\smash{{\SetFigFont{10}{12.0}{\rmdefault}{\mddefault}{\updefault}{\color[rgb]{0,0,0}$L_{i_1'}$}%
}}}}
\end{picture}%
\caption{\label{fig:half-strips}}
\end{centering}
\end{figure}%

Next, consider surfaces of the form 
\begin{equation} \label{eq:h-punctured}
S = H \setminus \{\zeta_*,\zeta_1,\dots,\zeta_d\}
\end{equation}
where $H \subset \bC$ is the closed upper half-plane, from which we remove a fixed interior point $\zeta_*$ (say $\zeta_* = i$) as well as boundary points $\zeta_1 < \cdots < \zeta_d$. We consider the $\zeta_k$ as inputs, and $\zeta_*$ as an output. The one-form $\beta_S$ should vanish near the $\zeta_k$, and its integral along a small (counterclockwise) loop around $\zeta_*$ should be equal to $1$ (in other words, $\beta_{\zeta_k} = 0$ and $\beta_{\zeta_*} = 1$). The representation \eqref{eq:monodromy} maps that same loop to $\phi$ (Figure \ref{fig:disc1} shows two equivalent pictures of this; in the right-hand one, we have drawn $H$ itself as a disc with one boundary point removed). Over $\partial H \subset H \setminus \{\zeta_*\}$, we can choose a section of the universal cover. Along that section we place the boundary conditions $L_{i_0},\dots,L_{i_d}$ with $i_0 < \cdots < i_d$, and then extend that to all of $\partial \tilde{S}$ in the unique way which satisfies \eqref{eq:equivariant-lagrangian}. The outcome of this construction are maps
\begin{equation} \label{eq:hochschild-maps}
\begin{CD}
\mathit{CF}^*(L_{i_{d-1}},L_{i_d},H_{i_{d-1}i_d}) \otimes \cdots \otimes \mathit{CF}^*(L_{i_0},L_{i_1},H_{i_0i_1}) \otimes \mathit{CF}^*(\phi(L_{i_d}),L_{i_0},H_{\phi,i_di_0}) \\
@VVV \\
\mathit{CF}^{*+n-d}(\phi,H_{\phi}),
\end{CD}
\end{equation}
where $H_\phi$ is a new Hamiltonian, chosen (together with a corresponding $J_\phi$) to form $\mathit{HF}^*(\phi,\delta,\epsilon)$. The \eqref{eq:hochschild-maps} are the components of a chain map, which induces \eqref{eq:open-closed-string-map}.
\begin{figure}
\begin{centering}
\begin{picture}(0,0)%
\includegraphics{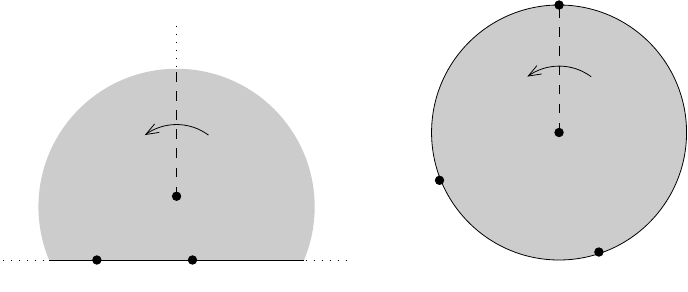}%
\end{picture}%
\setlength{\unitlength}{3355sp}%
\begingroup\makeatletter\ifx\SetFigFont\undefined%
\gdef\SetFigFont#1#2#3#4#5{%
  \reset@font\fontsize{#1}{#2pt}%
  \fontfamily{#3}\fontseries{#4}\fontshape{#5}%
  \selectfont}%
\fi\endgroup%
\begin{picture}(6470,2745)(-3311,-3059)
\put(901,-2611){\makebox(0,0)[lb]{\smash{{\SetFigFont{10}{12.0}{\rmdefault}{\mddefault}{\updefault}{\color[rgb]{0,0,0}$L_{i_1}$}%
}}}}
\put(676,-661){\makebox(0,0)[lb]{\smash{{\SetFigFont{10}{12.0}{\rmdefault}{\mddefault}{\updefault}{\color[rgb]{0,0,0}$L_{i_0}$}%
}}}}
\put(3001,-2386){\makebox(0,0)[lb]{\smash{{\SetFigFont{10}{12.0}{\rmdefault}{\mddefault}{\updefault}{\color[rgb]{0,0,0}$L_{i_2}$}%
}}}}
\put(2176,-1261){\makebox(0,0)[lb]{\smash{{\SetFigFont{10}{12.0}{\rmdefault}{\mddefault}{\updefault}{\color[rgb]{0,0,0}$\phi$}%
}}}}
\put(-2924,-2986){\makebox(0,0)[lb]{\smash{{\SetFigFont{10}{12.0}{\rmdefault}{\mddefault}{\updefault}{\color[rgb]{0,0,0}$L_{i_0}$}%
}}}}
\put(-2249,-2986){\makebox(0,0)[lb]{\smash{{\SetFigFont{10}{12.0}{\rmdefault}{\mddefault}{\updefault}{\color[rgb]{0,0,0}$L_{i_1}$}%
}}}}
\put(-1274,-2986){\makebox(0,0)[lb]{\smash{{\SetFigFont{10}{12.0}{\rmdefault}{\mddefault}{\updefault}{\color[rgb]{0,0,0}$L_{i_2}$}%
}}}}
\put(-1574,-1411){\makebox(0,0)[lb]{\smash{{\SetFigFont{10}{12.0}{\rmdefault}{\mddefault}{\updefault}{\color[rgb]{0,0,0}$\phi$}%
}}}}
\end{picture}%
\caption{\label{fig:disc1}}
\end{centering}
\end{figure}

The construction of \eqref{eq:double-open-closed} is a mixture of the previous two. One considers surfaces
\begin{equation} \label{eq:s-strip-2}
S = T \setminus \{\zeta_*,\zeta_1,\dots,\zeta_r,\zeta_1',\dots,\zeta_s'\},
\end{equation}
where $\zeta_*$ is a fixed interior point of $T$, say $\zeta_* = (0,1/2)$, and the other points are as in \eqref{eq:s-strip}. Note however that this time, both ends $s \rightarrow \pm \infty$ will be considered as inputs. The map \eqref{eq:monodromy} has monodromy $\phi = \phi_1\phi_0$ around $\zeta_*$. This structure, as well as the choice of boundary conditions, is indicated in Figure \ref{fig:disc2}. One chooses one-forms $\beta_S$ such that 
\begin{equation}
\beta_S = \begin{cases} -(\epsilon_0/\epsilon)\, \mathit{dt} & s \ll 0, \\
(\epsilon_1/\epsilon)\, \mathit{dt} & s \gg 0,
\end{cases}
\end{equation}
and which vanish near the $\zeta_k, \zeta_k'$. As a consequence, the integral of $\beta_S$ along a loop around $\zeta_*$ is necessarily $\epsilon_0/\epsilon+\epsilon_1/\epsilon = 1$. The outcome are maps
\begin{equation} \label{eq:double-open-closed-components}
\begin{CD}
\mathit{CF}^*(L_{i_{s-1}'},L_{i_s'},H_{i_{s-1}'i_s'}) \otimes \cdots \otimes \mathit{CF}^*(L_{i_0'},L_{i_1'},H_{i_0'i_1'}) \\ 
\otimes \, \mathit{CF}^*(\phi_1(L_{i_r}),L_{i_0'},H_{\phi_0,i_r i_0'})\, \otimes \\ 
\mathit{CF}^*(L_{i_{r-1}},L_{i_r},H_{i_{r-1}i_r}) \otimes \cdots \otimes
\mathit{CF}^*(L_{i_0},L_{i_1},H_{i_0i_1}) \\
\; \otimes \, \mathit{CF}^*(\phi_0(L_{i_s'}),L_{i_0},H_{\phi_1,i_{s'} i_0})
\\
@VVV
\\
\mathit{CF}^{*+n-r-s}(\phi,H_{\phi}).
\end{CD}
\end{equation}
These are the components of a chain map, which induces \eqref{eq:double-open-closed}.
\begin{figure}
\begin{centering}
\begin{picture}(0,0)%
\includegraphics{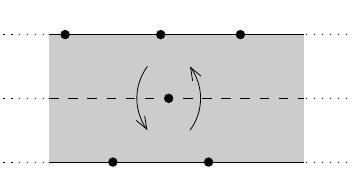}%
\end{picture}%
\setlength{\unitlength}{3355sp}%
\begingroup\makeatletter\ifx\SetFigFont\undefined%
\gdef\SetFigFont#1#2#3#4#5{%
  \reset@font\fontsize{#1}{#2pt}%
  \fontfamily{#3}\fontseries{#4}\fontshape{#5}%
  \selectfont}%
\fi\endgroup%
\begin{picture}(3324,1819)(3739,-2459)
\put(4201,-2386){\makebox(0,0)[lb]{\smash{{\SetFigFont{10}{12.0}{\rmdefault}{\mddefault}{\updefault}{\color[rgb]{0,0,0}$L_{i_0}$}%
}}}}
\put(6151,-811){\makebox(0,0)[lb]{\smash{{\SetFigFont{10}{12.0}{\rmdefault}{\mddefault}{\updefault}{\color[rgb]{0,0,0}$L_{i_0'}$}%
}}}}
\put(5101,-2386){\makebox(0,0)[lb]{\smash{{\SetFigFont{10}{12.0}{\rmdefault}{\mddefault}{\updefault}{\color[rgb]{0,0,0}$L_{i_1}$}%
}}}}
\put(6001,-2386){\makebox(0,0)[lb]{\smash{{\SetFigFont{10}{12.0}{\rmdefault}{\mddefault}{\updefault}{\color[rgb]{0,0,0}$L_{i_2}$}%
}}}}
\put(3901,-811){\makebox(0,0)[lb]{\smash{{\SetFigFont{10}{12.0}{\rmdefault}{\mddefault}{\updefault}{\color[rgb]{0,0,0}$L_{i_3'}$}%
}}}}
\put(4576,-811){\makebox(0,0)[lb]{\smash{{\SetFigFont{10}{12.0}{\rmdefault}{\mddefault}{\updefault}{\color[rgb]{0,0,0}$L_{i_2'}$}%
}}}}
\put(5476,-811){\makebox(0,0)[lb]{\smash{{\SetFigFont{10}{12.0}{\rmdefault}{\mddefault}{\updefault}{\color[rgb]{0,0,0}$L_{i_1'}$}%
}}}}
\put(5626,-1336){\makebox(0,0)[lb]{\smash{{\SetFigFont{10}{12.0}{\rmdefault}{\mddefault}{\updefault}{\color[rgb]{0,0,0}$\phi_1$}%
}}}}
\put(4651,-1861){\makebox(0,0)[lb]{\smash{{\SetFigFont{10}{12.0}{\rmdefault}{\mddefault}{\updefault}{\color[rgb]{0,0,0}$\phi_0$}%
}}}}
\end{picture}%
\caption{\label{fig:disc2}}
\end{centering}
\end{figure}

\subsection{The Lefschetz condition}
What's been missing so far is a natural source of Lagrangian submanifolds, which would make the construction of $\scrA$ meaningful. We will provide that now, in the form of Lefschetz thimbles.

\begin{setup} \label{th:setup-lefschetz}
(i) Take an exact symplectic fibration with singularities. We call it an {\em exact symplectic Lefschetz fibration} if it satisfies the following additional conditions. First of all, $TE^v_x = \mathit{ker}(D\pi_x) \subset TE_x$ is a symplectic subspace at every regular point $x$. Secondly, near each critical point, there is an (integrable and $\omega_E$-compatible) complex structure $I_E$, such that $\pi$ is $I_E$-holomorphic, and the (complex) Hessian at the critical point is nondegenerate. Together with the previous conditions, this implies that there are only finitely many critical points. For practical bookkeeping purposes, we also impose the additional condition that there should be at most one critical point in each fibre.

(ii) Suppose that we have an exact symplectic Lefschetz fibration. A {\em basis of vanishing paths} (see Figure \ref{fig:basis}) is a collection of properly embedded half-infinite paths $l_1,\dots,l_m \subset \bC$,  whose endpoints are precisely the critical values of $\pi$, and with the following properties. There is a half-plane $\{\mathrm{re}(y) \leq C\}$ which contains all critical values of $\pi$, and such that the parts of the $l_k$ lying in that half-plane are pairwise disjoint. Outside that half-plane, we have
\begin{equation}
l_k \cap \{\mathrm{re}(y) \geq C\} = \{\mathrm{im}(y) = \iota_k(\mathrm{re}(y)) \}.
\end{equation}
Here, the functions $\iota_1,\dots,\iota_m: [C,\infty] \rightarrow \bR$ are constant near infinity, let's say $\iota_k(r) = o_k$ for $r \gg 0$, and
\begin{equation}
\left\{
\begin{aligned}
& \iota_1(C) < \cdots < \iota_m(C), \\
& o_1 > \cdots > o_m.
\end{aligned}
\right.
\end{equation}
The corresponding {\em basis of Lefschetz thimbles} consists of the unique Lagrangian submanifolds $L_1,\dots,L_m \subset E$ such that $\pi(L_k) = l_k$. We choose orientations arbitrarily (since the Lefschetz thimbles are diffeomorphic to $\bR^n$, they carry unique {\em Spin} structures; and if $E$ has a symplectic Calabi-Yau structure, they can be equipped with gradings).
\end{setup}
\begin{figure} 
\begin{centering}
\begin{picture}(0,0)%
\includegraphics{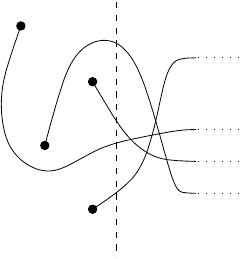}%
\end{picture}%
\setlength{\unitlength}{3355sp}%
\begingroup\makeatletter\ifx\SetFigFont\undefined%
\gdef\SetFigFont#1#2#3#4#5{%
  \reset@font\fontsize{#1}{#2pt}%
  \fontfamily{#3}\fontseries{#4}\fontshape{#5}%
  \selectfont}%
\fi\endgroup%
\begin{picture}(2311,2424)(1380,-2998)
\put(3676,-2461){\makebox(0,0)[lb]{\smash{{\SetFigFont{10}{12.0}{\rmdefault}{\mddefault}{\updefault}{\color[rgb]{0,0,0}$l_4$}%
}}}}
\put(3676,-1186){\makebox(0,0)[lb]{\smash{{\SetFigFont{10}{12.0}{\rmdefault}{\mddefault}{\updefault}{\color[rgb]{0,0,0}$l_1$}%
}}}}
\put(3676,-1861){\makebox(0,0)[lb]{\smash{{\SetFigFont{10}{12.0}{\rmdefault}{\mddefault}{\updefault}{\color[rgb]{0,0,0}$l_2$}%
}}}}
\put(3676,-2161){\makebox(0,0)[lb]{\smash{{\SetFigFont{10}{12.0}{\rmdefault}{\mddefault}{\updefault}{\color[rgb]{0,0,0}$l_3$}%
}}}}
\end{picture}%
\caption{\label{fig:basis}}
\end{centering}
\end{figure}

\begin{conjecture} \label{th:decomposition-of-diagonal}
Let's use a basis of vanishing cycles to define the $A_\infty$-algebra $\scrA$ and bimodules $\scrP_{\phi,\delta,\epsilon}$. Then, for any automorphism $\phi$ and $\delta \gg 0$, the open-closed string map \eqref{eq:open-closed-string-map} is an isomorphism.
\end{conjecture}

This is a kind of ``decomposition of the diagonal'' statement, which is of general interest since it would give a way of computing fixed point Floer cohomology in terms of open string (Fukaya category) data. It is plausible that it could be approached by the methods from \cite{abouzaid-ganatra14}, but we will not discuss that possibility further here.

\begin{lemma} \label{th:li-floer}
Let $\nu$ be the global monodromy. For any $\delta < 0$,
$\mathit{HF}^*(\nu(L_i),L_i,\delta,\epsilon)$ is one-dimensional and concentrated in degree $0$. Moreover, for any $i<j$ and any $\delta < o_j - o_i < 0$, the triangle product
\begin{equation} \label{eq:isomorphic-triangle-product}
\mathit{HF}^*(L_i,L_j) \otimes \mathit{HF}^*(\nu(L_i),L_i,\delta,\epsilon) \longrightarrow \mathit{HF}^*(\nu(L_i),L_j,\delta,\epsilon)
\end{equation}
is an isomorphism.
\end{lemma}

\begin{proof}[Sketch of proof]
Let $H$ be the Hamiltonian used to define $\mathit{HF}^*(\nu(L_i),L_i,\delta,\epsilon)$. The generators of the underlying chain complex correspond to points of $(\phi_H^1 \circ \nu)(L_i) \cap L_i$. After a compactly supported isotopy (see Figure \ref{fig:triangle}), these two Lagrangian submanifolds will intersect in a single point, which is a critical point of $\pi$; and one easily computes that its Maslov index is $0$. For a suitable choice of auxiliary data, the product \eqref{eq:isomorphic-triangle-product} is obtained by counting holomorphic triangles in $E$ which project to the triangle in the base shaded in Figure \ref{fig:triangle}. One can in principle determine those directly, but it is easier to apply a further isotopy as indicated in Figure \ref{fig:triangle2}, after which the triangle in the base can be shrunk to a point, making the computation straightforward (the same trick is used in \cite[Figure 6]{maydanskiy-seidel09}).
\end{proof}
\begin{figure} 
\begin{centering}
\begin{picture}(0,0)%
\includegraphics{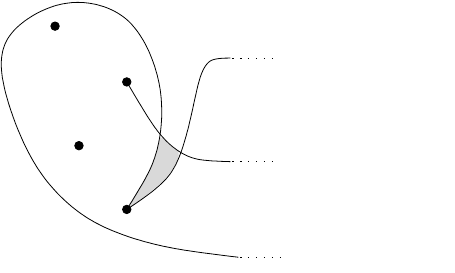}%
\end{picture}%
\setlength{\unitlength}{3355sp}%
\begingroup\makeatletter\ifx\SetFigFont\undefined%
\gdef\SetFigFont#1#2#3#4#5{%
  \reset@font\fontsize{#1}{#2pt}%
  \fontfamily{#3}\fontseries{#4}\fontshape{#5}%
  \selectfont}%
\fi\endgroup%
\begin{picture}(4261,2559)(1058,-3134)
\put(3751,-3061){\makebox(0,0)[lb]{\smash{{\SetFigFont{10}{12.0}{\rmdefault}{\mddefault}{\updefault}{\color[rgb]{0,0,0}$\nu(L_i)$}%
}}}}
\put(3226,-2311){\makebox(0,0)[lb]{\smash{{\SetFigFont{10}{12.0}{\rmdefault}{\mddefault}{\updefault}{\color[rgb]{0,0,0}$L_j$}%
}}}}
\put(3226,-961){\makebox(0,0)[lb]{\smash{{\SetFigFont{10}{12.0}{\rmdefault}{\mddefault}{\updefault}{\color[rgb]{0,0,0}$L_i$}%
}}}}
\put(3751,-2836){\makebox(0,0)[lb]{\smash{{\SetFigFont{10}{12.0}{\rmdefault}{\mddefault}{\updefault}{\color[rgb]{0,0,0}perturbed version of}%
}}}}
\end{picture}%
\caption{\label{fig:triangle}}
\end{centering}
\end{figure}
\begin{figure}
\begin{centering}
\begin{picture}(0,0)%
\includegraphics{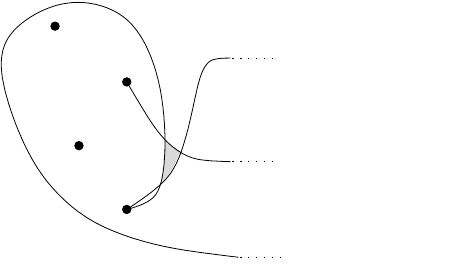}%
\end{picture}%
\setlength{\unitlength}{3355sp}%
\begingroup\makeatletter\ifx\SetFigFont\undefined%
\gdef\SetFigFont#1#2#3#4#5{%
  \reset@font\fontsize{#1}{#2pt}%
  \fontfamily{#3}\fontseries{#4}\fontshape{#5}%
  \selectfont}%
\fi\endgroup%
\begin{picture}(4261,2559)(1058,-3134)
\put(3226,-2311){\makebox(0,0)[lb]{\smash{{\SetFigFont{10}{12.0}{\rmdefault}{\mddefault}{\updefault}{\color[rgb]{0,0,0}$L_j$}%
}}}}
\put(3226,-961){\makebox(0,0)[lb]{\smash{{\SetFigFont{10}{12.0}{\rmdefault}{\mddefault}{\updefault}{\color[rgb]{0,0,0}$L_i$}%
}}}}
\put(3751,-2836){\makebox(0,0)[lb]{\smash{{\SetFigFont{10}{12.0}{\rmdefault}{\mddefault}{\updefault}{\color[rgb]{0,0,0}perturbed version of}%
}}}}
\put(3751,-3061){\makebox(0,0)[lb]{\smash{{\SetFigFont{10}{12.0}{\rmdefault}{\mddefault}{\updefault}{\color[rgb]{0,0,0}$\nu(L_i)$}%
}}}}
\end{picture}%
\caption{\label{fig:triangle2}}
\end{centering}
\end{figure}


Take $\gamma \in (0,2\pi)$, $\delta > 0$, and a small $\epsilon>0$. Consider the map
\begin{equation} \label{eq:combined-oc}
H^*(E) \iso \mathit{HF}^*(E,\gamma,\epsilon) \longrightarrow \mathit{HF}^*(\nu,\delta,\epsilon) \longrightarrow \mathit{HF}^*(\nu(L_i),L_i,\delta,\epsilon).
\end{equation}
Here, the first isomorphism is Lemma \ref{th:bv-vanish-3}; the map after that comes from Lemma \ref{th:rotation-translation}; and the final one is \eqref{eq:dual-oc0}. Let $u_i$ be the image of $1 \in H^0(E)$ under \eqref{eq:combined-oc}.

\begin{lemma} \label{th:last-lemma}
$u_i$ is nontrivial.
\end{lemma}

\begin{proof}[Sketch of proof]
One can define a version of Lagrangian Floer homology perturbed by a rotational Hamiltonian (in parallel with Section \ref{subsec:rotate}), which we will denote by $\mathit{HF}^*(L_i,L_i,\gamma,\epsilon)$. Then, the map \eqref{eq:rotation-translation} has a Lagrangian counterpart, which fits into a commutative diagram of the form including all the maps in \eqref{eq:combined-oc}:
\begin{equation}
\xymatrix{
H^*(E) \ar[d] \ar[r] & \mathit{HF}^*(E,\gamma,\epsilon) \ar[r] \ar[d] & \mathit{HF}^*(\nu,\delta,\epsilon) \ar[d] \\
H^*(L) \ar[r] & \mathit{HF}^*(L_i,L_i,\gamma,\epsilon) \ar[r] & \mathit{HF}^*(\nu(L_i),L_i,\delta,\epsilon).
}
\end{equation}
The leftmost $\downarrow$ is the standard restriction map; and all the $\rightarrow$ in the bottom row are isomorphisms (of one-dimensional vector spaces, concentrated in degree $0$).
\end{proof}

\begin{lemma}
Suppose that $\epsilon>0$ is small, and that $\delta < o_m - o_1$. Take the element of $\mathit{HF}^0(\nu,\delta,\epsilon)$ produced from $1 \in H^0(E)$ as in \eqref{eq:combined-oc}. Applying the open-closed string map in its dual form \eqref{eq:dual-open-closed-string-map} to that element yields a quasi-isomorphism $\scrA \rightarrow \scrP_{\nu,\delta,\epsilon}$.
\end{lemma}

\begin{proof}
This is a direct application of Lemma \ref{th:recognize-diagonal} (with $\scrP = \scrP_{\nu^{-1},-\delta,-\epsilon}[n] \iso \scrP_{\nu,\delta,\epsilon}^\vee$). To use that criterion, one has to consider the maps
\begin{equation}
e_j H^*(\scrA) e_i \longrightarrow e_j H^*(\scrP_{\nu,\delta,\epsilon}) e_i = \mathit{HF}^*(\nu(L_i),L_j,\delta,\epsilon)
\end{equation}
given by
\begin{equation}
\left\{
\begin{aligned}
& \text{the product with $u_i$, if $i<j$;} \\
& \text{taking the unit $e_i$ to $u_i$, if $i = j$;} \\
& \text{zero, if $i>j$.}
\end{aligned}
\right.
\end{equation}
The first two cases yield an isomorphism, by Lemmas \ref{th:li-floer} and \ref{th:last-lemma}. As for the last case, it is straightforward to show that $\mathit{HF}^*(\nu(L_i),L_j,\delta,\epsilon) = 0$ for $i>j$.
\end{proof}

To summarize, we now know that 
\begin{equation} \label{eq:p-dual-p}
\text{for $\delta \ll 0$,}\quad \left\{
\begin{aligned}
& \scrP_{\nu,\delta,\epsilon} \htp \scrA, \\ 
& \scrP_{\nu^{-1},-\delta,-\epsilon} \htp \scrA^\vee[-n],
\end{aligned}
\right.
\end{equation}
where the second part of the statement follows from the first one by \eqref{eq:dual-bimodule}. Note that $\epsilon$ can be arbitrary, since the bimodules are independent of $\epsilon$ up to quasi-isomorphism. Applying \eqref{eq:dual-double}, one therefore gets a map
\begin{equation} \label{eq:omg}
\begin{aligned}
\mathit{HF}^*(\nu^2,\delta,\epsilon) \longrightarrow & H^*\big(\mathit{hom}_{[\scrA,\scrA]}(\scrP_{\nu^{-1},-\delta/2,-\epsilon/2}, \scrP_{\nu,\delta/2,\epsilon/2})\big) \\ & \quad \iso H^*(\mathit{hom}_{[\scrA,\scrA]}(\scrA^\vee[-n],\scrA)).
\end{aligned}
\end{equation}

\begin{proof}[Proof of Theorem \ref{th:main}]
For small $\epsilon>0$, Proposition \ref{th:partial-splitting} and Lemma \ref{th:rotation-translation} yield a map
\begin{equation}
\mathit{HF}^{*+2}(\mu,\epsilon) \longrightarrow \mathit{HF}^*(\nu^2,\delta,\epsilon).
\end{equation}
One combines this with \eqref{eq:omg} to get the desired construction.
\end{proof}

\begin{proof}[Proof of Theorem \ref{th:fano}]
This is the same argument as before, but instead of Proposition \ref{th:partial-splitting}, one uses Proposition \ref{th:partial-splitting-2} (and $1<\epsilon<2$).
\end{proof}

\begin{remark}
By comparing \eqref{eq:g-les} and Conjecture \ref{th:conjecture}, one sees that for the conjecture to hold, the bimodule map $\scrA^\vee \rightarrow \scrA$ constructed in the proof of Theorem \ref{th:fano} must necessarily be invariant under self-conjugation \eqref{eq:c-automorphism}. Equivalently by Lemma \ref{th:z-action}, if one thinks of that map as an element of $\mathit{HH}_*(\scrA,\scrA^\vee \otimes_{\scrA} \scrA^\vee)^\vee$, it must be invariant under the $\bZ/2$-action which exchanges the two tensor factors. Going back to the geometric definition, which means using \eqref{eq:p-dual-p}, the desired statement is that the image of a specific element under the map
\begin{equation} \label{eq:zz}
\mathit{HF}^0(\nu^2,\delta,\epsilon) \longrightarrow
\mathit{HH}_0(\scrA, \scrP_{\nu^{-1},-\delta/2,-\epsilon/2} \otimes_{\scrA} \scrP_{\nu^{-1},-\delta/2,-\epsilon/2})^\vee
\end{equation}
is $\bZ/2$-invariant. Assuming $\mathrm{char}(\bK) \neq 2$, it follows from Remark \ref{th:swap-factors} and Example \ref{th:trivial-involution} that the entire image of \eqref{eq:zz} is $\bZ/2$-invariant.
\end{remark}

\begin{remark} 
As a generalization of Conjecture \ref{th:decomposition-of-diagonal}, one could consider maps \eqref{eq:double-open-closed}, but with an arbitrary number of bimodules involved. We will be interested only in one special case, written in dual form as in \eqref{eq:dual-double}:
\begin{equation} \label{eq:i-power}
\begin{aligned}
\mathit{HF}^*(\nu^k,\delta,\epsilon) \longrightarrow & H^*\big(\mathit{hom}_{[\scrA,\scrA]}(
\scrP_{\nu^{-1},-\delta/k,-\epsilon/k}^{\otimes_{\scrA} k-1}, \scrP_{\nu,\delta/k,\epsilon/k})\big) 
\\ & \quad \iso H^*(\mathit{hom}_{[\scrA,\scrA]}((\scrA^\vee)^{\otimes_{\scrA} k-1}, \scrA)),
\quad \delta \ll 0.
\end{aligned}
\end{equation}
Suppose that \eqref{eq:i-power} is an isomorphism. Let's specialize to Lefschetz fibrations coming from anticanonical Lefschetz pencils. In that case, we have seen in Example \ref{th:negative-degrees} that, for geometric reasons, the Floer cohomology groups of $\nu^k$ are concentrated in nonnegative degrees. It would then follow that \eqref{eq:no-negative-degree} is satisfied, so that Lemmas \ref{th:obstructions} and \ref{th:obstructions-2} would become applicable, leading to a proof of Conjecture \ref{th:conjecture} (however, this approach would not apply to other situations, such as that of Remark \ref{th:fractional-cy-2}).
\end{remark}


\end{document}